\documentclass{amsart}
\usepackage{amssymb,latexsym,mathtools}
\usepackage{mathrsfs}
\usepackage{mathtools}
\usepackage{graphicx}
\usepackage{esint}
\usepackage{braket}
\usepackage{url, hyperref}
\newtheorem{theorem}{Theorem}
\newtheorem{lemma}{Lemma}
\newtheorem{definition}{Definition}

\newtheorem{prop}{Proposition}

\newtheorem{remark}{Remark}

\begin{document}

\title[Convergence for HJ Equations on Junctions]{Convergence \& Rates for \\
Hamilton-Jacobi Equations with Kirchoff Junction Conditions}
\author[P.~Morfe]{Peter S. Morfe}
\address{Department of Mathematics \\ 
University of Chicago \\ 
Chicago, IL 60637 \\ USA}
\email{\href{mailto:pmorfe@math.uchicago.edu}{pmorfe@math.uchicago.edu}}

\date{\today}



\begin{abstract} 
We investigate rates of convergence for two approximation schemes of time-independent and time-dependent Hamilton-Jacobi equ-ations with Kirchoff junction conditions.  We analyze the vanishing viscosity limit and monotone finite-difference schemes.  Following recent work of Lions and Souganidis, we impose no convexity assumptions on the Hamiltonians.  For stationary Hamilton-Jacobi equations, we obtain the classical $\epsilon^{\frac{1}{2}}$ rate, while we obtain an $\epsilon^{\frac{1}{7}}$ rate for approximations of the Cauchy problem.  In addition, we present a number of new techniques of independent interest, including a quantified comparison proof for the Cauchy problem and an equivalent definition of the Kirchoff junction condition.      
\end{abstract}

\maketitle

\section{Introduction}

The goal of the present paper is to study rates of convergence for approximations of Hamilton-Jacobi equations on junctions with Kirchoff conditions.  

Given $K$ copies $\{I_{1},\dots,I_{K}\}$ of the interval $(-\infty,0)$, we define the junction as the disjoint union $\mathcal{I} := \bigcup_{i = 1}^{K} \overline{I_{i}}$ glued at zero.  The equations of interest to us are the stationary equation
\begin{equation} \label{E: stationary}
	\left\{
		\begin{array}{r l}
			u + H_{i}(x,u_{x_{i}}) =0 & \text{in} \, \, I_{i} \\
			\sum_{i = 1}^{K} u_{x_{i}} = B & \text{on} \, \, \{0\}
		\end{array}
	\right.
\end{equation}
and the Cauchy problem
\begin{equation} \label{E: time}
	\left\{ \begin{array}{r l}
		u_{t} + H_{i}(t,x,u_{x_{i}}) = 0 		& 		\quad \text{in} \, \, I_{i} \times (0,T)	\\
		\sum_{i = 1}^{K} u_{x_{i}} = B 		& 		\quad \text{on} \, \, \{0\} \times (0,T) \\
		u= u_{0} 				& 		\quad \text{on} \, \, \mathcal{I} \times \{0\}
	\end{array} \right.
\end{equation}
where the equations are understood in the viscosity sense (cf.\ \cite{time-independent,time-dependent}) and $B \in \mathbb{R}$ is a constant.  

We study two approximations of \eqref{E: stationary} and \eqref{E: time}, namely the vanishing viscosity limit and the finite-difference approximation, and prove that they converge with algebraic rates.

The vanishing viscosity approximations of \eqref{E: stationary} and \eqref{E: time} are given by, respectively,
\begin{equation} \label{E: viscous_stationary}
	\left\{
		\begin{array}{r l}
			u^{\epsilon} - \epsilon u^{\epsilon}_{x_{i} x_{i}} + H_{i}(x,u^{\epsilon}_{x_{i}}) =0 & \text{in} \, \, I_{i} \\
			\sum_{i = 1}^{K} u^{\epsilon}_{x_{i}} = B & \text{on} \, \, \{0\}
		\end{array}
	\right.
\end{equation}
and
\begin{equation} \label{E: viscous_time}
	\left\{ \begin{array}{r l}
		u^{\epsilon}_{t} - \epsilon u^{\epsilon}_{x_{i} x_{i}} + H_{i}(t,x,u^{\epsilon}_{x_{i}}) = 0 		& 		\quad \text{in} \, \, I_{i} \times (0,T)	\\
		\sum_{i = 1}^{K} u^{\epsilon}_{x_{i}} = B 		& 		\quad \text{on} \, \, \{0\} \times (0,T) \\
		u^{\epsilon} = u_{0} 				& 		\quad \text{on} \, \, \mathcal{I} \times \{0\}
	\end{array} \right.
\end{equation}
In the Euclidean setting, it is well known that, as $\epsilon \to 0^{+}$, $u^{\epsilon} \to u$ uniformly with an error that is on the order of $\epsilon^{\frac{1}{2}}$ (cf.\ \cite{old_paper}).  

For the Kirchoff junction problems \eqref{E: viscous_stationary} and \eqref{E: viscous_time}, we establish the following two results:

\begin{theorem} \label{T: rate_stationary_visc}  Assume the Hamiltonians $H_{1},\dots,H_{K}$ satisfy \eqref{As: continuity}, \eqref{As: coercive}, and \eqref{As: sub-sup}.  For each $\epsilon > 0$, let $u^{\epsilon}$ denote the unique bounded solution of \eqref{E: viscous_stationary}, and let $L = \sup \left\{\text{Lip}(u^{\epsilon}) \, \mid \, \epsilon > 0\right\}$.  There is a constant $C > 0$ depending only on $L$, the constant $M$ from \eqref{As: sub-sup}, and the Lipschitz constant of each Hamiltonian $H_{i}$ in $\overline{I_{i}} \times [-L,L]$ such that if $u$ is the unique bounded viscosity solution of \eqref{E: stationary}, then
\begin{equation*}
\sup \left\{ |u^{\epsilon}(x) - u(x)|\, \mid \, x \in \mathcal{I}\right\} \leq C \epsilon^{\frac{1}{2}}.
\end{equation*}
\end{theorem}    

\begin{theorem} \label{T: rate_cauchy_visc}  Assume \eqref{As: continuity}, \eqref{As: coercive}, \eqref{As: time_bound}, \eqref{As: sub-sup}, and \eqref{As: data_lip}.  For each $\epsilon > 0$, let $u^{\epsilon}$ denote the unique uniformly continuous solution of \eqref{E: viscous_time}.  For each $K \geq 1$, there is a constant $L_{K,T}$ depending only on $\text{Lip}(u_{0})$, $T$, and $K$ and a constant $C_{K} > 0$ depending only on $K$, $L_{K,T}$, $T$, and the Lipschitz constant of each Hamiltonian $H_{i}$ in $[0,T] \times \overline{I_{i}} \times [-(L_{K,T} + 1),L_{K,T} + 1]$ such that if $\epsilon \in (0,K]$ and $u$ denotes the unique uniformly continuous viscosity solution of \eqref{E: time}, then 
\begin{equation*}
\sup \left\{ |u^{\epsilon}(x,t) - u(x,t)| \, \mid \, (x,t) \in \mathcal{I} \times [0,T]\right\} \leq C_{K} \epsilon^{\frac{1}{7}}.
\end{equation*}
\end{theorem}

We also consider finite-difference schemes that approximate equations \eqref{E: stationary} and \eqref{E: time}.  We discretize the junction by gluing together $K$ discretized edges $J_{1},\dots,J_{K}$ at spatial scale $\Delta x$, and we discretize time similarly.  The finite-difference schemes considered in this paper take the form:
\begin{equation} \label{E: stationary_diff}
\left\{ \begin{array}{r l}
		U + F_{i}(D^{+}U, D^{-}U) = 0 &  \text{in} \, \, J_{i} \setminus \{0\} \\
		U(0) = \frac{1}{K} \sum_{i = 1}^{K} U(1_{i}) - B
\end{array} \right.
\end{equation}
and
\begin{equation} \label{E: cauchy_diff}
\left\{ \begin{array}{r l}
		D_{t}U + F_{i}(D^{+}U,D^{-}U) = 0 &  \text{in} \, \, (J_{i} \setminus \{0\}) \times \{1,\dots,N\} \\
		U(0,\cdot) = \frac{1}{K} \sum_{i = 1}^{K} U(1_{i},\cdot) - B &  \text{on} \, \,\{1,2,\dots,N\} \\
		U(\cdot,0) = U_{0} & \text{on} \, \, J_{i} \times \{0\}
\end{array} \right.
\end{equation}
Here $\{F_{1},\dots,F_{K}\}$ are operators that approximate the Hamiltonians, the points $1_{i}$ are the nearest neighbors of $0$ in each edge $J_{i}$, and $U_{0}$ is the restriction of the initial condition $u_{0}$ to the numerical grid.  The operators $\{F_{1},\dots,F_{K}\}$ are defined below in equation \eqref{E: F_op}, precise definitions of the difference quotient operators $D_{t},D^{+},D^{-}$ can be found in equations \eqref{E: diff} and \eqref{E: diff_time}, and the numerical grid is defined in Subsections \ref{S: stationary_diff_explanation} and \ref{S: cauchy_diff_explanation}.  

Error analysis of finite-difference schemes for Hamilton-Jacobi equations goes back to \cite{old_paper}, where an estimate on the order of $\Delta x^{\frac{1}{2}}$ was obtained.  As in the vanishing viscosity case, we derive the following two results:

\begin{theorem} \label{T: rate_stationary_diff} Assume \eqref{As: continuity}, \eqref{As: coercive}, \eqref{As: sub-sup}, \eqref{As: standard}, \eqref{As: num_Ham_lip}, \eqref{As: num_Ham_consistent}, \eqref{As: CFL_stationary}, and \eqref{As: cut_off_bound}.  There is a constant $C > 0$ depending only on the constant $M$ from \eqref{As: sub-sup}, the constants $L_{G}$ from \eqref{As: LG_constant} and $L_{2}$ from \eqref{As: CFL_stationary}, the constant $\overline{L}_{c}$ defined in Theorem \ref{T: stationary_scheme}, and the Lipschitz constant of each Hamiltonian $H_{i}$ in $\overline{I_{i}} \times [-(\overline{L}_{c} + 1),\overline{L}_{c} + 1]$ such that if $u$ is the unique bounded viscosity solution of \eqref{E: stationary} and $U$ is the unique bounded solution of \eqref{E: stationary_diff}, then
\begin{equation*}
\sup \left\{ |U(m) - u(-m\Delta x)| \, \mid \, m \in \mathcal{J} \right\} \leq C \Delta x^{\frac{1}{2}}.
\end{equation*}
\end{theorem}  

\begin{theorem} \label{T: rate_cauchy_diff}  Assume \eqref{As: continuity}, \eqref{As: coercive}, \eqref{As: time_bound}, \eqref{As: sub-sup}, \eqref{As: data_lip}, \eqref{As: standard}, \eqref{As: num_Ham_lip}, \eqref{As: num_Ham_consistent}, \eqref{As: CFL_cauchy}, and \eqref{As: cut_off_bound_time}.  There is a constant $C > 0$ depending only on the constants $L_{G}$ from \eqref{As: LG_constant} and $L_{2}$ from \eqref{As: CFL_cauchy}, the constant $\tilde{L}_{c}$ defined in Proposition \ref{P: monotone_time},  and the Lipschitz constant of each Hamiltonian $H_{i}$ in $[0,T] \times \overline{I_{i}} \times [-(\tilde{L}_{c} + 1),\tilde{L}_{c} + 1]$ such that if $u$ is the unique uniformly continuous viscosity solution of \eqref{E: stationary} and $U$ is the unique solution of \eqref{E: cauchy_diff}, then
\begin{equation*}
\sup \left\{|U(m,s)- u(-m\Delta x, s \Delta t)| \, \mid \, (m,s) \in \mathcal{J} \times S \right\}| \leq C \Delta x^{\frac{1}{7}}.
\end{equation*}
\end{theorem}  

In addition to the error analysis and auxiliary technical results, we give a largely complete presentation of the well-posedness results from \cite{time-dependent}.  We do this, at the expense of some repetition, for the convenience of the reader and to demonstrate how to quantify the uniqueness proof from \cite{time-dependent}.  We also prove existence of solutions of \eqref{E: viscous_time} and related estimates.

\subsection{Ideas and difficulties}  The error analysis of approximations of \eqref{E: stationary} relies on an insight from \cite{time-independent}.  This part is calculus.  Specifically, if $u$ is a function on $[-1,0]$, $u'(0)$ exists, $\theta, \delta > 0$ are parameters, and $y \in [-1,0]$, then one can prove that the function 
\begin{equation*}
x \mapsto u(x) - \frac{(x - y)^{2}}{2 \theta} - (u'(0) + \delta)(x - y)
\end{equation*} cannot attain its maximum at $0$.

The previous fact enables us to double variables when studying approximation schemes.  Consider the vanishing viscosity limit for simplicity.  The basic difficulty compared to the Euclidean setting is the junction condition.  Upon reflection, even if things were smooth, we realize difficulties arise if some of the maximum points of the test function occur at the junction.  However, we can get around this using the observation of \cite{time-independent}.  In the error analysis, we double variables by studying maxima of the function 
\begin{equation*}
\Phi_{i,\delta}(x,y) = u^{\epsilon}(x) - u(y) - \frac{(x - y)^{2}}{2 \sqrt{\epsilon}} - (u^{\epsilon}_{x_{i}}(0) + \delta) (x - y).
\end{equation*}  
If $(x_{i}(\delta),y_{i}(\delta))$ maximizes $\Phi_{i,\delta}$, then, in particular, $x_{i}(\delta)$ maximizes $x \mapsto \Phi_{i,\delta}(x,y_{i}(\delta))$, and the previous observation shows $x_{i}(\delta) < 0$.  Therefore, we can write down the equation solved by $u^{\epsilon}$ in the interior of $I_{i}$ and hope to use this to get a bound on $u^{\epsilon} - u$.  Further reflection leads us to realize, then, that the only trouble occurs if $y_{i}(\delta) = 0$ independently of $i$.  However, since the flux $\sum_{i = 1}^{K} \left(u^{\epsilon}_{x_{i}}(0) + \delta + \frac{x_{i}(\delta)}{\sqrt{\epsilon}}\right)$ is no larger than $K \delta$, we prove below that there is a continuity property of the Kirchoff condition that allows us to find a $j \in \{1,\dots,K\}$ such that
\begin{equation*}
u(y_{j}(\delta)) + H_{j} \left(y_{j}(\delta),\frac{x_{j}(\delta)}{\sqrt{\epsilon}} + u^{\epsilon}_{x_{j}}(0) + \delta\right) \geq -\omega(K \delta),
\end{equation*}
where $\omega$ is the modulus of continuity of $H_{j}$.  Combining this with the equation solved by $u^{\epsilon}$ at $x_{j}(\delta)$ and sending $\delta \to 0^{+}$, we prove below that, in this case, $u^{\epsilon}(x_{j}(\delta)) - u(y_{j}(\delta))$ is small.  

The remainder of the proof uses the topology of the junction and the Lipschitz continuity of the solutions.  Since $y_{i}(\delta) = 0$ and, at least formally, $x_{i}(\delta)$ must be quite close it, $u^{\epsilon}(x_{i}(\delta)) - u(y_{i}(\delta))$ is certainly close to $u^{\epsilon}(0) - u(0)$ by continuity.  Therefore, since $(x_{i}(\delta),y_{i}(\delta))$ maximizes $\Phi_{i,\delta}$ in $\overline{I_{i}}$ no matter the choice of $i$, and $0$ is an element of each ray $\overline{I_{i}}$, the smallness of $u^{\epsilon}(x_{j}(\delta)) - u(y_{j}(\delta))$ forces $u^{\epsilon} - u$ to be small in $\mathcal{I}$.

Our error analysis of approximations of the Cauchy problem \eqref{E: time} is significantly more complicated.  This reflects the fact that the proof of the comparison principle for \eqref{E: time} is more delicate than that of \eqref{E: stationary}.  The major step in the comparison proof presented in \cite{time-dependent} is the reduction to a stationary equation with a Kirchoff junction condition using a blow-up argument.  In order to obtain error estimates, we show in this paper that it is possible to reduce to a stationary equation without completely blowing-up the solutions at the junction.  Instead of performing a blow-up, we study the difference of the solutions at a small but positive scale near the junction.  To do this, we need to quantify the moduli of continuity of the time derivatives of the solutions.  This step is the major contributor to the error and the reason it differs from the classical rate.

\subsection{Assumptions}  In the statement of the assumptions concerning the Hamiltonians, we will write $H_{i} = H_{i}(t,x,p)$ with the understanding that, in the time-independent equations \eqref{E: stationary} and \eqref{E: viscous_stationary}, $H_{i,t}(t,x,p) \equiv 0$.  We assume that, for each $i$ and $R > 0$,
\begin{equation}
H_{i} : [0,T] \times \overline{I_{i}} \times B(0,R) \to \mathbb{R}  \, \, \text{is uniformly Lipschitz continuous,} \label{As: continuity}\\
\end{equation}
and
\begin{equation}
\lim_{|p| \to \infty} H_{i}(t,x,p) = \infty \, \, \text{uniformly with respect to} \, \, (x,t).\label{As: coercive}
\end{equation}
Additionally, we assume there is a $D > 0$ such that, for each $i$,
\begin{equation} \label{As: time_bound}
|H_{i}(t,x,p) - H_{i}(s,y,p)| \leq D |t -s| \quad \text{if} \, \, (x,t,p), (s,y,p) \in \overline{I_{i}} \times [0,T] \times \mathbb{R} .
\end{equation}
Finally, we assume that
\begin{equation} \label{As: sub-sup}
M := \sup\left\{|H_{i}(t,x,0)| \, \mid \, (x,t) \in \overline{I_{i}} \times [0,T], \, \, i \in \{1,2,\dots,K\}\right\} < \infty.
\end{equation}

Regarding the initial datum in \eqref{E: time}, we only require that
\begin{equation} \label{As: data_lip}
u_{0} \in \text{Lip} \left(\mathcal{I}\right),
\end{equation}
where $\text{Lip}(\mathcal{I})$ denote the spaces of (possibly unbounded) uniformly Lipschitz functions on $\mathcal{I}$.  (The topology we put on $\mathcal{I}$ is made precise in Subsection \ref{S: notation}.)

In this paper, assumptions \eqref{As: sub-sup} and \eqref{As: time_bound} are used in order to obtain uniform Lipschitz bounds, which, in turn, are used in the derivation of error estimates.  It is known that without a quantitative assumption like \eqref{As: time_bound}, it may not be possible to prove space-time Lipschitz estimates for solutions of the associated HJ equations.  See the counter-example in \cite[Section 5]{cardaliaguet cannarsa}.  

Concerning the network geometry, we only restrict to unbounded edges so as to avoid addressing questions related to boundary layers at the other ends.  In fact, the techniques of this paper carry over to the analysis of Hamilton-Jacobi equations on finite networks with a combination of Kirchoff junction conditions, Dirichlet conditions, and (generalized) Neumann conditions provided the solutions satisfy the Dirichlet boundary conditions classically.  For Hamilton-Jacobi equations on network geometries more general than junctions, see, for instance, \cite{actual networks}, \cite{Imbert}, \cite{illustrated_barles}, and the references therein.  

In what follows, we always assume that $B = 0$ in \eqref{E: stationary} and \eqref{E: time}.  However, the proofs still work with minor modifications for arbitrary choices of $B \in \mathbb{R}$.  


\subsection{Related Work}  The well-posedness of \eqref{E: stationary} and \eqref{E: time} was recently established by Lions and Souganidis in \cite{time-independent,time-dependent}.  In addition to establishing comparison for these equations for general (non-convex) Hamiltonians, they showed that HJ equations with Kirchoff junction conditions include as a special case the so-called flux-limited Hamilton-Jacobi equations introduced by Imbert and Monneau \cite{Imbert} in the setting of convex and quasi-convex Hamiltonians.  

For a comprehensive discussion of the various notions of solutions for HJ equations on junctions, including Kirchoff conditions and flux-limiters and the relations between them, see the book \cite{illustrated_barles} by Barles and Chasseigne.  \cite{illustrated_barles} also contains a presentation of the proof of \cite{time-dependent} in which the blow-up argument is treated in essentially the same manner as is done here.  

Error analysis of finite-difference schemes approximating flux-limited HJ equations with quasi-convex Hamiltonians was already conducted by Guerand and Koumaiha in \cite{flux rate}.  They obtained the $\epsilon^{\frac{1}{2}}$ rate when the equation is strictly flux-limited, and an $\epsilon^{\frac{2}{5}}$ rate in general.  Their approach relies heavily on a so-called vertex test function, which is used in place of the traditional quadratic term in a variable doubling argument.  The test function is specifically adapted to the Hamiltonian and the convexity of the latter is used in a fundamental way.

Finally, we note that there are similarities between the Kirchoff junction condition and Neumann boundary conditions, and these similarities are exposed in the present work.  Most notably, we rely on a continuity property of the junction condition that was first recognized by Lions in \cite{neumann} in the context of HJ equations with Neumann boundary conditions.  We refer to Appendix \ref{A: reformulation} for this continuity property, which can be phrased as an equivalent definition of the Kirchoff condition.  In the setting of HJ equations with Neumann boundary conditions, the point of view of Lions was used by Rouy to obtain the $\epsilon^{\frac{1}{2}}$ convergence rate for finite-difference schemes in \cite{Rouy}.  This initially inspired our idea to reformulate the Kirchoff condition. 

\subsection{Outline}  The paper is divided into three parts.  In the first part, we repeat the well-posedness results of \cite{time-dependent}, showing that the equations are well-posed and demonstrating how to quantify the blow-up argument in the time-dependent case.  The second part is devoted to the error analysis.  Finally, the third part consists of appendices in which we provide the details for a number of technical results that were used, including the reformulation of the Kirchoff condition.

The well-posedness of the time-independent problems is treated in Section \ref{S: stationary_comparison}.  The corresponding results for the Cauchy problems are addressed in Section \ref{S: time_comparison}.  

Sections \ref{S: stationary_problem} and \ref{S: cauchy_problem} are devoted to the error estimates for the vanishing viscosity approximation in the time-independent and time-dependent settings, respectively.  Sections \ref{S: stationary_diff} and \ref{S: cauchy_problem_diff} discuss the error estimates for the finite-difference schemes.  

The finite-difference schemes \eqref{E: stationary_diff} and \eqref{E: cauchy_diff} are introduced in Subsections \ref{S: stationary_diff_explanation} and \ref{S: cauchy_diff_explanation}, respectively, while the details regarding their well-posedness are provided in Subsection \ref{S: stationary_fd_explanation} and Appendix \ref{A: cauchy_fd_explanation}.  

The reformulation of the Kirchoff junction condition that implies the continuity property mentioned above is presented in Appendix \ref{A: reformulation}.   The viscosity theoretic result that allows us to quantify the blow-up argument is discussed in Appendix \ref{A: dimensionality_reduction}.  The existence of solutions of the Cauchy problems \eqref{E: time} and \eqref{E: viscous_time} is proved in Appendix \ref{A: cauchy_existence}.  Finally, in Appendix \ref{A: purely_technical}, we give a proof of a technical result used in \cite{time-dependent} and in Section \ref{S: stationary_comparison}.

\subsection{Notation and Conventions} \label{S: notation}  We let $d : \mathcal{I}^{2} \to [0,\infty)$ be the metric on the network given by 
\begin{equation*}
d(x,y) = \left\{ \begin{array}{r l}
			|x - y|, & \text{if} \, \, x,y \in \overline{I_{i}} \, \, \text{for some} \, \, i, \\
			|x| + |y|, & \text{otherwise}.
			\end{array} \right.
\end{equation*}
In what follows, the topology on $\mathcal{I}$ is always the one induced by $d$.  In particular, $C(\mathcal{I})$ is the space of (possibly unbounded) functions defined in $\mathcal{I}$ that are continuous with respect to $d$.  $UC(\mathcal{I})$ is the subspace consisting of functions that are uniformly continuous in $\mathcal{I}$.

We put the product topology on $\mathcal{I} \times [0,T]$ obtained from the $d$-metric topology on $\mathcal{I}$ and the standard Euclidean topology on $[0,T]$.  $C(\mathcal{I} \times [0,T])$ will be used to denote the space of (possibly unbounded) continuous functions on $\mathcal{I} \times [0,T]$, and $UC(\mathcal{I} \times [0,T])$, the subspace of functions uniformly continuous in $\mathcal{I} \times [0,T]$.  

In Section \ref{S: cauchy_problem} and Appendix \ref{A: cauchy_existence}, it will frequently be convenient employ the abbreviation $I_{i}^{\delta} = I_{i} \cap (-\delta,0)$ for a given $\delta > 0$.  We also write $\overline{I_{i}^{\delta}} = I_{i} \cap [-\delta,0]$.  In this case, to stress the dependence on $i$, it is convenient to denote by $-\delta_{i}$ the point with coordinate $-\delta$ in $\overline{I^{\delta}_{i}}$.

If $\varphi \in C(\mathcal{I})$ and $x \in I_{i}$ for some $i \in \{1,2,\dots,K\}$, we let $\varphi_{x_{i}}(x)$ and $\varphi_{x_{i}x_{i}}(x)$ denote the first and second derivatives of $\varphi$ at $x$ with respect to the differential structure on $I_{i}$ inherited from the real line, provided they exist.  When $x = 0$, these should be understood as the one-sided derivatives of $\varphi$ at $0$ in $I_{i}$.  In this paper, we will always write $\varphi_{x_{i}}(0)$ for this one-sided derivative and never use the notation $\varphi_{x_{i}}(0^{-})$.  

For $k \in \mathbb{N}$, we denote by $C^{k}(\mathcal{I})$ the space of functions $\varphi \in C(\mathcal{I})$ such that, for each $i$, $\varphi$ restricts to a $C^{k}$-function on $\overline{I_{i}}$.  Note that if $\varphi \in C^{k}(\mathcal{I})$, then, in general, $\varphi_{x_{i}}(0) \neq \varphi_{x_{j}}(0)$ if $i \neq j$, and the same can be said of second derivatives when $k = 2$.  A prototypical example of a function in $C^{1}(\mathcal{I})$ is given by $\varphi(x) = 2x$ in $\overline{I_{1}}$ and $\varphi(x) = 0$, otherwise. 

Similarly, we denote by $C^{k,1}(\mathcal{I} \times [0,T])$ the space of functions $\varphi \in C(\mathcal{I} \times [0,T])$ such that, for each $x \in \mathcal{I}$, $\varphi(x,\cdot) \in C^{1}([0,T])$, and, for each $t \in [0,T]$, $\varphi(\cdot, t) \in C^{k}(\mathcal{I})$.  

We denote by $\text{Lip}(\mathcal{I})$ the space of continuous functions $u \in UC(\mathcal{I})$ such that $\text{Lip}(u) := \sup \left\{\frac{|u(x) - u(y)|}{d(x,y)} \, \mid \, x, y \in \mathcal{I}, \, \, x \neq y \right\} < \infty$.  Similarly, $\text{Lip}(\mathcal{I} \times [0,T])$ is the space of continuous functions $u \in UC(\mathcal{I} \times [0,T])$ such that 
	\begin{equation*}
		\text{Lip}(u) := \sup \left\{\frac{|u(x,t) - u(y,s)|}{d(x,y) + |t - s|} \, \mid \, (x,t), (y,s) \in \mathcal{I} \times [0,T], (x,t) \neq (y,s)\right\} < \infty.
	\end{equation*}  

Recall that a function $f : X \to \mathbb{R}$ defined on a topological space $X$ is upper (resp.\ lower) semi-continuous if for each $\alpha \in \mathbb{R}$, the set $\{x \in \mathcal{I} \, \mid \, f(x) < \alpha\}$ (resp.\ $\{x \in \mathcal{I} \, \mid \, f(x) > \alpha\}$) is open in $X$.  We denote by $\text{USC}(\mathcal{I})$ the space of upper semi-continuous functions on $\mathcal{I}$, and by $\text{LSC}(\mathcal{I})$, the space of lower semi-continuous functions.  $\text{USC}(\mathcal{I} \times [0,T])$ and $\text{LSC}(\mathcal{I} \times [0,T])$ denote the respective spaces on $\mathcal{I} \times [0,T]$.

In the construction of solutions of the Cauchy problems \eqref{E: time} and \eqref{E: viscous_time}, it will be convenient to introduce a number of semi-norms.  Given $\alpha \in (0,1]$, $i \in \{1,2,\dots,K\}$, and a function $u : \overline{I_{i}} \times [0,T] \to \mathbb{R}$, we define the semi-norms $[u]_{i,\alpha}$ and $[u]_{i,0}$ by
	\begin{align*}
		[u]_{i,\alpha} &= \sup \left\{ \frac{u(x,t) - u(y,s)}{(|x - y| + |t - s|^{\frac{1}{2}})^{\alpha}} \, \mid \, (x,t), (y,s) \in \overline{I_{i}} \times [0,T], \, \, \right. \\
			&\qquad \qquad \qquad \left. (x,t) \neq (y,s) \right\}, \\
		[u]_{i,0} &= \sup \left\{ |u(x,t)| \, \mid \, (x,t) \in \overline{I_{i}} \times [0,T] \right\}.
	\end{align*}
If $u : \mathcal{I} \times [0,T] \to \mathbb{R}$, we write $[u]_{\alpha} = \max_{i} [u]_{i,\alpha}$.  

In the same set-up as the previous paragraph, if $\alpha \in (0,1]$, we will write
	\begin{align*}
		[u]_{i,1 + \alpha} &= [u_{x_{i}}]_{i,\alpha}, \\
		[u]_{1 + \alpha} &= \max_{i} [u]_{i,1 + \alpha},
	\end{align*}
and we define $[u]_{k + \alpha}$ analogously when $k \in \mathbb{N} \setminus \{0,1\}$.  Notice that, by our convention, $[u]_{i,1} \neq [u_{x_{i}}]_{i,0}$ since only the left-hand side measures the regularity of $u$ in the time variable.  
	
If $u : \mathcal{I} \to \mathbb{R}$, we define $[u]_{k + \alpha}$ by treating $u$ as constant in time and following the previous prescriptions.  We proceed analogously if instead $u : \overline{I_{i}} \to \mathbb{R}$.  

Given two functions $f,g : (0,\infty) \to (0,\infty)$, we write $f = o(g)$ if $\lim_{\epsilon \to 0^{+}} \frac{f(\epsilon)}{g(\epsilon)} = 0$.

If $a,b \in \mathbb{R}$, we write $a \vee b = \max\{a,b\}$, $a \wedge b = \min\{a,b\}$, $a^{+} = \max\{a,0\}$, and $a^{-} = -\min\{a,0\}$.

Finally, we will denote by $C$ a positive constant whose exact value may change from line to line.  We will not make explicit the dependence of the constant on the Hamiltonians or the solutions.

\subsection{Preliminaries on Viscosity Solutions}  If $\{F_{1},\dots,F_{K}\}$ is a family of functions such that $F_{i} : \overline{I_{i}} \times \mathbb{R} \times \mathbb{R} \times \mathbb{R} \to \mathbb{R}$ for each $i \in \{1,2,\dots,K\}$, we say that $u \in \text{USC}(\mathcal{I})$ (resp.\ $u \in \text{LSC}(\mathcal{I})$) is a sub-solution (resp.\ super-solution) of the equation
	\begin{equation*}
		\left\{ 
			\begin{array}{r l}
				F_{i}(x,u,u_{x_{i}},u_{x_{i}x_{i}}) = 0 & \text{in} \, \, I_{i} \\
				\sum_{i = 1}^{K} u_{x_{i}} = B & \text{on} \, \, \{0\}
			\end{array} 
		\right.
	\end{equation*}
if for each $\varphi \in C^{2}(\mathcal{I})$ such that $u - \varphi$ has a local maximum (resp. local minimum) at $x_{0} \in \mathcal{I}$, the following conditions are satisfied:
	\begin{equation*}
		\left\{
			\begin{array}{r l}
				F_{i}(x_{0},u(x_{0}),\varphi_{x_{i}}(x_{0}),\varphi_{x_{i}x_{i}}(x_{0})) \leq 0 & \text{if} \, \, x_{0} \in I_{i}, \\
				\min \left\{ \sum_{i = 1}^{K} \varphi_{x_{i}}(0) - B, \min_{i} F_{i}(0,u(0), \varphi_{x_{i}}(0),\varphi_{x_{i}x_{i}}(0)) \right\} \leq 0 & \text{if} \, \, x_{0} = 0
			\end{array}
		\right.
	\end{equation*}
(resp.\ 
	\begin{equation*}
		\left\{
			\begin{array}{r l}
				F_{i}(x_{0},u(x_{0}), \varphi_{x_{i}}(x_{0}),\varphi_{x_{i}x_{i}}(x_{0})) \geq 0 & \text{if} \, \, x_{0} \in I_{i}, \\
				\max\left\{ \sum_{i = 1}^{K} \varphi_{x_{i}}(0) - B, \max_{i} F_{i}(0,u(0), \varphi_{x_{i}}(0),\varphi_{x_{i}x_{i}}(0)) \right\} \geq 0 & \text{if} \, \, x_{0} = 0.)
			\end{array}
		\right.
	\end{equation*}
We say that $u \in C(\mathcal{I})$ is a viscosity solution if it is both a sub- and a super-solution.  
	
Similarly, we say that $u \in \text{USC}(\mathcal{I} \times [0,T])$ (resp.\ $u \in \text{LSC}(\mathcal{I} \times [0,T])$) is a sub-solution (resp.\ super-solution) of the equation
	\begin{equation*}
		\left\{
			\begin{array}{r l}
				F_{i}(t,x,u,u_{t},u_{x_{i}},u_{x_{i}x_{i}}) = 0 & \text{in} \, \, I_{i} \times (0,T) \\
				\sum_{i = 1}^{K} u_{x_{i}} = B & \text{on} \, \, \{0\} \times (0,T) \\
				u = u_{0} & \text{on} \, \, \mathcal{I} \times \{0\}
			\end{array} 
		\right.
	\end{equation*}
if $u \leq u_{0}$ (resp.\ $u \geq u_{0}$) on $\mathcal{I} \times \{0\}$ and for each $\varphi \in C^{2,1}(\mathcal{I} \times [0,T])$ such that $u - \varphi$ has a local maximum (resp.\ local minimum) at $(x_{0},t_{0}) \in \mathcal{I} \times (0,T]$, either
	\begin{equation*}
				F_{i}(t_{0},x_{0},u(x_{0},t_{0}),\varphi_{t}(x_{0},t_{0}),\varphi_{x_{i}}(x_{0},t_{0}),\varphi_{x_{i}x_{i}}(x_{0},t_{0})) \leq 0 \quad (\text{resp.} \, \, \geq 0)
	\end{equation*} 
if $x_{0} \in I_{i}$, or 
	\begin{equation*}
		\min \left\{ \sum_{i = 1}^{K} \varphi_{x_{i}}(0,t_{0}) - B, \min_{i} F_{i}(t_{0},0,u(0,t_{0}),\varphi_{t}(0,t_{0}),\varphi_{x_{i}}(0,t_{0}),\varphi_{x_{i}x_{i}}(0,t_{0})) \right\} \leq 0
	\end{equation*}
if $x_{0} = 0$
(resp.
	\begin{equation*}
		\max \left\{ \sum_{i = 1}^{K} \varphi_{x_{i}}(0,t_{0}) - B, \max_{i} F_{i}(t_{0},0,u(0,t_{0}),\varphi_{t}(0,t_{0}),\varphi_{x_{i}}(0,t_{0}),\varphi_{x_{i}x_{i}}(0,t_{0})) \right\} \geq 0
	\end{equation*}
if $x_{0} = 0$.)
As in the time-independent case, $u \in C(\mathcal{I} \times [0,T])$ is a viscosity solution if it is both a sub- and super-solution.
	
We remark that in \eqref{E: viscous_stationary} and \eqref{E: viscous_time}, the sub- and super-solution conditions at $0$ simplify considerably due to the regularizing effect of the second order term.  In this paper, we will only need that fact in the time-independent case.  It is stated precisely in Lemma \ref{L: classical_kirchoff} below.    

Throughout the paper, unless stated otherwise, we always work with viscosity solutions.  Therefore, it should be assumed that differential equations and inequalities are understood in the viscosity sense, and we will not repeat the word ``viscosity" in each statement.


\section{Stationary Problem: Existence and Uniqueness} \label{S: stationary_comparison}

We begin by showing that \eqref{E: stationary} and \eqref{E: viscous_stationary} are well-posed.  The two main results of this section are:

\begin{theorem} \label{T: existence_uniqueness_inviscid_stationary}  If $u$ is a bounded, upper semi-continuous sub-solution of \eqref{E: stationary} and $v$ is a bounded, continuous super-solution, then $u \leq v$ in $\mathcal{I}$.  \end{theorem}  

\begin{theorem} \label{T: existence_uniqueness_viscous_stationary}  If $u$ is a bounded, upper semi-continuous sub-solution of \eqref{E: viscous_stationary} and $v$ is a bounded, continuous super-solution, then $u \leq v$ in $\mathcal{I}$.  \end{theorem} 

Theorem \ref{T: existence_uniqueness_inviscid_stationary} is an application of the approach introduced in \cite{time-dependent}.  Theorem \ref{T: existence_uniqueness_viscous_stationary} uses more-or-less standard techniques from the theory of viscosity solutions, the Kirchoff condition notwithstanding.  
 
 In view of the technicalities arising from the Kirchoff condition, we will need to carefully study the behavior of sub- and super-solutions near the junction.  We will proceed by first stating the necessary lemmas and giving the proofs of Theorems \ref{T: existence_uniqueness_inviscid_stationary} and \ref{T: existence_uniqueness_viscous_stationary}.  The remainder of the section will be devoted to the proofs of the lemmas.

Existence of bounded solutions of \eqref{E: stationary} and \eqref{E: viscous_stationary} can be proved in this setting using Perron's Method arguing as in \cite{User}.  Here a key input is \eqref{As: sub-sup}, which provides a priori bounds on solutions.  Alternatively, in the case of \eqref{E: stationary}, we discuss in Remark \ref{R: half_relaxed_stationary} below how to prove existence using the finite-difference approximation and the method of half-relaxed limits.

\subsection{Proof of Theorem \ref{T: existence_uniqueness_inviscid_stationary}}

By the coercivity assumption \eqref{As: coercive}, uniformly bounded, upper semi-continuous sub-solutions of \eqref{E: stationary} and \eqref{E: viscous_stationary} are equi-Lipschitz.  We state this as the first lemma:

\begin{lemma} \label{L: uniform_Lipschitz}  If $u$ is a bounded, upper semi-continuous sub-solution of \eqref{E: stationary} or \eqref{E: viscous_stationary}, then there is a constant $L \geq 0$ depending on $\sup \left\{|u(x)| \, \mid \, x \in \mathcal{I} \right\}$, but not on $\epsilon$, such that $\text{Lip}(u) \leq KL$.  \end{lemma}  

Our proof of Lemma \ref{L: uniform_Lipschitz} is inspired by arguments coming from \cite[Section 8]{barles book}.  Note that since we are working in one space dimension, the proof applies to the viscous equation as well as the inviscid one.

The proof of Theorem \ref{T: existence_uniqueness_inviscid_stationary} relies on the following argument appearing in the proof of \cite[Theorem 2.1]{time-dependent}.  Since it is so important, we state it as a theorem in its own right.  This result establishes a sort-of maximum principle at the junction that allows us to rule out the possibility that a sub-solution is larger than a super-solution there.  

\begin{theorem} \label{T: comparison} Suppose $u : \mathcal{I} \to \mathbb{R}$ is a continuous sub-solution of \eqref{E: stationary}, and define $p^{+}_{i} = \limsup_{I_{i} \ni x \to 0} \frac{u(x)- u(0)}{x}$ for each $i \in \{1,2,\dots,K\}$.  Then the following conditions are satisfied: 
	\begin{itemize}
		\item[(i)] $|p_{i}^{+}| < \infty$ and $u(0) + H_{i}(0,p_{i}^{+}) \leq 0$ 
		\item[(ii)] If $(\tilde{p}_{1},\dots,\tilde{p}_{K}) \in \mathbb{R}^{K}$ satisfies $\tilde{p}_{i} \leq p_{i}^{+}$ for each $i \in \{1,2,\dots,K\}$, then
			\begin{equation} \label{E: crucial_part}
				\min \left\{ \sum_{i = 1}^{K} \tilde{p}_{i}, u(0) + \min_{i} H_{i}(0,\tilde{p}_{i}) \right\} \leq 0.
			\end{equation}
	\end{itemize}

Similarly, suppose $v : \mathcal{I}\ \to \mathbb{R}$ is a continuous super-solution of \eqref{E: stationary}, and define $q^{-}_{i} = \liminf_{I_{i} \ni x \to 0} \frac{v(x) - v   (0)}{x}$ for each $i \in \{1,2,\dots,K\}$.  Then 
	\begin{itemize}
		\item[(i)] If $|q_{i}^{-}| < \infty$, then $v(0) + H_{i}(0,q_{i}^{-}) \geq 0$.
		\item[(ii)] If $(\tilde{q}_{1},\dots,\tilde{q}_{K}) \in \mathbb{R}^{K}$ satisfies $\tilde{q}_{i} \geq q_{i}^{-}$ for each $i \in \{1,2,\dots,K\}$, then
			\begin{equation*}
				\max \left\{ \sum_{i = 1}^{K} \tilde{q}_{i}, v(0) + \max_{i} H_{i}(0,\tilde{q}_{i}) \right\} \geq 0.
			\end{equation*}
	\end{itemize}
\end{theorem}  

For the proof of Theorem \ref{T: comparison}, see Appendix \ref{A: purely_technical}.  

We now prove Theorem \ref{T: existence_uniqueness_inviscid_stationary}:

\begin{proof}[Proof of Theorem \ref{T: existence_uniqueness_inviscid_stationary}]  We argue by contradiction.  First, we remark that an elementary argument shows that for each $i \in \{1,2,\dots,K\}$, 
\begin{equation} \label{E: key_inequality}
\sup \left\{ u(x) - v(x) \, \mid \, x \in I_{i} \right\} \leq (u(0) - v(0))^{+}.
\end{equation}  

In view of \eqref{E: key_inequality}, it only remains to argue that $u(0) \leq v(0)$.  In what follows, we will assume $u(0) > v(0)$ and use Theorem \ref{T: comparison} and \cite[Lemma 3.1]{time-dependent} to obtain a contradiction.

Since our assumption implies $u(0) - v(0) = \sup \left\{u(x) -v(x) \, \mid \, x \in \mathcal{I} \right\}$, the function $x \mapsto u(x) - v(x)$ is maximized at $0$.  Thus, independently of the choice of $i \in \{1,2,\dots,K\}$, we have
\begin{equation*}
p_{i} := \limsup_{I_{i} \ni x \to 0} \frac{u(x) - u(0)}{x} \geq \liminf_{I_{i} \ni x \to 0} \frac{v(x) - v(0)}{x} =: q_{i}.
\end{equation*}
We claim that the following conditions are satisfied:
\begin{itemize}
\item[(a)] If $i \in \{1,2,\dots,K\}$, then $p_{i} \geq q_{i}$, $u(0) + H_{i}(0,p_{i}) \leq 0$, and if $q_{i} > -\infty$, then $0 \leq v(0) + H_{i}(0,q_{i})$.
\item[(b)] If $(\tilde{p}_{1},\dots,\tilde{p}_{K}) \in \mathbb{R}^{K}$ satisfies $\tilde{p}_{i} \leq p_{i}$ for all $i \in \{1,2,\dots,K\}$, then 
\begin{equation*}
\min \left\{ \sum_{i = 1}^{K} \tilde{p}_{i}, u(0) + \min_{i} H_{i}(0,\tilde{p}_{i}) \right\} \leq 0.
\end{equation*}
\item[(c)] If $(\tilde{q}_{1},\dots,\tilde{q}_{K}) \in \mathbb{R}^{K}$ satisfies $\tilde{q}_{i} \geq q_{i}$ for all $i \in \{1,2,\dots,K\}$, then 
\begin{equation*}
\max \left\{ \sum_{i = 1}^{K} \tilde{q}_{i}, v(0) + \max_{i} H_{i}(0,\tilde{q}_{i}) \right\} \geq 0.
\end{equation*}
\end{itemize}

Regarding (a), we have already established that $p_{i} \geq q_{i}$, no matter the choice of $i$.  The rest of the assertions in (a)-(c) follow from the definitions of $(p_{1},\dots,p_{K})$ and $(q_{1},\dots,q_{K})$ and a direct application of Theorem \ref{T: comparison}.

Let $a = u(0)$ and $b = v(0)$.  At this stage, it is convenient to define a vector $(q_{1}^{*},\dots,q_{K}^{*}) \in \mathbb{R}^{K}$ to replace $(q_{1},\dots,q_{K})$.  If $q_{i} > -\infty$, define $q_{i}^{*} = q_{i}$.  Otherwise, pick $q_{i}^{*} \in (-\infty,p_{i})$ small enough that $b + H_{i}(0,q_{i}^{*}) \geq 0$.  Now items (a)-(c) above show that the numbers $a$ and $b$ and vectors $(p_{1},\dots,p_{K})$ and $(q_{1}^{*},\dots,q_{K}^{*})$ satisfy the hypotheses of \cite[Lemma 3.1]{time-dependent}.  Therefore, according to the conclusion of that lemma, the inequality $a \leq b$ holds.  By the definition of $a$ and $b$, this contradicts our assumption that $u(0) > v(0)$.  

We conclude that $u(0) \leq v(0)$ and, thus, by \eqref{E: key_inequality}, $u \leq v$ in $\mathcal{I}$.\end{proof}  

\subsection{Proof of Theorem \ref{T: existence_uniqueness_viscous_stationary}}  The viscous case has to be handled slightly differently than the inviscid one.  In fact, the viscosity term makes the arguments easier since the Kirchoff condition has to hold classically.  That fact is implied by the next two lemmas:

\begin{lemma} \label{L: differentiability}  Fix $i \in \{1,2,\dots,K\}$.  Suppose $\epsilon > 0$ and $u : (-\infty,0] \to \mathbb{R}$ is a continuous sub-solution of $u - \epsilon u_{xx} + H_{i}(x,u_{x}) = 0$ in $(-\infty,0)$.  Then $\lim_{x \to 0^{-}} \frac{u(x) - u(0)}{x}$ exists.\end{lemma}

\begin{lemma} \label{L: classical_kirchoff}  If $u$ is a bounded, upper semi-continuous sub-solution of \eqref{E: viscous_stationary} in $\mathcal{I}$, $\varphi \in C^{2}(\mathcal{I})$, and $u - \varphi$ has a local maximum at $0$, then $\sum_{i = 1}^{K} \varphi_{x_{i}}(0) \leq 0$.  Similarly, if $v$ is a bounded, lower semi-continuous sub-solution of \eqref{E: viscous_stationary} in $\mathcal{I}$, $\varphi \in C^{2}(\mathcal{I})$, and $u - \varphi$ has a local minimum at $0$, then $\sum_{i = 1}^{K} \varphi_{x_{i}}(0) \geq 0$.   \end{lemma}  

The next remark will be helpful in the proof of Theorem \ref{T: existence_uniqueness_viscous_stationary} and Lemma \ref{L: uniform_Lipschitz}:  

\begin{remark} \label{R: calculus}  Suppose $u : (-\infty,0] \to \mathbb{R}$ is upper semi-continuous.  Define $p \in [-\infty,\infty]$ by 
\begin{equation*}
p= \liminf_{x \to 0^{-}} \frac{u(x) - u(0)}{x}.
\end{equation*}
Furthermore, assume that $p > -\infty$ and $\tilde{p} \in (-\infty,p)$.  Then a straightforward calculus exercise shows $x \mapsto u(x) - \tilde{p}x$ has a local maximum at $0$.

Similarly, if $v : (-\infty, 0] \to \mathbb{R}$ is lower semi-continuous, $q \in [-\infty,\infty]$ is defined by 
	\begin{equation*}
		q= \limsup_{x \to 0^{-}} \frac{v(x) - v(0)}{x},
	\end{equation*}
and $q < \tilde{q} < \infty$, then $x \mapsto v(x) - \tilde{q}x$ has a local minimum at $0$.

Combining the previous observations with Lemma \ref{L: classical_kirchoff}, we see that if $v$ is a bounded, lower-semi-continuous sub-solution of \eqref{E: viscous_stationary} in $\mathcal{I}$ and if $q_{i} = \limsup_{I_{i} \ni x \to 0} \frac{v(x) - v(0)}{x}$ for each $i$, then $\sum_{i = 1}^{K} \tilde{q}_{i} \geq 0$ whenever $(\tilde{q}_{1},\dots,\tilde{q}_{K}) \in \mathbb{R}^{K}$ satisfies $\tilde{q}_{i} > q_{i}$ for all $i$.  Passing to the limit $(\tilde{q}_{1},\dots,\tilde{q}_{K}) \to (q_{1},\dots,q_{K})$, we conclude that $\sum_{i = 1}^{K} q_{i} \geq 0$.        \end{remark}  

The proof of Lemma \ref{L: classical_kirchoff} is very similar in spirit to that of Lemma \ref{L: differentiability}, and is therefore omitted.  Deferring the proof of Lemma \ref{L: differentiability} until later, we now prove Theorem \ref{T: existence_uniqueness_viscous_stationary}:\

\begin{proof}[Proof of Theorem \ref{T: existence_uniqueness_viscous_stationary}]  First, as in the inviscid case, we note that for each $i \in \{1,2,\dots,K\}$, 
\begin{equation} \label{E: key_inequality_viscous}
\sup \left\{u(x) - v(x) \, \mid \, x \in I_{i} \right\} \leq (u(0) - v(0))^{+}.
\end{equation}
Therefore, to conclude, we only need to show $u(0) \leq v(0)$.  

To prove this, we bend $u$ a little bit to make the Kirchoff condition strict.  One convenient way to do this is to fix $\delta, \Delta > 0$ and define $u^{\delta,\Delta}$ in $\mathcal{I}$ by 
\begin{equation*}
u^{\delta,\Delta}(x) = u(x) - \delta x - \frac{\Delta x^{2}}{2}.
\end{equation*}
By Lemmas \ref{L: differentiability} and \ref{L: classical_kirchoff}, $u^{\delta,\Delta}$ is differentiable from the left at $0$ in each ray, and $\sum_{i = 1}^{K} u_{x_{i}}^{\delta,\Delta}(0) = \sum_{i = 1}^{K} (u_{x_{i}}(0) - \delta) \leq -\delta K$.  

We claim that there is at least one $i \in \{1,2,\dots,K\}$ such that the maximum of $u^{\delta,\Delta} - v$ in $\overline{I_{i}}$ is not achieved at zero.  Indeed, this follows from Lemmas \ref{L: differentiability} and \ref{L: classical_kirchoff} and the last paragraph of Remark \ref{R: calculus}, which together imply
\begin{equation*}
\sum_{i = 1}^{K} \left(\liminf_{I_{i} \ni x \to 0} \frac{(u^{\delta,\Delta}(x) - v(x)) - (u^{\delta,\Delta}(0) - v(0))}{x}\right) \leq \sum_{i = 1}^{K} u_{x_{i}}^{\delta,\Delta}(0) \leq  - \delta K < 0.
\end{equation*}
Thus, $\liminf_{I_{j} \ni x \to 0} \frac{(u^{\delta,\Delta}(x) - v(x)) - (u^{\delta,\Delta}(0) - v(0))}{x} < 0$ for at least one $j$, which implies $u^{\delta,\Delta} - v$ is not maximized at $0$. Therefore, in what follows, we fix $j_{\delta} \in \{1,2,\dots,K\}$ such that 
	\begin{equation*}
		\sup \left\{u^{\delta,\Delta}(x) - v(x) \, \mid \, x \in I_{j_{\delta}}\right\} > u^{\delta,\Delta}(0) - v(0).
	\end{equation*}

Note additionally that although the point where $u^{\delta,\Delta} - v$ is maximized in $I_{j_{\delta}}$ may not be unique, the distance to any such point will converge to zero as $\delta \to 0^{+}$ by \eqref{E: key_inequality_viscous} and the definition of $u^{0,\Delta}$.  That is, if $d_{\delta,\Delta} = \max \left\{ d(y,0) \, \mid \, y \, \, \text{maximizes} \, \, u^{\delta,\Delta} - v \, \, \text{in} \, \, I_{j_{\delta}}\right\}$, then $\lim_{\delta \to 0^{+}} d_{\delta,\Delta} = 0$.

Let $\kappa > 0$ and consider the function $\Phi : \overline{I_{j_{\delta}}} \times \overline{I_{j_{\delta}}} \to \mathbb{R}$ given by 
\begin{equation*}
\Phi(x,y) = u(x) - v(y) - \frac{|x - y|^{2}}{2 \kappa} - \delta y - \frac{\Delta y^{2}}{2}.
\end{equation*}
In view of the definition of $\Phi$, we can fix $(x_{\kappa},y_{\kappa}) \in \overline{I_{j_{\delta}}} \times \overline{I_{j_{\delta}}}$ such that $\Phi(x_{\kappa},y_{\kappa}) = \sup \left\{\Phi(x,y) \, \mid \, x,y \in I_{j_{\delta}}\right\}$.  Moreover, $\lim_{\kappa \to 0^{+}} x_{\kappa} = \lim_{\kappa \to 0^{+}} y_{\kappa}$ equals a maximum point of $u^{\delta,\Delta} - v$ along subsequences.  In particular, 
	\begin{equation*}
		\liminf_{\kappa \to 0^{+}} \left(d(x_{\kappa},0) + d(y_{\kappa},0)\right) > 0.
	\end{equation*}
Therefore, if $\kappa$ is small enough, we can invoke the maximum principle for semi-continuous functions (cf.\ \cite[Theorem 3.2]{User}) and the equations satisfied by $u$ and $v$ to find $X, Y \in \mathbb{R}$ with $X \leq Y$ such that
\begin{equation*}
	\left\{
		\begin{array}{r}
			u(x_{\kappa}) - \epsilon X + H_{i}\left(x_{\kappa}, \frac{x_{\kappa} - y_{\kappa}}{\kappa} \right) \leq 0 \\
	v(y_{\kappa}) - \epsilon Y + \epsilon \Delta + H_{i}\left(y_{\kappa}, \frac{x_{\kappa} - y_{\kappa}}{\kappa} - \delta - \Delta y_{\kappa} \right) \geq 0
		\end{array}
	\right.
\end{equation*}

Since $\Phi_{i}(y_{\kappa},y_{\kappa}) \leq \Phi_{i}(x_{\kappa},y_{\kappa})$, the Lipschitz continuity of $u$ implies that $\kappa^{-1}|x_{\kappa} - y_{\kappa}|$ remains bounded as $\kappa \to 0^{+}$ and, thus, $\lim_{\kappa \to 0^{+}} |x_{\kappa} - y_{\kappa}| = 0$.  Therefore, we can subtract the previous inequalities to find
\begin{equation*}
u(x_{\kappa}) - v(y_{\kappa}) \leq \epsilon \Delta + \omega \left(|x_{\kappa} - y_{\kappa}| + \delta + \Delta d(y_{\kappa},0)\right),
\end{equation*}
where $\omega$ is a modulus of continuity of $H_{i}$ in $\overline{I_{i}} \times B(0,R)$ for a fixed $R > 0$ independent of $\kappa$ and $\delta$, the existence of which follows from assumption \eqref{As: continuity} and the boundedness of $\frac{x_{\kappa} - y_{\kappa}}{\kappa}$ independent of $\kappa$ and $\delta$.  

Fix a subsequence $(\kappa_{k})_{k \in \mathbb{N}} \subseteq (0,\infty)$ and a maximum point $\bar{x}_{\delta}$ of $u^{\delta,\Delta} - v$ such that $\lim_{k \to \infty} \kappa_{k} = 0$ and $\lim_{k \to \infty} x_{\kappa_{k}} = \lim_{k \to \infty} y_{\kappa_{k}} = \bar{x}_{\delta}$.  Sending $k \to \infty$ in the previous inequality, we obtain
\begin{equation} \label{E: grinding}
u(\bar{x}_{\delta}) - v(\bar{x}_{\delta}) \leq \epsilon \Delta + \omega(\delta + \Delta d_{\delta,\Delta}).
\end{equation}

Taking the limit $\delta \to 0^{+}$ with $\Delta$ fixed and recalling that $d_{\delta,\Delta} \to 0$ in the process, we find
\begin{equation*}
u(0) - v(0) \leq \epsilon \Delta.
\end{equation*}
Finally, we send $\Delta \to 0^{+}$ to conclude:
\begin{equation*}
\sup \left\{u(x) - v(x) \, \mid \, x \in \mathcal{I} \right\} \leq (u(0) - v(0))^{+} = 0.
\end{equation*}
\end{proof}    

\subsection{Proofs of Lemmas \ref{L: uniform_Lipschitz} and \ref{L: differentiability}}


%

\begin{proof}[Proof of Lemma \ref{L: uniform_Lipschitz}]First, we claim that if $j \in \{1,2,\dots,K\}$, then $u$ is uniformly Lipschitz continuous in $I_{j}$.  Fix such a $j$ and let $C = \sup \{|u(x)| \, \mid \, x \in \mathcal{I} \}$.

By \eqref{As: coercive}, there is an $L \geq 1$ such that 
\begin{equation} \label{E: coercivity_implies_lipschitz}
-C + H_{i}(x,p) \geq 1 \quad \text{if} \, \, x \in \overline{I_{i}}, \, \, |p| \geq L, \, \, i \in \{1,2,\dots,K\}.
\end{equation}
We claim that $|u(x) - u(y)| \leq KL |x - y|$ if $x,y \in I_{j}$.  

Fix $x \in I_{j}$ and define a test function $\varphi : \mathcal{I} \to \mathbb{R}$ by 
\begin{equation} \label{E: knife_test_function}
\varphi(y) = \left\{ 
				\begin{array}{r l}
							u(x) + KL|x - y|, & \text{if} \, \, y \in I_{j} \\
							u(x) + KL|x| + L|y|, & \text{if} \, \, y \in I_{i}, \, \, i \neq j
				\end{array}
		\right.
\end{equation}
Notice that $\varphi_{x_{j}}(0) = KL$ while $\varphi_{x_{i}}(0) = -L$ if $i \neq j$.  In particular, we find $\sum_{i = 1}^{K} \varphi_{x_{i}}(0) = L$.

Since $u$ is bounded, we can let $x_{0}$ be a point where $u - \varphi$ is maximized in $\mathcal{I}$.  If $x_{0} \in I_{j} \setminus \{0, x\}$, then the smoothness of $\varphi$ at $x_{0}$ together with the sub-solution property yield
\begin{equation*}
-C + H_{j}(x_{0},KL) \leq u(x_{0}) + H_{j}(x_{0},KL) \leq 0,
\end{equation*}
contradicting \eqref{E: coercivity_implies_lipschitz}.  A similar argument shows that $x_{0} \notin I_{i}$ if $i \neq j$.  Finally, if $x_{0} = 0$, then the sub-solution property implies
\begin{equation*}
\min \left\{ \sum_{i = 1}^{K} \varphi_{x_{i}}(0), u(0) + \min_{i} H_{i}(0,\varphi_{x_{i}}(0))\right\} \leq 0.
\end{equation*}
However, in view of the choice of $L$ and the definition of $\varphi$, the left-hand side is no less than $1$, a contradiction.  We conclude that $x_{0} = x$.  

Since $u - \varphi$ is maximized at $x_{0} = x$, we find
\begin{equation*}
u(y) - u(x) \leq \varphi(y) - \varphi(x) = KL |x - y| \quad \text{if} \, \, y \in I_{j}.
\end{equation*}
As $x$ is an arbitrary point in $I_{j}$, we conclude that $u$ is uniformly Lipschitz continuous in $I_{j}$ with constant $KL$.  

We have shown that $u$ is uniformly Lipschitz continuous in the interior of each edge.  To see that $u$ is uniformly Lipschitz in $\mathcal{I}$, it only remains to show that $u$ is continuous at $0$.  Note that this is not automatic since we are only assuming that $u$ is upper semi-continuous in $\mathcal{I}$.  

It is convenient to argue by contradiction.  Suppose that there is a $j \in \{1,2,\dots,K\}$ such that $\lim_{I_{j} \ni x \to 0} u(x) < u(0)$.  Since $u$ is uniformly Lipschitz in $I_{j}$, the following is an immediate consequence:      
\begin{equation*}
\limsup_{I_{j} \ni x \to 0} \frac{u(x) - u(0)}{x} = \liminf_{I_{j} \ni x \to 0} \frac{u(x) - u(0)}{x} = +\infty.
\end{equation*}
For each $i \neq j$, let $p_{i} = \min\left\{\liminf_{I_{i} \ni x \to 0} \frac{u(x) - u(0)}{x}, 0\right\}$, and then fix $E_{0} \geq \max_{i \neq j} |p_{i}|$.  By Remark \ref{R: calculus}, if $E > E_{0}$ and $\varphi : \mathcal{I} \to \mathbb{R}$ is defined by 
\begin{equation*}
\varphi(x) = \left\{
				\begin{array}{r l}
							u(0) + KEx, & x \in I_{j} \\
							u(0) - Ex, & x \in I_{i}, \, \, i \neq j
				\end{array}
			\right.
\end{equation*}
then $u - \varphi$ has a local maximum at $0$.  Appealing to the equation, we find
\begin{equation*}
\min \left\{ E, u(0) + \min\{H_{j}(0,KE), \min_{i \neq j} H_{i}(0,-E)\} \right\} \leq 0.
\end{equation*}
Sending $E \to \infty$ and invoking \eqref{As: coercive}, we obtain a contradiction.  

We conclude that $\lim_{I_{j} \ni x \to 0} u(x) = u(0)$.  Since $j$ was chosen arbitrarily, $u$ is continuous at $0$, and, therefore, uniformly Lipschitz continuous in $\mathcal{I}$ with $\text{Lip}(u) \leq KL$. 
\end{proof}  

Finally, we show that, in the viscous case, sub-solutions are necessarily differentiable at the junction.  

\begin{proof}[Proof of Lemma \ref{L: differentiability}]  We argue by contradiction.  Define $p^{-}, p^{+}$ by
\begin{equation*}
p^{+} := \limsup_{x \to 0^{-}} \frac{u(x) - u(0)}{x}, \quad \liminf_{x \to 0^{-}} \frac{u(x) - u(0)}{x} =: p^{-}.
\end{equation*}  
Note that since Lemma \ref{L: uniform_Lipschitz} applies, $\{p^{-},p^{+}\} \subseteq \mathbb{R}$ holds.  Assume that $u$ is not differentiable at $0$, that is, assume $p^{-} < p^{+}$.  

Let $p = \frac{p^{+} + p^{-}}{2}$.  Given $\Delta > 0$, define $\varphi^{\Delta} : (-\infty,0] \to \mathbb{R}$ by 
\begin{equation*}
\varphi^{\Delta}(x) = u(0) + px - \frac{\Delta x^{2}}{2}.
\end{equation*}  
Since $\varphi^{\Delta}_{x}(0) = p \in (p^{-},p^{+})$, it is not hard to show that there are sequences $(y_{n})_{n \in \mathbb{N}}, (x_{n})_{n \in \mathbb{N}} \subseteq (-\infty,0)$ such that $y_{n} < x_{n} < y_{n + 1}$, $\lim_{n \to \infty} y_{n} = 0$, and 
\begin{equation*}
u(y_{n}) - \varphi^{\Delta}(y_{n}) = 0 < u(x_{n}) - \varphi^{\Delta}(x_{n}) \quad \text{if} \, \, n \in \mathbb{N}.
\end{equation*}
Since $u - \varphi^{\Delta}$ is continuous, we can let $\bar{x}_{n}$ be a point where it achieves its maximum in $[y_{n},y_{n + 1}]$.  By construction, $\bar{x}_{n} \in (y_{n},y_{n + 1})$.  Thus, since $u$ is a sub-solution, we obtain
\begin{equation*}
u(\bar{x}_{n}) + \epsilon \Delta + H_{i}(\bar{x}_{n},p - \Delta \bar{x}_{n}) = u(\bar{x}_{n}) - \epsilon \varphi^{\Delta}_{xx}(\bar{x}_{n}) + H_{i}(\bar{x}_{n},\varphi^{\Delta}_{x}(\bar{x}_{n})) \leq 0.
\end{equation*}  

Sending $n \to \infty$ with $\Delta$ fixed, we find, by continuity of $u$,
\begin{equation*}
u(0) + \epsilon \Delta + H(0,p) \leq 0.
\end{equation*}
Since $\Delta$ is an arbitrary positive number and $\epsilon > 0$, we conclude that $u(0) + H(0,p) = -\infty$, contradicting the fact that $\{u(0),H(0,p)\} \subseteq \mathbb{R}$.  We conclude that $p^{-} = p^{+}$.  \end{proof}


\section{Cauchy Problem: Existence and Uniqueness} \label{S: time_comparison}

In this section, we revisit the comparison principle associated with \eqref{E: time} and comment briefly on the modifications necessary to establish one in the case of \eqref{E: viscous_time}.  

Where \eqref{E: time} is concerned, the comparison principle has already been proven in \cite{time-dependent}.  The proof we present here is a slight variation on the one appearing there, which is useful since we need to quantify the proof.  The difference is we show how to forego the blow-up argument.  At the request of an anonymous reviewer,  we will give a complete proof, which, at any rate, is closely related to what was done in the previous section and motivates our approach to the error estimates.

As in the time-independent case, the comparison principle for \eqref{E: viscous_time} is slightly easier to prove than for \eqref{E: time} due to the second-order term.  Since the proof can be obtained by combining some of the arguments used to treat \eqref{E: time} with those used to analyze \eqref{E: viscous_stationary}, we only give a sketch in the sequel.

The two comparison results established in this section are:

\begin{theorem} \label{T: inviscid_comparison_time}  If $u \in UC(\mathcal{I} \times [0,T])$ is a sub-solution of \eqref{E: time} and $v \in UC(\mathcal{I} \times [0,T])$ is a super-solution, then $u \leq v$ in $\mathcal{I} \times [0,T]$.  \end{theorem}  

\begin{theorem} \label{T: viscous_comparison_time}  If $u \in UC(\mathcal{I} \times [0,T])$ is a sub-solution of \eqref{E: viscous_time} and $v \in UC(\mathcal{I} \times [0,T])$ is a super-solution, then $u \leq v$ on $\mathcal{I} \times [0,T]$.    \end{theorem}

As is standard in the theory of viscosity solutions, Theorem \ref{T: inviscid_comparison_time} implies a contractivity property of the semi-groups associated with \eqref{E: time} and \eqref{E: viscous_time}.  This will be useful later when we study error estimates.  We give the precise statement in the next remark: 

\begin{remark} \label{R: contractivity}  If $u$ is a solution of \eqref{E: time} with initial datum $u_{0}$ and $v$ is a solution of \eqref{E: time} with initial datum $v_{0}$, then for each $(x,t) \in \mathcal{I} \times [0,T]$, 
	\begin{equation} \label{E: contractivity}
		|u(x,t) - v(x,t)| \leq \sup \left\{|u_{0}(x) - v_{0}(x)| \, \mid \, x \in \mathcal{I} \right\}.
	\end{equation} 
	
To see this, observe that the semi-group associated with \eqref{E: time} commutes with the addition of constant functions.  Therefore, if $C = \sup \{(u_{0}(x) - v_{0}(x))^{+} \, \mid \,x \in \mathcal{I}\}$, then $v + C$ is a solution of \eqref{E: time} as long as $v$ is, and it is at least as large as $u$ at the initial time.  Therefore, by Theorem \ref{T: inviscid_comparison_time}, $u \leq v + C$.  Reversing the roles of $u$ and $v$, we obtain \eqref{E: contractivity}.  

Of course, the same observations apply to \eqref{E: viscous_time} by Theorem \ref{T: viscous_comparison_time}. \end{remark}  

Existence of solutions of \eqref{E: time} and \eqref{E: viscous_time} is treated in Appendix \ref{A: cauchy_existence}.  In the case of \eqref{E: time}, existence also follows by applying the method of half-relaxed limits to the finite-difference scheme.  See Remark \ref{R: half_relaxed_time} below.  

\subsection{Proof of Theorem \ref{T: inviscid_comparison_time}} \label{S: proof_of_comparison_time} The proof of Theorem \ref{T: inviscid_comparison_time} uses inf- and sup-convolutions in time and a time-freezing argument.  The main ingredients are stated next as lemmas.  These are proved in subsequent subsections.  

Given $\theta > 0$, we define the sup-convolution $u^{\theta}$ of $u$ in time by 
\begin{equation} \label{E: sup_convolution}
u^{\theta}(x,t) = \sup \left\{ u(x,s) - \frac{(t - s)^{2}}{2 \theta} \, \mid \, s \in [0,T] \right\}.
\end{equation}
Analogously, the inf-convolution $v_{\theta}$ of $v$ is defined by 
\begin{equation} \label{E: inf_convolution}
v_{\theta}(x,t) = \inf \left\{ v(x,s) + \frac{(t - s)^{2}}{2 \theta} \, \mid \, s \in [0,T] \right\}.
\end{equation}

The first result we need concerns the regularity of $u^{\theta}$ and $v_{\theta}$.  In the sequel, we will use the fact that these functions are respectively semi-convex and semi-concave.  Let us start by recalling the definition, following \cite{cannarsa sinestrari}.    

We say that $f : [0,T] \to \mathbb{R}$ is semi-convex with linear modulus $\lambda_{f} >0$ if $t \mapsto f(t) + \frac{\lambda_{f} t^{2}}{2}$ is continuous and convex in $[0,T]$.  Similarly, we say that $g : [0,T] \to \mathbb{R}$ is semi-concave with linear modulus $\lambda_{g} > 0$ if $t \mapsto g(t) - \frac{\lambda_{g} t^{2}}{2}$ is continuous and concave in $[0,T]$.

\begin{prop} \label{P: sup_inf_convolution}  If $u, v \in UC(\mathcal{I} \times [0,T])$, then the functions $u^{\theta}$ and $v_{\theta}$ defined respectively in \eqref{E: sup_convolution} and \eqref{E: inf_convolution} also belong to $UC(\mathcal{I} \times [0,T])$.  Moreover, for each $x \in \mathcal{I}$, the functions $t \mapsto u^{\theta}(x,t)$ and $t \mapsto v_{\theta}(x,t)$ are respectively semi-convex and semi-concave in $[0,T]$, both with linear modulus $\theta^{-1}$. \end{prop}  

In what follows, we define $\omega : [0,\infty) \to [0,\infty)$ by 
\begin{align}
\omega(\xi) &= \sup \left\{ |u(x,t) - u(y,s)| \vee |v(x,t) - v(y,s)| \, \mid \, (x,t), (y,s) \in  \right. \label{E: modulus} \\
	&\quad \quad \quad \quad \quad \left. \mathcal{I} \times [0,T], \, \, d(x,y) + |t - s| \leq \xi \right\}. \nonumber
\end{align}
Note that the assumptions of Theorem \ref{T: inviscid_comparison_time} imply $\lim_{\xi \to 0^{+}} \omega(\xi) = 0$.  

The relationship between the functions $u^{\theta}$ and $v_{\theta}$ and the PDE is summarized in the next lemma.

\begin{lemma} \label{L: sup_inf_convolution}  Under the hypotheses of Theorem \ref{T: inviscid_comparison_time}, if $u^{\theta}$ is defined by \eqref{E: sup_convolution}, $v_{\theta}$ is defined by \eqref{E: inf_convolution}, $\omega$ is defined by \eqref{E: modulus}, and $\theta$ is sufficiently small, then there is a $T_{\theta} \in (0,T)$ such that $u^{\theta}$ (resp.\ $v_{\theta}$) is a sub-solution (resp.\ super-solution) of the following problem:
\begin{equation} \label{E: modified_equation}
\left\{ 
	\begin{array}{r l}
				u^{\theta}_{t} + H_{i}(t,x,u^{\theta}_{x_{i}}) - DT_{\theta} = 0 & \text{in} \, \, I_{i} \times (T_{\theta}, T) \\
				\sum_{i = 1}^{K} u^{\theta}_{x_{i}} = 0 & \text{on} \, \, \{0\} \times (T_{\theta},T) \\
				u^{\theta} = u_{0} + 2\omega(T_{\theta}) & \text{on} \, \, \mathcal{I}  \times \{T_{\theta}\}
	\end{array}
\right.
\end{equation}
(resp.\ 
\begin{equation} \label{E: modified_equation_2}
\left\{ 
	\begin{array}{r l}
				v_{\theta,t} + H_{i}(t,x,v_{\theta,x_{i}})  + DT_{\theta} = 0 & \text{in} \, \, I_{i} \times (T_{\theta}, T) \\
				\sum_{i = 1}^{K} v_{\theta,x_{i}} = 0 & \text{on} \, \, \{0\} \times (T_{\theta},T) \\
				v_{\theta} = u_{0} - 2\omega(T_{\theta}) & \text{on} \, \, \mathcal{I} \times \{T_{\theta}\}).
	\end{array}
\right.
\end{equation}
Moreover, $\lim_{\theta \to 0^{+}} T_{\theta} = 0$, and $u^{\theta} \in \text{Lip}(\mathcal{I} \times [T_{\theta},T])$.  
\end{lemma}  

The remainder of the proof of Theorem \ref{T: inviscid_comparison_time} consists in estimating $u^{\theta} - v_{\theta}$ and then sending $\theta \to 0^{+}$.  First, we need a weak comparison result that shows, effectively, that this task reduces to studying $u^{\theta} - v_{\theta}$ near the surfaces $\{t = T_{\theta}\}$ and $\{x = 0\}$.  

\begin{lemma} \label{L: basic_comparison}  Under the hypotheses of Theorem \ref{T: inviscid_comparison_time}, if $\delta > 0$ and $\theta$ is small enough that $T_{\theta} < T$, then the following inequality holds:
\begin{equation} \label{E: gronwall}
\sup \left\{ u^{\theta}(x,t) - v_{\theta}(x,t) - (2DT_{\theta} + \delta)t \, \mid \, (x,t) \in \mathcal{I} \times [T_{\theta},T] \right\} \leq f_{\delta}(\theta,u,v),
\end{equation}
where $f_{\delta}$ is given by 
\begin{equation*}
f_{\delta}(\theta,u,v) = (4 \omega(T_{\theta})) \vee \max\{u^{\theta}(0,t) - v_{\theta}(0,t) - (2DT_{\theta} + \delta) t \, \mid \, t \in [T_{\theta},T]\}.
\end{equation*}
\end{lemma}

Next, we need to show that the behavior at $\{x = 0\}$ is actually controlled by what happens at time $T_{\theta}$.  This is where we freeze time and use the junction condition, and the result is stated in the next lemma:  

\begin{lemma} \label{L: zero_behavior} Under the hypotheses of Theorem \ref{T: inviscid_comparison_time}, if $u^{\theta}$ is defined by \eqref{E: sup_convolution}, $v_{\theta}$ is defined by \eqref{E: inf_convolution}, and $\theta$ is so small that $T_{\theta} < T$, then 
\begin{equation} \label{E: boundary_behavior}
\max \left\{ u^{\theta}(0,t) - v_{\theta}(0,t) - 2 DT_{\theta}t \, \mid \, t \in [T_{\theta},T] \right\} \leq 4 \omega(T_{\theta}).
\end{equation}
\end{lemma}    

With the lemmas in hand, we can prove Theorem \ref{T: inviscid_comparison_time}:

\begin{proof}[Proof of Theorem \ref{T: inviscid_comparison_time}]  Suppose $(x,t) \in \mathcal{I} \times [0,T]$.  We claim that $u(x,t) - v(x,t) \leq 0$.  Of course, if $t = 0$, then this follows from the fact that $u(x,0) \leq v(x,0)$.  Therefore, in what follows, assume $t > 0$.  

Let $\theta \mapsto T_{\theta}$ be the function defined in Lemma \ref{L: sup_inf_convolution}.  Since $\lim_{\theta \to 0^{+}} T_{\theta} = 0$, we can fix $\theta_{0} > 0$ such that $T_{\theta} < t$ if $\theta \in (0,\theta_{0})$.  Henceforth, assume $\theta \in (0,\theta_{0})$.  

By definition of $u^{\theta}$ and $v_{\theta}$, we have
\begin{equation*}
u(x,t) - v(x,t) \leq u^{\theta}(x,t) - v_{\theta}(x,t).
\end{equation*}  
Thus, if $\delta > 0$, then Lemma \ref{L: basic_comparison} implies
\begin{equation*}
u(x,t) - v(x,t) - (2DT_{\theta} + \delta) t \leq f_{\delta}(\theta,u,v).
\end{equation*}
Sending $\delta \to 0^{+}$ and appealing to the conclusion of Lemma \ref{L: zero_behavior}, we find
\begin{equation*}
u(x,t) - v(x,t) - 2DT_{\theta} t \leq \lim_{\delta \to 0^{+}} f_{\delta}(\theta,u,v) = 4 \omega(T_{\theta}).
\end{equation*}
Finally, letting $\theta \to 0^{+}$, we conclude $u(x,t) - v(x,t) \leq 0$.  
\end{proof}  

\subsection{Properties of Sup- and Inf-convolutions}  Now we give the proofs of Proposition \ref{P: sup_inf_convolution} and Lemma \ref{L: sup_inf_convolution}.  

\begin{proof}[Proof of Proposition \ref{P: sup_inf_convolution}]  That $u^{\theta}, v_{\theta} \in \text{UC}(\mathcal{I} \times [0,T])$ follows directly from the assumptions on $u$ and $v$ and manipulation of \eqref{E: sup_convolution} and \eqref{E: inf_convolution}.  Notice that the function $t \mapsto u^{\theta}(x,t) + \frac{t^{2}}{2 \theta}$ can be written as a supremum of affine functions as follows: 
	\begin{equation*}
		u^{\theta}(x,t) + \frac{t^{2}}{2\theta} = \sup \left\{ u(x,s) - \frac{s^{2}}{2\theta} + \theta^{-1} st \, \mid \, s \in [0,T]\right\},
	\end{equation*}
Thus, that function is convex, and, in particular, $u^{\theta}$ is semi-convex with linear modulus $\theta^{-1}$, as claimed.  We show $v_{\theta}$ is semi-concave arguing similarly.    \end{proof}  

\begin{proof}[Proof of Lemma \ref{L: sup_inf_convolution}]  We will only give the details for $u^{\theta}$ since the arguments for $v_{\theta}$ follow via analogous arguments.  

First, we claim that there is a $T_{\theta} > 0$ such that if $T_{\theta} < t \leq T$ and $x \in \mathcal{I}$, then $u^{\theta}(x,t) = u(x,s) - \frac{(t - s)^{2}}{2 \theta}$ for some $s > 0$.  Indeed, for each $t \in [0,T]$, the continuity of $u$ and compactness imply we can fix an $s \in [0,T]$ for which such an equality holds.  It remains to show that $s > 0$ holds if $t$ is large enough.

By definition of $u^{\theta}$, we find
\begin{equation*}
u(x,t) \leq u^{\theta}(x,t) = u(x,s) - \frac{(t - s)^{2}}{2 \theta}.
\end{equation*}
In particular, by the definition of $\omega$, this gives
\begin{equation} \label{E: important_later}
	\frac{(t - s)^{2}}{2 \theta} \leq u(x,s) - u(x,t) \leq \omega(T).
\end{equation}  
That is,
\begin{equation*}
|t - s| \leq  \sqrt{2 \omega(T)\theta}.
\end{equation*} 
From this, we see that if there is a $\delta > 0$ such that $\sqrt{2 \omega(T)\theta} + \delta \leq t$, then $s \geq \delta$, which proves our claim.   Henceforth, set $T_{\theta} = \sqrt{2 \omega(T) \theta}$.

Note that the previous paragraph implies the sub-solution property at the initial time $T_{\theta}$.  Specifically, suppose $x \in \mathcal{I}$ and pick an $s \in [0,T]$ is such that $u^{\theta}(x,T_{\theta}) = u(x,s) - \frac{(T_{\theta}-  s)^{2}}{2\theta}$.  If $\omega$ is the modulus defined in \eqref{E: modulus}, then
\begin{align*}
u^{\theta}(x,T_{\theta}) &= u(x,s) - \frac{(T_{\theta} - s)^{2}}{2\theta} \\
		&\leq u_{0}(x) + (u(x,T_{\theta}) -u_{0}(x)) + (u(x,s) - u(x,T_{\theta})) \\
		&\leq u_{0}(x) + \omega(T_{\theta}) + \omega(|T_{\theta} - s|)
\end{align*}
Since $|T_{\theta} - s| \leq T_{\theta}$ by our previous arguments and $\omega$ is non-decreasing, we conclude $u^{\theta}(\cdot,T_{\theta}) \leq u_{0} + 2\omega(T_{\theta})$.

Next, we show that $u^{\theta}$ satisfies the necessary differential inequalities.  First, assume $t > T_{\theta}$, $x \in \mathcal{I}$, $\varphi \in C^{2,1}(\mathcal{I} \times [0,T])$, and $u^{\theta} - \varphi$ has a local maximum at $(x,t)$. 

Fix $j \in \{1,2,\dots,K\}$ such that $x \in \overline{I_{j}}$, and choose $s > 0$ such that $u^{\theta}(x,t) = u(x,s) - \frac{(t - s)^{2}}{2 \theta}$.  Since $u^{\theta} - \varphi$ has a local maximum at $(x,t)$, it follows that if $(y,r)$ is sufficiently close to $(x,s)$, then
\begin{align*}
u(y,r) - \varphi(y,r + (t - s)) &\leq u^{\theta}(y,r + (t - s)) + \frac{(t - s)^{2}}{2 \theta} - \varphi(y,r + (t - s)) \\
				&\leq u^{\theta}(x,t) - \varphi(x,t) + \frac{(t - s)^{2}}{2 \theta} \\
				&= u(x,s) - \varphi(x,t).
\end{align*}
In other words, $(y,r) \mapsto u(y,r) - \varphi(y,r + (t - s))$ has a local maximum at $(x,s)$.  Therefore, if $x \neq 0$, the sub-solution property of $u$ implies
\begin{equation*}
\varphi_{t}(x,t) + H_{j}(s,x,\varphi_{x_{j}}(x,t)) \leq 0.
\end{equation*}
Recalling from our previous computations that $|t - s| \leq T_{\theta}$, we use \eqref{As: time_bound} to obtain
\begin{equation*}
\varphi_{t}(x,t) + H_{j}(t,x,\varphi_{x_{j}}(x,t)) \leq D T_{\theta}.
\end{equation*}
This establishes the necessary differential inequality in case $x \neq 0$. 

On the other hand, if $x = 0$, the sub-solution property satisfied by $u$ at the junction implies 
\begin{equation*}
\min \left\{ \sum_{i = 1}^{K} \varphi_{x_{i}}(0,t), \varphi_{t}(0,t) + \min_{i} H_{i}(s,0,\varphi_{x_{i}}(0,t)) \right\} \leq 0.
\end{equation*}
Again, replacing $s$ with $t$ in the argument of the Hamiltonians introduces an error, leaving us with the following inequality:
\begin{equation*}
\min \left\{\sum_{i = 1}^{K} \varphi_{x_{i}}(0,t), \varphi_{t}(0,t) + \min_{i} H_{i}(t,0,\varphi_{x_{i}}(0,t)) - DT_{\theta} \right\} \leq 0.
\end{equation*}
This completes the proof that $u^{\theta}$ is a sub-solution of \eqref{E: modified_equation}.

Finally, notice that a straightforward manipulation of \eqref{E: sup_convolution} shows that $|u^{\theta}(x,t) - u^{\theta}(x,s)| \leq \theta^{-1} T |t - s|$.  Thus, $u^{\theta}$ is a viscosity sub-solution of the equation $|u_{t}| = \theta^{-1}T$ in $\mathcal{I} \times [0,T]$.  From this and what was shown earlier in the proof, we see that, for each $i \in \{1,\dots,K\}$, $u^{\theta}$ is a sub-solution of $-\theta^{-1}T + H_{i}(t,x,u_{x_{i}}) - DT_{\theta} = 0$ in $I_{i} \times (T_{\theta},T)$.  Appealing, for example, to Lemma \ref{L: critical} in Appendix \ref{A: dimensionality_reduction} and arguing as in Lemma \ref{L: uniform_Lipschitz}, we conclude that $\text{Lip}(u^{\theta}(\cdot,t))$ is bounded independently of $t \in [T_{\theta},T]$.  In particular, $u^{\theta} \in \text{Lip}(\mathcal{I} \times [T_{\theta},T])$.  
\end{proof}  

\subsection{Comparison lemmas}  This subsection is devoted to the proofs of Lemmas \ref{L: basic_comparison} and \ref{L: zero_behavior}.  

\begin{proof}[Sketch of the proof of Lemma \ref{L: basic_comparison}]  This is more-or-less classical so we will only sketch the proof.  Fix $\delta > 0$.  Formally, if $(x,t) \mapsto u^{\theta}(x,t) - v_{\theta}(x,t) - (2DT_{\theta} + \delta)t$ is maximized at a point $(x_{0},t_{0})$ with $x_{0} \in I_{i}$ for some $i$ and $t_{0} > 0$, then the equations satisfied by $u^{\theta}$ and $v_{\theta}$ imply
\begin{equation*}
	\left\{
		\begin{array}{r}
			u^{\theta}_{t}(x_{0},t_{0}) + H_{i}(t_{0},x_{0},u^{\theta}_{x_{i}}(x_{0},t_{0})) \leq DT_{\theta} \\
		v_{\theta,t}(x_{0},t_{0}) + H_{i}(t_{0},x_{0},v_{\theta,x_{i}}(x_{0},t_{0})) \geq -DT_{\theta}
		\end{array}
	\right.
\end{equation*}   
Subtracting these and using the fact that $u^{\theta}_{x_{i}}(x_{0},t_{0}) = v_{\theta,x_{i}}(x_{0},t_{0})$, we find 
\begin{equation*}
u^{\theta}_{t}(x_{0},t_{0}) - v_{\theta,t}(x_{0},t_{0}) \leq 2 D T_{\theta}.
\end{equation*}
On the other hand, since $t_{0}$ maximizes the function $t \mapsto u^{\theta}(x_{0},t) - v_{\theta}(x_{0},t) - (2DT_{\theta} + \delta)t$, we have $u^{\theta}_{t}(x_{0},t_{0}) - v_{\theta,t}(x_{0},t_{0}) \geq 2D T_{\theta} + \delta$.  Putting the two inequalities together, we conclude $2 D T_{\theta} + \delta \leq 2 D T_{\theta}$, a contradiction.  

To make the sketch rigorous, we argue by contradiction.  Assume that \eqref{E: gronwall} does not hold.  In that case, we can double variables, which is possible since the maximum occurs away from the junction, and then we use the equation to derive a contradiction.  Though we do not provide the details here, similar arguments appear below in Case 3 of the proof of Theorem 2 in Subsection \ref{S: ugly_proof}.  \end{proof} 

We now turn to the proof of Lemma \ref{L: zero_behavior}.  Here we will use a number of technical results concerning semi-convex and semi-concave functions.  First, it will be helpful to recall an important fact concerning touching a semi-convex function above by a semi-concave one.  This is covered by the next proposition.

Before we proceed, we need to define sub- and super-differentials of semi-convex and semi-concave functions, again following \cite{cannarsa sinestrari}.  If $f : [0,T] \to \mathbb{R}$ is semi-convex with linear modulus $\lambda_{f}$ and $t \in [0,T]$, we define $\partial^{-}f(t)$ to be the set of points $a \in \mathbb{R}$ such that
\begin{equation*}
f(s) \geq f(t) + a(s - t) - \frac{\lambda_{f}(s - t)^{2}}{2} \quad \text{if} \, \, s \in [0,T].
\end{equation*}
Similarly, if $g : [0,T] \to \mathbb{R}$ is semi-concave with linear modulus $\lambda_{g}$ and $t \in [0,T]$, then $\partial^{+}g(t)$ is the set of points $b \in \mathbb{R}$ such that
\begin{equation*}
g(s) \leq g(t) + b(s - t) + \frac{\lambda_{g} (s - t)^{2}}{2}.
\end{equation*}

Notice that if $t_{0} \in [0,T]$ and $f$ is as above, then $\partial^{-}f(t_{0})$ is non-empty.  In fact, if we write $\tilde{f}(t) = f(t) + \frac{\lambda t^{2}}{2}$, then a calculus exercise shows $\partial^{-}f(t_{0}) = \partial^{-}\tilde{f}(t_{0}) - \lambda t_{0}$, and it is well-known that $\partial^{-}\tilde{f}(t_{0})$, as the sub-differential of a convex function, is non-empty.  Similarly, in the semi-concave case, super-differentials are always non-empty.

\begin{prop} \label{P: derivatives_exist}  Let $\psi : [0,T] \to \mathbb{R}$ be a smooth function.  Suppose $f : [0,T] \to \mathbb{R}$ is semi-convex with modulus $\lambda_{f} >0$ and $g : [0,T] \to \mathbb{R}$ is semi-concave with modulus $\lambda_{g} > 0$.  If $f - g - \psi$ is maximized at a point $t_{0} \in (0,T)$, then $f$ and $g$ are both differentiable at $t_{0}$ and $f'(t_{0}) - g'(t_{0}) = \psi'(t_{0})$.  If $f - g$ is maximized at $T$, then the one-sided derivatives $f'(T)$ and $g'(T)$ of $f$ and $g$ both exist at $T$ and $f'(T) - g'(T) \geq \psi'(T)$.   \end{prop}  

\begin{proof}  First, consider the case when $t_{0} = T$.  Let $\tilde{f}(t) = f(t) + \frac{\lambda_{f} t^{2}}{2}$ and $\tilde{g}(t) = g(t) - \frac{\lambda_{g}t^{2}}{2}$.  Since $\tilde{f}$ is continuous and convex in $[0,T]$, the one-sided derivative $\tilde{f}'(T)$ exists.  Similarly, $\tilde{g}'(T)$ exists.  Since $f$ and $g$ differ from these functions by quadratic terms, it follows that $f'(T)$ and $g'(T)$ exist as one-sided derivatives as well.  Since $T$ is an endpoint maximum, $f'(T) - g'(T) \geq \psi'(T)$ follows.      

Now consider the case when $t_{0} \in (0,T)$.  Fix $a \in \partial^{-}f(t_{0})$ and $b \in \partial^{+} g(t_{0})$.  Since $f - g - \psi$ is maximized at $t_{0}$, it follows that, for each $t \in (0,T)$, we have
	\begin{align*}
		a(t - t_{0}) - \frac{\lambda_{f}(t - t_{0})^{2}}{2} &\leq f(t) - f(t_{0}) \\
			&\leq (g(t) - g(t_{0})) + (\psi(t) - \psi(t_{0})) \\
			&\leq b(t - t_{0}) + \frac{\lambda_{g}(t - t_{0})^{2}}{2} + \psi'(t_{0})(t - t_{0}) \\
			&\quad + \frac{\psi''(t_{0})}{2}(t - t_{0})^{2} + o(|t - t_{0}|^{2}).
	\end{align*}
Dividing by $t - t_{0}$, considering separately the cases $t > t_{0}$ and $t < t_{0}$, and sending $t \to t_{0}$, we conclude that $a = b + \psi'(t_{0})$.  This proves $\partial^{-}f(t_{0}) = \{b + \psi'(t_{0})\}$ and $\partial^{+}g(t_{0}) = \{a - \psi'(t_{0})\}$.  As in the case of convex or concave functions, the sub-differential (resp.\ super-differential) of a semi-convex (resp.\ semi-concave) function is a singleton at a point if and only if the function is differentiable at that point.  In particular, $f'(t_{0})$ and $g'(t_{0})$ exist and $f'(t_{0}) = g'(t_{0}) + \psi'(t_{0})$.    \end{proof}  

Recall from Proposition \ref{P: sup_inf_convolution} that $t \mapsto u^{\theta}(x,t)$ is semi-convex and $t \mapsto v_{\theta}(x,t)$ is semi-concave, both with linear modulus $\theta^{-1}$, no matter the choice of $x \in \mathcal{I}$.  In what follows, we will write $\partial^{-}u^{\theta}(x,t)$ and $\partial^{+} v_{\theta}(x,t)$ for the sub-differential and super-differential, respectively, of these functions with $x$ fixed.  

By Proposition \ref{P: derivatives_exist}, if $t \mapsto u^{\theta}(0,t) - v_{\theta}(0,t) - (2D T_{\theta} + \delta)t$ is maximized at a point $t \in (0,T]$, then both $u^{\theta}$ and $v_{\theta}$ are differentiable in time at $t_{0}$.  It turns out that semi-convexity and semi-concavity are strong enough properties to enable us to freeze equations \eqref{E: modified_equation} and \eqref{E: modified_equation_2} at $t = t_{0}$ and treat the associated functions $x\mapsto u^{\theta}(x,t_{0})$ and $x \mapsto v_{\theta}(x,t_{0})$ as solutions of the corresponding time-independent equations.  This is made precise in the next result.

\begin{prop} \label{P: time_freezing}  Suppose that $T_{\theta} < T$ and $t \mapsto u^{\theta}(0,t) - v_{\theta}(0,t) - (2DT_{\theta} + \delta)t$ is maximized in the interval $[T_{\theta},T]$ at the point $t_{0} \in (T_{\theta},T)$ (resp.\ $t_{0} = T$).  Let $u^{\theta}_{t}(0,t_{0})$ and $v_{\theta,t}(0,t_{0})$ be the derivatives (resp.\ one-sided derivatives) in time, which exist by Proposition \ref{P: derivatives_exist}.    For each $\zeta > 0$, there is a $\nu \in (0,t_{0})$ such that $x \mapsto u^{\theta}(x,t_{0})$ is a sub-solution of
	\begin{equation} \label{E: time_independent_modified_freeze_1}
		\left\{
			\begin{array}{r l}
				u^{\theta}_{t}(0,t_{0}) + H_{i}(t_{0},x,u^{\theta}_{x_{i}}(\cdot,t_{0})) - (D T_{\theta} + \zeta) = 0 & \text{in} \, \, I_{i}^{\nu} \\
				\sum_{i = 1}^{K} u_{x_{i}}^{\theta}(\cdot,t_{0}) = 0 & \text{on} \, \, \{0\}
			\end{array} 
		\right.
	\end{equation}
and $x \mapsto v_{\theta}(x,t_{0})$ is a super-solution of 
	\begin{equation} \label{E: time_independent_modified_freeze_2}
		\left\{
			\begin{array}{r l}
				v_{\theta,t}(0,t_{0}) + H_{i}(t_{0},x,v_{\theta,x_{i}}(\cdot,t_{0})) + D T_{\theta} + \zeta = 0 & \text{in} \, \, I_{i}^{\nu} \\
				\sum_{i = 1}^{K} v_{\theta,x_{i}}(\cdot,t_{0}) = 0 & \text{on} \, \, \{0\}
			\end{array}
		\right.
	\end{equation}
\end{prop}    

In what follows, we will use the fact that if $f : [0,T] \to \mathbb{R}$ is semi-convex with linear modulus $\lambda_{f} > 0$ and Lipschitz with $|f(t) - f(s)| \leq A|t - s|$ for all $t, s \in [0,T]$, and if $t_{0} \in (0,T]$, then $\partial^{-}f(t_{0}) \subseteq [-A,\infty)$.  In fact, $\partial^{-}f(t_{0}) \subseteq [-A,A]$ whenever $t_{0} \in (0,T)$, and $\partial^{-}f(T) = [f'(T),\infty)$ with $f'(T)$ interpreted as a one-sided derivative.  The proof is an exercise in convex analysis that we leave to the reader.  An analogous statement is true in the semi-concave context. 

\begin{proof}  We only give the details for $u^{\theta}$ since those for $v_{\theta}$ are similar.  We will prove this as an application of the dimensionality reduction lemma appearing in Appendix \ref{A: dimensionality_reduction}.  By that lemma, it is enough to show that for each $\zeta > 0$, there is a $\nu \in (0,t_{0})$ such that $u^{\theta}$ is a sub-solution of
	\begin{equation} \label{E: time_independent_modified_1}
		\left\{ 
			\begin{array}{r l}
				u_{t}^{\theta}(0,t_{0}) + H_{i}(t,x,u_{x_{i}}^{\theta}) - \tilde{C}_{\zeta} = 0 & \text{in} \, \, I_{i}^{\nu} \times (t_{0} - \nu,(t_{0} + \nu) \wedge T) \\
				\sum_{i = 1}^{K} u_{x_{i}}^{\theta} = 0 & \text{on} \, \, \{0\} \times (t_{0} - \nu, (t_{0} + \nu) \wedge T)
			\end{array} 
		\right.
	\end{equation}
where $\tilde{C}_{\zeta} = DT_{\theta} + \zeta$.

To show \eqref{E: time_independent_modified_1} holds, we will prove that for each $\zeta > 0$, there is a $\nu \in (0,t_{0})$ such that if $\max\{|x|, |t - t_{0}|\}< \nu$, $t \leq T$, and $\varphi \in C^{2,1}(\mathcal{I} \times [0,T])$ is such that $u - \varphi$ has a local maximum at $(x,t)$, then $\varphi_{t}(x,t) >  u_{t}^{\theta}(0,t_{0}) - \zeta$.  As usual in continuity statements, it suffices to show that if $(x_{n},t_{n}) \to (0,t_{0})$ and $u - \varphi_{n}$ has a local maximum at $(x_{n},t_{n})$, then $\liminf_{n \to \infty} \varphi_{n}(x_{n},t_{n}) \geq u_{t}^{\theta}(0,t_{0})$.

Indeed, given such a sequence $(x_{n},t_{n},\varphi_{n})$, a straightforward convex analysis exercise shows that $a_{n} : = \varphi_{n,t}(x_{n},t_{n}) \in \partial^{-} u^{\theta}(x_{n},t_{n})$ for each $n \in \mathbb{N}$, and, thus,
	\begin{equation} \label{E: sub_differential_computation_1}
		u^{\theta}(x_{n},t) \geq u^{\theta}(x_{n},t_{n}) + a_{n}(t - t_{n}) - \frac{(t - t_{n})^{2}}{2 \theta} \quad \text{if} \, \, t \in [T_{\theta},T].
	\end{equation}
The remark preceding this proof and the arguments in the proof of Lemma \ref{L: sup_inf_convolution} show that $a_{n} \in \partial^{-} u^{\theta}(x_{n},t_{n}) \subseteq [-\frac{T}{\theta},\infty)$ for all $n$.  Thus, either we have $\liminf_{n \to \infty} a_{n} = \infty$, in which case there is nothing left to show, or else $\liminf_{n\to \infty} a_{n}$ is finite.  In the latter case, let us assume by passing to a subsequence that $a = \lim_{n \to \infty} a_{n}$ exists.  In the limit $n \to \infty$, \eqref{E: sub_differential_computation_1} leads to
	\begin{equation} \label{E: sub_differential_computation_2}
		u^{\theta}(0,t) \geq u^{\theta}(0,t_{0}) + a(t - t_{0}) - \frac{(t - t_{0})^{2}}{2 \theta} \quad \text{if} \, \, t \in [T_{\theta},T].
	\end{equation}
In particular, $a \in \partial^{-}u^{\theta}(0,t_{0})$.  Since the (possibly one-sided) differentiability of $u^{\theta}(0,\cdot)$ at $t_{0}$ implies $\partial^{-} u^{\theta}(0,t_{0}) \subseteq [u_{t}^{\theta}(0,t_{0}),\infty)$, we conclude that $a \geq u_{t}^{\theta}(0,t_{0})$.

From the results of the previous two paragraphs, we see that for a given $\zeta > 0$, there is a $\nu \in (0,t_{0})$ such that if $u^{\theta} - \varphi$ has a maximum at the point $(x_{1},t_{1}) \in I_{i}^{\nu} \times (t_{0} - \nu, (t_{0} + \nu) \wedge T]$, then $\varphi_{t}(x_{1},t_{1}) > u_{t}^{\theta}(0,t_{0}) -\zeta$.  In particular, if $x_{1} \neq 0$, then \eqref{E: modified_equation} implies
	\begin{equation*}
		u_{t}^{\theta}(0,t_{0}) - \zeta + H_{i}(t_{1},x_{1},\varphi_{x_{i}}(x_{1},t_{1})) \leq DT_{\theta},
	\end{equation*}
while the case $(x_{1},t_{1}) \in \{0\} \times (t_{0} - \nu, (t_{0} + \nu)\wedge T)$ yields
	\begin{equation*}
		\min \left\{ \sum_{i = 1}^{K} \varphi_{x_{i}}(x_{1},t_{1}), u_{t}^{\theta}(0,t_{0}) -\tilde{C}_{\zeta} + \min_{i} H_{i}(t_{1},x_{1},\varphi_{x_{i}}(x_{1},t_{1})) \right\} \leq 0.
	\end{equation*}
Therefore, $u^{\theta}$ solves \eqref{E: time_independent_modified_1} as claimed, and $x \mapsto u^{\theta}(x,t_{0})$ satisfies \eqref{E: time_independent_modified_freeze_1} by Lemma \ref{L: critical} of Appendix \ref{A: dimensionality_reduction}.  \end{proof}  

Now we have all the ingredients ready to prove Lemma \ref{L: zero_behavior}.  If the lemma did not hold, then we would find a maximum $t_{0} \in (T_{\theta},T]$ of $t \mapsto u^{\theta}(0,t) - v_{\theta}(0,t) - (2DT_{\theta} + \delta)t$.  Propositions \ref{P: derivatives_exist} and \ref{P: time_freezing} would then allow us to freeze the equations at $t = t_{0}$, and then we could argue using \cite[Lemma 3.1]{time-dependent} as in Section \ref{S: stationary_comparison} to derive a contradiction.   

\begin{proof}[Proof of Lemma \ref{L: zero_behavior}]  Assume that the conclusion of the lemma does not hold, that is,
\begin{equation*}
\max \left\{ u^{\theta}(0,t) - v_{\theta}(0,t) - 2D T_{\theta}t \, \mid \, t \in [T_{\theta},T] \right\} > 4\omega(T_{\theta}).
\end{equation*}  
By Lemma \ref{L: basic_comparison}, it follows that there is a $t_{0} \in [T_{\theta},T]$ and a small $\delta > 0$  such that the function $(x,t) \mapsto u^{\theta}(x,t) - v_{\theta}(x,t) - (2D T_{\theta} + \delta) t$ defined in $\mathcal{I} \times [T_{\theta},T]$ is maximized at $(0,t_{0})$.   

First, consider the case when $t_{0} = T_{\theta}$.  By Lemma \ref{L: sup_inf_convolution}, we have
\begin{align*}
4 \omega(T_{\theta}) &< \max \left\{ u^{\theta}(0,t) - v_{\theta}(0,t) - (2D T_{\theta} + \delta)t \, \mid \, t \in [T_{\theta},T] \right\} \\
	&= u^{\theta}(0,T_{\theta}) - v_{\theta}(0,T_{\theta}) - (2D T_{\theta} + \delta)T_{\theta} \\
	&< 4 \omega(T_{\theta}),
\end{align*}
which is a contradiction.  Therefore, in what follows, we can assume $t_{0} > T_{\theta}$ holds.

Fix $\zeta \in (0,\frac{\delta}{2})$.  By Proposition \ref{P: time_freezing}, there is a $\nu \in (0,t_{0})$ such that $x \mapsto u^{\theta}(x,t_{0})$ satisfies \eqref{E: time_independent_modified_freeze_1}
and $x \mapsto v_{\theta}(x,t_{0})$ satisfies \eqref{E: time_independent_modified_freeze_2}.  Moreover, $x \mapsto u^{\theta}(x,t_{0}) - v_{\theta}(x,t_{0})$ is maximized at $0$.  

Let $a = u^{\theta}_{t}(0,t_{0}) - DT_{\theta} - \zeta$ and $b = v_{\theta,t}(0,t_{0}) + DT_{\theta} + \zeta$, where $u^{\theta}_{t}(0,t_{0})$ and $v_{\theta,t}(0,t_{0})$ are interpreted as one-sided derivatives if $t_{0} = T$.  For each $i \in \{1,2,\dots,K\}$, define $p_{i}$ and $q_{i}$ by 
\begin{equation*}
p_{i} = \limsup_{I_{i} \ni x \to 0} \frac{u^{\theta}(x,t_{0}) - u^{\theta}(0,t_{0})}{x}, \quad q_{i} = \liminf_{I_{i} \ni x \to 0} \frac{v_{\theta}(x,t_{0}) - v_{\theta}(0,t_{0})}{x}.
\end{equation*}  
Notice that $p_{i} \geq q_{i}$ for all $i$ since $0$ is the maximum of $x \mapsto u^{\theta}(x,t_{0}) - v_{\theta}(x,t_{0})$.  

The arguments of the two previous paragraphs show Theorem \ref{T: comparison} applies, and, thus, the following conditions are satisfied:
\begin{itemize}
\item[(i)] If $i \in \{1,2,\dots,K\}$, then $p_{i} \geq q_{i}$ and $a + H_{i}(t_{0},0,p_{i}) \leq 0$.  Furthermore, if $q_{i} > - \infty$, then $b + H_{i}(t_{0},0,q_{i}) \geq 0$.
\item[(ii)] If $(\tilde{p}_{1},\dots,\tilde{p}_{K}) \in \mathbb{R}^{K}$ and $\tilde{p}_{i} \leq p_{i}$ for each $i$, then 
\begin{equation*}
\max \left\{ \sum_{i = 1}^{K} \tilde{p}_{i}, a + \max_{i} H_{i}(t_{0},0,\tilde{p}_{i}) \right\} \leq 0.
\end{equation*}
\item[(iii)] If $(\tilde{q}_{1},\dots,\tilde{q}_{K}) \in \mathbb{R}^{K}$ and $\tilde{q}_{i} \geq q_{i}$ for each $i$, then 
\begin{equation*}
\min \left\{ \sum_{i = 1}^{K} \tilde{q}_{i}, b + \min_{i} H_{i}(t_{0},0,\tilde{q}_{i}) \right\} \geq 0.
\end{equation*}
\end{itemize}
Replacing the vector $(q_{1},\dots,q_{K})$ by a vector $(q_{1}^{*},\dots,q_{K}^{*})$ exactly as in the proof of Theorem \ref{T: existence_uniqueness_inviscid_stationary}, the numbers $a$ and $b$ and vectors $(p_{1},\dots,p_{K})$ and $(q_{1}^{*},\dots,q_{K}^{*})$ satisfy the conditions of \cite[Lemma 3.1]{time-dependent} and, therefore, $a \leq b$. 

Since $a \leq b$, we find 
\begin{align*}
2DT_{\theta} + \delta &\leq u_{t}^{\theta}(0,t_{0}) - v_{\theta,t}(0,t_{0}) \\
	&= (a - b) + 2 DT_{\theta} + 2 \zeta \\
	&\leq 2 DT_{\theta} + 2 \zeta.
\end{align*} 
As a consequence, we deduce $\delta \leq 2 \zeta$, which contradicts the choice of $\zeta$.    
\end{proof}  

\subsection{Proof of Theorem \ref{T: viscous_comparison_time}} 

\begin{proof}[Proof of Theorem \ref{T: viscous_comparison_time}]  If $u$ and $v$ are as in the statement of the theorem, then we pass to the sup-convolution $u^{\theta}$ of $u$ and inf-convolution $v_{\theta}$ of $v$ as before.  Proposition \ref{P: sup_inf_convolution} and Lemmas \ref{L: sup_inf_convolution} and \ref{L: basic_comparison} remain true with no alterations in the proofs.  (Of course, a second-order term appears in the equations satisfied by $u^{\theta}$ and $v_{\theta}$ in this case.)  Although Lemma \ref{L: zero_behavior} is still true, the proof needs to be modified.  At the end of the proof, we show that $a \leq b$ using the idea of the proof of Theorem \ref{T: existence_uniqueness_viscous_stationary}  (i.e.\ ``bending" the sub-solution) instead of Theorem \ref{T: comparison}.      \end{proof}  


\section{Time-Independent Problem: Vanishing Viscosity Limit}  \label{S: stationary_problem}

As explained in the introduction, the error estimate for the vanishing viscosity approximation of \eqref{E: stationary} is obtained by doubling variables and using a test function that forces the variable associated with $u^{\epsilon}$ away from the junction.  This approach combined with an auxiliary lemma on the continuity of the junction condition suffices to carry through a modified version of the classical proof.  

\subsection{Preliminaries}  




In the remainder of the section, $u$ denotes the bounded solution of \eqref{E: stationary} and $u^{\epsilon}$, the bounded solution of \eqref{E: viscous_stationary} for a given $\epsilon > 0$, both of which are unique by Theorems \ref{T: existence_uniqueness_inviscid_stationary} and \ref{T: existence_uniqueness_viscous_stationary} and exist by Perron's Method.  

By \eqref{As: sub-sup}, the constant functions $u_{\text{sub}}(x) = -M$ and $u_{\text{super}}(x) = M$ are bounded sub- and super-solutions, respectively, of \eqref{E: stationary} and \eqref{E: viscous_stationary}.  Therefore, Theorems \ref{T: existence_uniqueness_inviscid_stationary} and \ref{T: existence_uniqueness_viscous_stationary} imply that
	\begin{equation} \label{E: a priori bound}
		\sup \left\{|u^{\epsilon}(x)| \vee |u(x)| \, \mid \, x \in \mathcal{I}, \, \, \epsilon > 0\right\} \leq M.
	\end{equation}

By Lemma \ref{L: uniform_Lipschitz} and \eqref{E: a priori bound}, we can henceforth fix an $L> 0$, independent of $\epsilon > 0$, such that 
	\begin{equation} \label{E: stationary_viscous_lipschitz}
		\text{Lip}(u) \vee \text{Lip}(u^{\epsilon}) \leq L.
	\end{equation}
Later, we will see that we can specifically define $L = \sup \left\{ \text{Lip}(u^{\epsilon}) \, \mid \, \epsilon > 0\right\}$ in agreement with the statement of Theorem \ref{T: rate_stationary_visc}.  See Remark \ref{R: silly} at the end of the proof.  

\subsection{The proof of Theorem \ref{T: rate_stationary_visc}} \label{S: stationary_proof}  In what follows, we will only prove $\sup(u^{\epsilon} - u) \leq C \epsilon^{\frac{1}{2}}$.  The lower bound follows by interchanging the roles of $u$ and $u^{\epsilon}$.  

The following lemma, which reformulates the definition of \eqref{E: stationary}, is instrumental in the proof of Theorem \ref{T: rate_stationary_visc}.  We present its proof in Appendix \ref{A: reformulation}.  In the statement, recall from Subsection \ref{S: notation} that we use the notation $a^{+} = \max\{a,0\}$ and $a^{-} = - \min\{a,0\}$ .

\begin{lemma} \label{L: neumann_type_lemma} A function $u \in C(\mathcal{I})$ solves \eqref{E: stationary} if and only if for each $i \in \{1,2,\dots,K\}$, $u$ solves
\begin{equation*}
u + H_{i}(x,u_{x_{i}}) = 0 \quad \text{in} \, \, I_{i}
\end{equation*}
and, in addition, $u$ satisfies the following two differential inequalities at $0$:
\begin{equation}
	u(0) + \min_{i} \min_{\tilde{\theta} \in [0,1]} H_{i} \left(0, u_{x_{i}}(0) + \tilde{\theta} \left(\sum_{j = 1}^{K} u_{x_{j}}(0) \right)^{-} \right) \leq 0 \label{E: crucial sub} 
\end{equation}
and
\begin{equation}
	u(0) + \max_{i} \max_{\tilde{\theta} \in [0,1]} H_{i} \left(0, u_{x_{i}}(0) - \tilde{\theta} \left(\sum_{j =1}^{K} u_{x_{j}}(0) \right)^{+} \right) \geq 0. \label{E: crucial sup}
\end{equation}
\end{lemma}  

The exact meaning of the differential inequalities in the lemma is provided in Appendix \ref{A: reformulation}.  

It follows from the lemma that if $u - \varphi$ has a minimum at $0$ and $\sum_{i = 1}^{K} \varphi_{x_{i}}(0) \leq \delta$, then there is an index $j$ such that
\begin{equation*}
	u(0) + H_{j}(0,\varphi_{x_{j}}(0)) \geq  -\omega(\delta),
\end{equation*}
where $\omega$ is a modulus of continuity for $H_{j}$ in the $\delta$-neighborhood of $\varphi_{x_{j}}(0)$.  In this sense, Lemma \ref{L: neumann_type_lemma} is a continuity property of the junction condition.

In what follows, we would like to use a variable doubling approach by studying the solutions near a maximum point $(\bar{x}_{i},\bar{y}_{i})$ of the function $\Phi_{i}$ on $\overline{I_{i}} \times \overline{I_{i}}$ given by
\begin{equation*}
	\Phi_{i}(x,y) = u^{\epsilon}(x) - u(y) - \frac{(x - y)^{2}}{2  \sqrt{\epsilon}}.
\end{equation*}
The goal is to write, as in the classical proof,
\begin{equation}
	u^{\epsilon}(\bar{x}_{i}) - \sqrt{\epsilon} + H_{i} \left(\bar{x}_{i},\frac{\bar{x}_{i} - \bar{y}_{i}}{\sqrt{\epsilon}} \right) \leq 0, \label{E: problem1}
\end{equation}
and
\begin{equation}
	u(\bar{y}_{i}) + H_{i} \left(\bar{y}_{i},\frac{\bar{x}_{i} - \bar{y}_{i}}{\sqrt{\epsilon}}\right) \geq 0. \label{E: problem2}
\end{equation}
However, there are three problems with this approach.  First, $u^{\epsilon}$ may not satisfy \eqref{E: problem1} if $\bar{x}_{i} = 0$.  Secondly, the Kirchoff condition implies that if $\bar{y}_{i} = 0$ for each $i$, then \eqref{E: problem2} may only hold for a subset of the indices $i$.  As indicated in the introduction, the first issue can be remedied by tilting the test function.  Moreover, if $\bar{x}_{i} < 0$ independently of $i$, then Lemma \ref{L: neumann_type_lemma} implies there is a $j$ such that \eqref{E: problem1} and \eqref{E: problem2} both approximately hold with $i = j$, and, thus, we can estimate $u^{\epsilon}(\bar{x}_{j}) - u(\bar{y}_{j})$.  In other words, the second issue is corrected if we not only guarantee that $\bar{x}_{i} < 0$ for some $i$, but even that $\bar{x}_{i} < 0$ for all $i$.  Lemma \ref{L: magic} below accomplishes this.  Finally, the third issue to address is the unboundedness of the domain.  Since we are working with infinite rays, the function $\Phi_{i}$ may not attain its supremum.  Therefore, as is customary in the theory of viscosity solutions, we will add penalization terms to correct this.

Now that we have summarized the difficulties involved in analyzing the error near the junction, we provide the details.  

Henceforth we let $C_{i}(\delta) = u^{\epsilon}_{x_{i}}(0) + \delta$, which is well-defined by Lemma \ref{L: differentiability}.  Given $\delta,\alpha \in (0,1)$, we define $\Phi_{i,\delta,\alpha} : \overline{I_{i}} \times \overline{I_{i}} \to \mathbb{R}$ by
\begin{equation*}
	\Phi_{i,\delta,\alpha}(x,y) = u^{\epsilon}(x) - u(y) - \frac{(x - y)^{2}}{2 \sqrt{\epsilon}} - C_{i}(\delta)(x - y) - \alpha y^{2}.
\end{equation*}
The key lemma that allows us to control the behavior at the junction is stated next.  

\begin{lemma}  \label{L: magic} If $(x_{i}(\delta,\alpha),y_{i}(\delta,\alpha))$ is a global maximum of $\Phi_{i,\delta,\alpha}$, then $x_{i}(\delta,\alpha) < 0$.\end{lemma}

\begin{proof} Since $x \mapsto \Phi_{i,\delta,\alpha}(x,y_{i}(\delta,\alpha))$ is maximized at $x_{i}(\delta,\alpha)$, we observe that 
\begin{equation*}
\frac{x_{i}(\delta,\alpha) - y_{i}(\delta,\alpha)}{\sqrt{\epsilon}} + C_{i}(\delta) \leq u^{\epsilon}_{x_{i}}(x_{i}(\delta,\alpha)).
\end{equation*} 
Thus, the inequality $-\frac{y_{i}(\delta,\alpha)}{\sqrt{\epsilon}} + C_{i}(\delta) > u^{\epsilon}_{x_{i}}(0)$ implies $x_{i}(\delta,\alpha) < 0$. \end{proof}

We are now in a position to complete the error estimate.  Lemma \ref{L: magic} addresses the problems mentioned earlier.  Since $x_{i}(\delta,\alpha) < 0$ independently of $i$, there is always an $i$ such that equations \eqref{E: problem1} and \eqref{E: problem2} hold with an error term introduced by Lemma \ref{L: neumann_type_lemma}.  Note that the error term is small since $\sum_{i = 1}^{K} C_{i}(\delta) = K \delta$ follows from Lemma \ref{L: classical_kirchoff}.  The remainder of the error estimate therefore follows from a combination of classical viscosity arguments and new but elementary ideas.

\begin{proof}[Proof of Theorem 1]  As before, for each index $i$, we let $(x_{i}(\delta,\alpha),y_{i}(\delta,\alpha))$ denote a global maximum of $\Phi_{i,\delta,\alpha}$, and we fix $k \in \{1,2,\dots,K\}$ such that
\begin{equation*}
\sup \left\{u^{\epsilon}(x) - u(x) \, \mid \, x \in \mathcal{I} \right\} = \sup \left\{u^{\epsilon}(x) - u(x) \, \mid \, x \in \overline{I_{k}} \right\}.
\end{equation*}
Let us begin by estimating $|y_{i}(\delta,\alpha)|$, and $|x_{i}(\delta,\alpha) - y_{i}(\delta,\alpha)|$.  We will write $x_{i} = x_{i}(\delta,\alpha)$ and $y_{i} = y_{i}(\delta,\alpha)$ when there is no risk of confusion.  

Given $i$, $\alpha$, $\delta$, we can write $\Phi_{i,\delta,\alpha}(0,0) \leq \Phi_{i,\delta,\alpha}(x_{i}(\delta,\alpha),y_{i}(\delta,\alpha))$ and, thus, by \eqref{E: a priori bound} and \eqref{E: stationary_viscous_lipschitz}, we obtain
\begin{equation*}
\alpha y_{i}^{2} + \frac{(x_{i}- y_{i})^{2}}{2 \sqrt{\epsilon}} \leq 4M + (L + 1)|x_{i} - y_{i}|.
\end{equation*}  
Appealing to Young's inequality, we find there is a $C > 0$ independent of $(\epsilon,\alpha)$ such that
\begin{equation*}
\alpha y_{i}(\delta,\alpha)^{2} + \frac{(x_{i}(\delta,\alpha) - y_{i}(\delta,\alpha))^{2}}{2 \sqrt{\epsilon}} \leq C.
\end{equation*}
In particular, this shows $x_{i}(\delta,\alpha)$ is close to the junction if $y_{i}(\delta,\alpha) = 0$ and, moreover, uniformly with respect to $\epsilon$, 
\begin{equation*}
\alpha |y_{i}(\delta,\alpha)| \to 0 \quad \text{as} \, \, \alpha \to 0^{+}.
\end{equation*} 

In view of Lemma \ref{L: magic}, there are only two cases to consider, namely, (i) there is a $j$ and a subsequence $(\delta_{n},\alpha_{n}) \to (0,0)$ such that $y_{j}(\delta_{n},\alpha_{n}) < 0$ for all $n$ and (ii) $y_{i}(\delta,\alpha) = 0$ independently of $i$ whenever $\delta + \alpha$ is sufficiently small.  Case (i) can be addressed using classical arguments with a minor twist.  We postpone it for now and consider instead case (ii).

As the estimates above show, in case (ii), $x_{i}(\delta,\alpha)$ remains bounded independently of $\alpha,\epsilon$.  Therefore, provided $\delta$ is sufficiently small, we can send $\alpha \to 0^{+}$ and appeal to compactness and the arguments in Lemma \ref{L: magic} to find a limit $x_{i}(\delta) < 0$ such that $(x_{i}(\delta),0)$ maximizes the function 
\begin{equation*}
\Phi_{i,\delta}(x,y) = u^{\epsilon}(x) - u(y) - \frac{(x - y)^{2}}{2 \sqrt{\epsilon}} - C_{i}(\delta)(x - y).
\end{equation*}

It remains to use the equations to bound $u^{\epsilon}(x_{i}(\delta)) - u(0)$.  However, as already remarked in the introduction, it may be necessary to transfer bounds obtained in one edge of the network to the other edges.  Thus, the proof proceeds in two steps.  

In the first step, we observe that since the function 
\begin{equation*}
y \mapsto u(y) - u^{\epsilon}(x_{i}(\delta)) + \frac{(y - x_{i}(\delta))^{2}}{2 \sqrt{\epsilon}} - C_{i}(\delta)(y - x_{i}(\delta))
\end{equation*}
has a minimum at $0$ independently of $i$, Lemma \ref{L: neumann_type_lemma} implies that there is a $j$ and a $\tilde{\theta} \in [0,1]$ such that
\begin{equation*}
u(0) + H_{j} \left(0,C_{j}(\delta) + \frac{x_{j}(\delta)}{\sqrt{\epsilon}} - \tilde{\theta} F(\delta)^{+}\right) \geq 0, 
\end{equation*}  
where 
\begin{equation*}
F(\delta) = \sum_{i = 1}^{K} \left(C_{i}(\delta) + \frac{x_{i}(\delta)}{\sqrt{\epsilon}}\right).
\end{equation*}  It follows from Lemma \ref{L: classical_kirchoff} that $F(\delta)^{+} < K\delta$.  Moreover, since $x \mapsto \Phi_{j,\delta}(x,0)$ is maximized at $x_{j}(\delta)$ for each $j$, the bound $\text{Lip}(u^{\epsilon}) \leq L$ implies $|C_{j}(\delta) + \frac{x_{j}(\delta)}{\sqrt{\epsilon}}| \leq L$ independently of $(j,\delta,\epsilon)$.  Thus, appealing to \eqref{As: continuity}, we find a modulus $\omega$ such that
	\begin{equation} \label{E: reviewer_eq_1}
		u(0) + H_{j}\left(0,C_{j}(\delta) + \frac{x_{j}(\delta)}{\sqrt{\epsilon}}\right) \geq - \omega(K\delta).
	\end{equation} 
At the same time, since $x \mapsto \Phi_{j,\delta}(x,0)$ is maximized at $x_{j}(\delta)$, the equation for $u^{\epsilon}$ yields
	\begin{equation} \label{E: reviewer_eq_2}
		u^{\epsilon}(x_{j}(\delta)) - \sqrt{\epsilon} + H_{j} \left(x_{j}(\delta),C_{j}(\delta) + \frac{x_{j}(\delta)}{\sqrt{\epsilon}}\right) \leq 0.
	\end{equation}
Finally, subtracting \eqref{E: reviewer_eq_1} and \eqref{E: reviewer_eq_2} and appealing again to \eqref{As: continuity}, we find
\begin{equation} \label{E: step1}
u^{\epsilon}(x_{j}(\delta)) - u(0) \leq \sqrt{\epsilon} + \omega(K \delta) + C |x_{j}(\delta)|,
\end{equation} 
where $C$ depends only on the modulus of continuity of $H_{j}$ in $\overline{I_{j}} \times [-L,L]$.

In the second step of the proof, we use \eqref{E: step1}, which a priori only provides an error estimate in $\overline{I_{j}}$, to obtain a global error estimate.  This can be done using the fact that $x_{j}(\delta)$ is close to zero and $u$ is uniformly Lipschitz.  

Note that we previously established that $|x_{j}(\delta)| = O(\epsilon^{\frac{1}{4}})$.  However, since $u^{\epsilon}$ is Lipschitz, the inequality $\Phi_{j,\delta}(0,0) \leq \Phi_{j,\delta}(x_{j}(\delta),0)$ implies $|x_{j}(\delta)| = O(\epsilon^{\frac{1}{2}})$.  Therefore, we can improve \eqref{E: step1} to 
\begin{equation*}
u^{\epsilon}(0) - u(0) \leq C\sqrt{\epsilon} + \omega(K \delta).
\end{equation*} 
Since $0$ is in each edge of the network, the same reasoning yields
\begin{equation*}
u^{\epsilon}(x_{i}(\delta)) - u(0) \leq C \sqrt{\epsilon} + \omega(K \delta) \quad \text{if} \, \, i \in \{1,2,\dots,K\}.
\end{equation*} 
Finally, from the inequality $\Phi_{i,\delta}(x,x) \leq \Phi_{i,\delta}(x_{i}(\delta),0)$, we obtain
\begin{equation*}
u^{\epsilon}(x) - u(x) \leq C \sqrt{\epsilon} + \omega(K \delta).
\end{equation*}
The proof of the upper bound in case (ii) thus follows after sending $\delta \to 0^{+}$.  

Finally, we address case (i), when there is a $j$ such that $y_{j}(\delta_{n},\alpha_{n}) < 0$ for some sequence $(\delta_{n},\alpha_{n}) \to (0,0)$.  If $j = k$, then the standard proof works here (cf.\ \cite{old_paper}).  Therefore, it's only necessary to analyze what happens when $j \neq k$.  In particular, we can assume $y_{k}(\delta_{n},\alpha_{n}) = 0$ for all $n$.  As we already saw in case (ii) above, we can leverage the Lipschitz continuity of $u^{\epsilon}$ to find $|x_{k}(\delta_{n},\alpha_{n})| = O(\sqrt{\epsilon})$.  Thus,
\begin{equation} \label{E: tedious1}
u^{\epsilon}(x_{k}(\delta_{n},\alpha_{n})) - u(0) \leq C \sqrt{\epsilon} + u^{\epsilon}(0) - u(0).
\end{equation}
By classical arguments restricted to the edge $I_{j}$, we find, after sending $n \to \infty$,
\begin{equation} \label{E: tedious2}
u^{\epsilon}(0) - u(0) \leq C\sqrt{\epsilon}.
\end{equation} 
It follows from \eqref{E: tedious1} and \eqref{E: tedious2} and the inequality $\Phi_{k,\delta,\alpha}(x,x) \leq \Phi_{k,\delta,\alpha}(x_{k},y_{k})$ that
\begin{equation*}
\sup \left\{u^{\epsilon}(x) - u(x) \, \mid \, x \in \overline{I_{k}}\right\} \leq C \sqrt{\epsilon}.
\end{equation*}
In view of the choice of $k$, this completes the proof in case (i).  
 \end{proof}  
 
 \begin{remark} \label{R: silly}  Since $u^{\epsilon} \to u$ pointwise in $\mathcal{I}$, it follows that 
 	\begin{equation*}
		\text{Lip}(u) \leq \sup \left\{ \text{Lip}(u^{\epsilon}) \, \mid \, \epsilon > 0\right\}.
	\end{equation*}  
Therefore, in \eqref{E: stationary_viscous_lipschitz}, we could have defined $L : = \sup \left\{ \text{Lip}(u^{\epsilon}) \, \mid \, \epsilon > 0\right\}$.  Knowing this, we can redo the proof using this specific value of $L$.  This justifies our claims concerning the dependence of the constant in Theorem \ref{T: rate_stationary_visc}.  \end{remark}
 
 \section{Time-Independent Problem: Finite-Difference Approximation} \label{S: stationary_diff}  
 
 The arguments in the previous section extend in a straightforward manner to monotone finite difference schemes approximating solutions of \eqref{E: stationary}.  We describe the schemes in question next.
 
 To simplify the construction, we make here and in Section \ref{S: cauchy_problem_diff} the assumption that the Hamiltonians are in separated form, that is,
 \begin{equation} \label{As: standard}
 H_{i}(x,p) = H_{i}(p) - f_{i}(x).
 \end{equation}
 The general case follows from minor technical modifications.  
 
 \subsection{Preliminaries} \label{S: stationary_diff_explanation}  Given scales $\Delta x$, we discretize each edge $I_{i}$ as $J_{i} = \{0,1,2,\dots\}$ with the point $m \in J_{i}$ corresponding to $- m \Delta x \in I_{i}$.  The discretized junction is the union $\mathcal{J} := \bigcup_{i = 1}^{K} J_{i}$ glued at $0$.  The finite difference scheme generates a function $U : \mathcal{J} \to \mathbb{R}$ satisfying the difference equations
\begin{equation} \label{E: difference_equation}
\left\{ \begin{array}{r l}
	U(m) + F_{i}(D^{+}U(m),D^{-}U(m)) = f_{i}(-m \Delta x) & \quad \text{if} \, \, m \in J_{i} \setminus \{0\} \\
	U(0) = \frac{1}{K} \sum_{i = 1}^{K} U(1_{i})
\end{array} \right.
\end{equation}
Here we denote by $1_{i}$ the point $1$ in $J_{i}$ and the operators $D^{+},D^{-},F_{i}$ are given by
\begin{equation} \label{E: diff} 
D^{+}U(m) = \frac{U(m - 1) - U(m)}{\Delta x}, \quad D^{-}U(m) = \frac{U(m) - U(m + 1)}{\Delta x}, 
\end{equation}
and
\begin{equation} \label{E: F_op}
F_{i}(p_{1},p_{2}) = - \frac{\epsilon}{\Delta x} \left(p_{1} - p_{2} \right) + G_{i}\left(p_{1},p_{2} \right). 
\end{equation}
The numerical Hamiltonians $G_{1},\dots,G_{K}$ approximate the Hamiltonians and $\epsilon > 0$ is a parameter.  For each $i$, we impose the following assumptions:
\begin{equation}
G_{i} : \mathbb{R} \times \mathbb{R} \to \mathbb{R} \quad \text{is uniformly Lipschitz continuous} \label{As: num_Ham_lip} 
\end{equation}
and there is an $L_{c} > 0$ such that
\begin{equation}
G_{i}(p,p) = H_{i}(p) \quad \text{if} \, \, p \in [-L_{c},L_{c}], \, \, i \in \{1,2,\dots,K\}. \label{As: num_Ham_consistent}
\end{equation}

Following \cite{old_paper} and \cite{numerical}, we constrain the artificial viscosity $\epsilon$ in order to ensure that the scheme satisfies the classical Courant-Friedrichs-Lewy (CFL) condition, which ensures that the scheme is monotone.  The details are discussed below.

In addition to well-posedness of the scheme, we also prove that the solution is uniformly Lipschitz with respect to the approximation parameters.  We say that a function $V : \mathcal{J} \to \mathbb{R}$ is uniformly Lipschitz continuous if 
\begin{equation*}
\text{Lip}(V) := \sup \left\{ |V(m + 1) - V(m)| \, \mid \, m \in J_{i}, i \in \{1,\dots,K\} \right\} < \infty.
\end{equation*}
The result concerning well-posedness and existence of Lipschitz solutions is stated next.  

In what follows, we define $\{L_{G}^{(1)},\dots,L_{G}^{(K)}\}$ and $L_{G}$ by 
	\begin{align} 
		L_{G}^{(i)} &= \sup \left\{ \frac{|G_{i}(p_{1},p_{2}) - G_{i}(q_{1},q_{2})|}{|p_{1} - q_{1}| + |p_{2} - q_{2}|} \, \mid \, (p_{1},p_{2}),(q_{1},q_{2}) \in \mathbb{R}^{2} \right\}, \\
		L_{G} &= \max\{L_{G}^{(1)},\dots,L_{G}^{(K)}\}.\label{As: LG_constant}
	\end{align}

\begin{theorem} \label{T: stationary_scheme}  There is a constant $\overline{L}_{c} > 0$, independent of $\epsilon$ and $\Delta x$, such that if $\epsilon \geq 2L_{G} \Delta x$ and $L_{c} \geq \overline{L}_{c}$, then the difference equation \eqref{E: difference_equation} is monotone and has a unique bounded solution $U$ satisfying $\sup \left\{|U(m)| \, \mid \, m \in \mathcal{J} \right\} \leq M$ and $\text{Lip}(U) \leq \overline{L}_{c} \Delta x$.  \end{theorem}  

A precise definition of ``monotone," inspired by \cite{numerical}, is stated in the next subsection.

%
%
%
%

In order to establish Theorem \ref{T: rate_stationary_diff}, we fix a constant $L_{2} > 0$ such that
\begin{equation} \label{As: CFL_stationary}
2L_{G} \Delta x \leq \epsilon \leq L_{2} \Delta x.
\end{equation}
The lower bound $2 L_{G} \Delta x$ guarantees that the scheme is monotone and the upper bound ensures that the discretization errors have the right order.

To guarantee that the scheme converges, the proof requires that $L_{c}$ is sufficiently large.  We will assume the following lower bound on $L_{c}$ is satisfied:
	\begin{equation} \label{As: cut_off_bound}
		L_{c} \geq \overline{L}_{c} + 1.
	\end{equation}
This ensures that when $\varphi$ is one of the test functions used in the error analysis and $\epsilon$ is sufficiently small, we have $G_{i}(\varphi_{x_{i}}(x),\varphi_{x_{i}}(x)) = H_{i}(\varphi_{x_{i}}(x))$.  Note that a similar restriction was used in \cite{old_paper}.

\begin{remark} \label{R: examples_Ham} Let us give some examples of numerical Hamiltonians satisfying \eqref{As: num_Ham_lip} and \eqref{As: num_Ham_consistent}.  First, taking advantage of \eqref{As: coercive}, we can use
	\begin{equation*}
		G_{i}(p_{1},p_{2}) = H_{i}\left(\frac{p_{1} + p_{2}}{2}\right) \wedge R_{c},
	\end{equation*}
where $R_{c}$ is chosen so large that \eqref{As: num_Ham_consistent} holds.  Alternatively, we can avoid using \eqref{As: coercive} by defining $G_{i}$ by 
	\begin{equation*}
		G_{i}(p_{1},p_{2}) = H_{i} \left(\frac{((p_{1} \vee -L_{c})\wedge L_{c}) + ((p_{2} \vee -L_{c})\wedge L_{c})}{2} \right). 
	\end{equation*}
	
\end{remark}  

\subsection{Well-posedness of the scheme} \label{S: stationary_fd_explanation}  In this subsection, we prove Theorem \ref{T: stationary_scheme}.

In preparation for the proof, we begin with some terminology.  A function $V : \mathcal{J} \to \mathbb{R}$ is a sub-solution of the the scheme \eqref{E: difference_equation} if
\begin{equation} \label{E: difference_sub}
\left\{ \begin{array}{r l}
	V(m) + F_{i}(D^{+}V(m),D^{-}V(m)) \leq f_{i}(-m \Delta x) & \quad \text{if} \, \, m \in J_{i} \setminus \{0\} \\
	V(0) \leq \frac{1}{K} \sum_{i = 1}^{K} V(1_{i})
\end{array} \right.
\end{equation}
Similarly, a function $W : \mathcal{J} \to \mathbb{R}$ is a super-solution of \eqref{E: difference_equation} if
\begin{equation} \label{E: difference_super}
\left\{ \begin{array}{r l}
	W(m) + F_{i}(D^{+}W(m),D^{-}W(m)) \geq f_{i}(-m \Delta x) & \quad \text{if} \, \, m \in J_{i} \setminus \{0\} \\
	W(0) \geq \frac{1}{K} \sum_{i = 1}^{K} W(1_{i})
\end{array} \right.
\end{equation}
As always, $U$ is a solution of \eqref{E: difference_equation} if and only if it is both a sub- and super-solution.  

Along with the notion of sub- and super-solutions, there is a corresponding comparison principle.  This is encapsulated in the following definition of monotonicity:  

\begin{definition} \label{D: monotone} The scheme \eqref{E: difference_equation} is \emph{monotone} if the following two criteria hold:

(i) If $V,\chi : \mathcal{J} \to \mathbb{R}$, $V$ is a sub-solution, and $V - \chi$ has a global maximum at $m \in J_{i}$, then
\begin{equation*}
\left\{ \begin{array}{r l} 
V(m) + F_{i}(D^{+}\chi(m),D^{-}\chi(m)) \leq f_{i}(-m \Delta x) & \quad \text{if} \, \, m \neq 0, \\
\chi(0) \leq \frac{1}{K} \sum_{i = 1}^{K} \chi(1_{i}) & \text{otherwise}. \end{array} \right.
\end{equation*}

(ii) If $W,\chi : \mathcal{J} \to \mathbb{R}$, $W$ is a super-solution, and $W - \chi$ has a global minimum at $m \in J_{i}$, then 
\begin{equation*}
\left\{ \begin{array}{r l} 
W(m) + F_{i}(D^{+}\chi(m),D^{-}\chi(m)) \geq f_{i}(-m \Delta x) & \quad \text{if} \, \, m \neq 0, \\
\chi(0) \geq \frac{1}{K} \sum_{i = 1}^{K} \chi(1_{i}) & \text{otherwise}.
\end{array} \right.
\end{equation*}
\end{definition}

When we use the term ``monotone" or ``monotonicity" in this paper, we will always mean it in the sense of Definition \ref{D: monotone} above or, in the time-dependent case, Definition \ref{D: monotone_time} below.  Here we have in mind the abstract framework exposed in \cite{numerical}.

As we will see below, the structure of the operators $F_{1},\dots,F_{K}$, and assumptions \eqref{As: num_Ham_lip} and \eqref{As: CFL_stationary} ensure that \eqref{E: difference_equation} is monotone.  In this sense, \eqref{E: difference_equation} plays the role of a CFL condition.

We are now prepared to prove Theorem \ref{T: stationary_scheme}:

\begin{proof}[Proof of Theorem \ref{T: stationary_scheme}] \textbf{Existence:}  We begin by establishing the existence of a bounded solution.  We will use a discrete version of Perron's Method.  Toward that end, we begin by observing that if $M > 0$ is given by \eqref{As: sub-sup}, then the constant functions $U_{\text{sup}}(m) = M$ and $U_{\text{sub}}(m) = -M$ are super- and sub-solutions respectively. 

We now choose the artificial viscosity $\epsilon$ in such a way that the scheme is monotone in the sense of Definition \ref{D: monotone}.  Recall the definition of the constant $L_{G}$ in \eqref{As: LG_constant}.
Observe that since $\epsilon \geq 2L_{G} \Delta x$, the maps $F_{1},\dots,F_{K} :\mathbb{R}^{2} \to \mathbb{R}$ satisfy, for each $i \in \{1,2,\dots,K\}$,
\begin{equation} \label{E: monotonicity_sufficient}
	\left\{ 
		\begin{array}{r l}
			F_{i}(p_{1},p) \leq F_{i}(q_{1},p) & \text{if} \, \, p_{1} > q_{1} \\
			F_{i}(p,p_{1}) \geq F_{i}(p,q_{1}) & \text{if} \, \, p_{1} > q_{1}
		\end{array}
	\right.
\end{equation}
It is easy to see that \eqref{E: monotonicity_sufficient} implies that \eqref{E: difference_equation} is monotone. 

Next, we use a discrete version of Perron's Method (cf.\ \cite{Ishii,User}) to obtain a solution.  Let $\mathcal{S}$ denote the set of all sub-solutions of \eqref{E: difference_equation} that are bounded above by $M$ and below by $-M$.  Define a function $U : \mathcal{J} \to \mathbb{R}$ by 
\begin{equation*}
U(m) = \sup \left\{V(m) \, \mid \, V \in \mathcal{S}\right\}.
\end{equation*}
Notice that $\sup \{|U(m)| \, \mid \, m \in \mathcal{J}\} \leq M$.  Additionally, the monotonicity of the scheme readily implies that $U$ is a sub-solution.   

It remains to verify that $U$ is a super-solution.  We proceed by contradiction.  If $U$ fails to be a super-solution, then there is a point $m \in \mathcal{J}$ at which the corresponding finite-difference inequality does not hold.  We assume $m \neq 0$, leaving the case $m = 0$ to the reader.  In particular, there is a $j \in \{1,\dots,K\}$ such that $m \in J_{j}$ and a $\delta >0$ such that:
\begin{equation} \label{E: strict_sub}
U(m) + F_{j}(D^{+}U(m),D^{-}U(m)) < f_{j}(-m\Delta x) - \delta.
\end{equation}
Note that the monotonicity of \eqref{E: difference_equation} implies that $U(m) < M$.  Define $V : \mathcal{J} \to \mathbb{R}$ by setting $V(k) = U(k)$ if $k \neq m$ and $V(m) = U(m) + \delta'$, where $\delta' > 0$ is so small that $V(m) < M$ and \eqref{E: strict_sub} remains true if $V$ replaces $U$.  

Using \eqref{E: monotonicity_sufficient}, it is not hard to show that $V$ satisfies the sub-solution inequalities at the other nodes.  Therefore, $V \in \mathcal{S}$.  In particular, we deduce $U(m) < V(m) \leq U(m)$, which is a contradiciton. 

Finally, we claim that $\text{Lip}(U) \leq \overline{L}_{c}$ for some $\overline{L}_{c} > 0$ provided $L_{c}$ is sufficiently large.  By \eqref{As: coercive}, we can fix $\overline{L}_{c} > 0$ such that
\begin{equation*}
H_{i}(p) \geq M + \sup \{|f_{i}(x)| \, \mid \, x \in \mathcal{I}\} + 1 \quad \text{if} \, \, |p| \geq K^{-1}\overline{L}_{c}, \, \, i \in \{1,2,\dots,K\}.
\end{equation*}
Henceforth, assume that $L_{c} \geq \overline{L}_{c}$.  Given $i$ and $m \in J_{i}$, define a test function $\chi$ by
\begin{equation*}
\chi(k) = \left\{ \begin{array}{r l}
			U(m) + \overline{L}_{c} \Delta x |k - m|, & \quad k \in J_{i} \\
			U(m) + \overline{L}_{c} m \Delta x + K^{-1}\overline{L}_{c} k \Delta x, & \quad \text{otherwise}
		\end{array} \right.
\end{equation*}
Observe that by the choice of $\overline{L}_{c}$, the bounds on $U$, the consistency assumption \eqref{As: num_Ham_consistent}, and monotonicity of the scheme, $U - \chi$ attains its maximum in $\mathcal{J}$ at $m$.  In particular, $U(k) \leq \chi(k)$ if $k \in J_{i}$.  Since $U(m) = \chi(m)$, the previous inequality readily implies $D^{-}U(m) \geq -\overline{L}_{c}$ in general.  If, in addition, $m \neq 0$, we also find $D^{+}U(m) \leq \overline{L}_{c}$.

Since $m$ was arbitrary, by observing that $D^{+}U(m) = D^{-}U(m - 1)$ if $m \neq 0$, we conclude $|D^{+}U(m)| \vee |D^{-}U(m)| \leq \overline{L}_{c}$ for all $m \in J_{i} \setminus \{0\}$ and all $i \in \{1,2,\dots,K\}$.  We leave it to the reader to verify that this implies $\text{Lip}(U) \leq \overline{L}_{c} \Delta x$.  

\textbf{Uniqueness:} Finally, we show that $U$ is the unique bounded solution.  Suppose $\tilde{U}$ is some other bounded solution of \eqref{E: difference_equation}.  We claim that $U \leq \tilde{U}$.  For a given $\alpha > 0$, let $m_{\alpha}$ denote the maximum of the function $m \mapsto U(m) - \tilde{U}(m) + \alpha m \Delta x- \alpha |m\Delta x|^{2}$.  A straightforward computation shows that 
	\begin{equation*}
		\lim_{\alpha \to 0^{+}} \left(U(m_{\alpha}) - \tilde{U}(m_{\alpha}) \right) = \sup \left\{U(m) - \tilde{U}(m) \, \mid \, m \in \mathcal{J} \right\}.
	\end{equation*}  
Thus, we only need to show that the left-hand side is non-positive.  

Consider the case $m_{\alpha} \in J_{j} \setminus \{0\}$ for some $j$.  Since $U$ is a solution of \eqref{E: difference_equation} and $m \mapsto U(m) - \tilde{U}(m) + \alpha m \Delta x - \alpha |m\Delta x|^{2}$ has a global maximum at $m_{\alpha}$, the monotonicity of \eqref{E: difference_equation} and \eqref{As: num_Ham_lip} implies
	\begin{equation*}
		U(m_{\alpha}) + F_{j} \left(D^{+}\tilde{U}(m_{\alpha}), D^{-}\tilde{U}(m_{\alpha}\right) \leq C (\alpha|m_{\alpha}\Delta x| + \alpha \Delta x).
	\end{equation*}
Subtracting from this the equation $\tilde{U}(m_{\alpha}) + F_{j} \left(D^{+}\tilde{U}(m_{\alpha}),D^{-}\tilde{U}(m_{\alpha})\right) = 0$, we obtain
	\begin{equation*}
		U(m_{\alpha}) - \tilde{U}(m_{\alpha}) \leq C(\alpha |m_{\alpha}\Delta x| + \alpha \Delta x).
	\end{equation*}
At the same time, since $m_{\alpha}$ is a maximum, we know that
	\begin{align*}
		\alpha |m_{\alpha}\Delta x|^{2} &\leq \tilde{U}(m_{\alpha}) - U(m_{\alpha}) - (U(0) - \tilde{U}(0)) + \alpha m_{\alpha} \Delta x \\
			&\leq 2 \left(M + \sup \left\{\tilde{U}(m) \, \mid \, m \in \mathcal{J} \right\}\right) + m_{\alpha}|\Delta x|,
	\end{align*}
and, thus, $\lim_{\alpha \to 0^{+}} \alpha|m_{\alpha}\Delta x| = 0$.  Therefore, $\lim_{\alpha \to 0^{+}} (U(m_{\alpha}) - \tilde{U}(m_{\alpha})) \leq 0$. 

If $m_{\alpha} = 0$, we argue similarly.  Thus, we have proved that $U \leq \tilde{U}$.  Interchanging the roles of $U$ and $\tilde{U}$, we see that $\tilde{U} \leq U$ also holds.  In particular, $U = \tilde{U}$.    
\end{proof}

\subsection{The proof of convergence} \label{S: error_analysis_num_stationary}  The proof of Theorem \ref{T: rate_stationary_diff} is almost identical to that of Theorem \ref{T: rate_stationary_visc}, the only difference being that we cannot send $\delta \to 0^{+}$ in the definition of $C_{i}(\delta)$.  

Instead, for each $i \in \{1,\dots,K\}$, we define $\tilde{C}_{i}= D^{+}U(1_{i}) + \frac{\Delta x}{\sqrt{\epsilon}}$ and, by analogy with the proof in the vanishing viscosity case, we study test functions $\Phi_{i,\alpha} : J_{i} \times \overline{I_{i}} \to \mathbb{R}$ defined by
\begin{equation*}
\Phi_{i,\alpha}(m,y) = U(m) - u(y) - \frac{(-m \Delta x - y)^{2}}{2 \sqrt{\epsilon}} - \tilde{C}_{i}(-m \Delta x - y) - \alpha (m \Delta x)^{2},
\end{equation*}
Since $U$ and $u$ are both bounded by $M$, the penalization ensures that we can fix a point $(m_{i}(\alpha),y_{i}(\alpha))$ where $\Phi_{i,\alpha}$ attains its maximum in $J_{i} \times \overline{I_{i}}$.  
As a consequence of the definition of $\tilde{C}_{i}$, $m_{i}(\alpha) = 0$ only if 
\begin{equation*}
U(0) - U(1_{i}) \geq \tilde{C}_{i} \Delta x - \frac{y_{i}(\alpha) \Delta x}{\sqrt{\epsilon}} - \frac{(\Delta x)^{2}}{2 \sqrt{\epsilon}} - \alpha (\Delta x)^{2}.
\end{equation*} 
Since $- \frac{y_{i}(\delta,\alpha) \Delta x}{\sqrt{\epsilon}} \geq 0$, we can divide by $\Delta x$ and appeal to the definition of $\tilde{C}_{i}$ to find
	\begin{equation*}
		D^{+}U(1_{i}) = \frac{U(0) - U(1_{i})}{\Delta x} \geq D^{+}U(1_{i}) + \frac{\Delta x}{2 \sqrt{\epsilon}} - \alpha \Delta x,
	\end{equation*}
which is a contradiction if $\alpha$ is small enough.  Thus, a numerical version of Lemma \ref{L: magic} holds, but the $\frac{\Delta x}{\sqrt{\epsilon}}$ term adds a discretization error to the subsequent estimates.  It follows from \eqref{As: CFL_stationary} that this additional discretization error vanishes in the limit like $\sqrt{\Delta x}$.  In particular, the discretization does not change the order of the approximation.  

\begin{proof}[Sketch of the Proof of Theorem \ref{T: rate_stationary_diff}]  For a given $\Delta x$ and $\epsilon$ satisfying the assumptions of the theorem, we will show how to obtain a bound on $U - u$ following the approach of the proof of Theorem \ref{T: rate_stationary_visc} in Subsection \ref{S: stationary_proof}.  Let $\alpha > 0$ be a free parameter and define $\{\Phi_{1,\alpha},\dots, \Phi_{K,\alpha}\}$ as in the discussion above.  In view of the penalization terms, for each $i$, we can fix a point $(m_{i}(\alpha),y_{i}(\alpha))$ maximizing $\Phi_{i,\alpha}$ in its domain.  As in the vanishing viscosity argument, there are two cases: (i) there is a $j$ and a sequence $\alpha_{n} \to 0$ such that $y_{j}(\alpha_{n}) < 0$ for all $n$, or (ii) $y_{i}(\alpha) = 0$ for all $i \in \{1,\dots,K\}$ and all $\alpha > 0$ sufficiently small.  

Assume we are in case (ii).  By passing to the limit $\alpha \to 0^{+}$ and arguing by compactness, we recover, for each $i$, a maximizer $(m_{i},0)$ of the function 
	\begin{equation*}
		\Phi_{i}(m,y) = U(m) - u(y) - \frac{(-m\Delta x - y)^{2}}{2 \sqrt{\epsilon}}- \tilde{C}_{i}(-m\Delta x - y).
	\end{equation*}
Notice that the arguments of the previous paragraph show that $m_{i} \neq 0$.

First, since $m \mapsto \Phi_{i}(m,0)$ is maximized at $m_{i}$, we use the inequality $\Phi_{i}(m_{i} + 1,0) \vee \Phi_{i}(m_{i} - 1, 0) \leq \Phi_{i}(m_{i},0)$ to obtain 
	\begin{align}
		U(m_{i}) - U(m_{i} + 1) &\geq \left(\tilde{C}_{i} - \frac{m_{i} \Delta x}{\sqrt{\epsilon}} - \frac{(\Delta x)}{2 \sqrt{\epsilon}} \right) \Delta x \label{E: error_analysis_annoying_1} \\
		U(m_{i} - 1) - U(m_{i}) &\leq \left( \tilde{C}_{i} - \frac{m_{i}\Delta x}{\sqrt{\epsilon}} + \frac{\Delta x}{2 \sqrt{\epsilon}} \right) \Delta x \label{E: error_analysis_annoying_2}
	\end{align}
Thus, from the inequality $\text{Lip}(U) \leq \overline{L}_{c} \Delta x$, we conclude
	\begin{equation*}
		\left|\tilde{C}_{i}- \frac{m_{i} \Delta x}{\sqrt{\epsilon}} \right| \leq \overline{L}_{c} + \frac{\Delta x}{2 \sqrt{\epsilon}}.
	\end{equation*}
In particular, by \eqref{As: CFL_stationary}, this quantity is less than $\overline{L}_{c} + 1$ if $\epsilon \leq 16 L_{G}^{2}$.  

Next, notice that monotonicity of the scheme yields
	\begin{equation*}
		U(m_{i}) + F_{i} \left( D^{+} \Phi_{i}(m_{i},0), D^{-}\Phi_{i}(m_{i},0)\right) \leq f_{i}(-m_{i} \Delta x),
	\end{equation*}
where $D^{+} \Phi_{i}(m,0)$ and $D^{-}\Phi_{i}(m,0)$ are to be understood as $D^{+}$ and $D^{-}$ acting on the grid function $m \mapsto \Phi_{i}(m,0)$.  
Observe that $D^{+}\Phi_{i}(m_{i},0)$ and $D^{-}\Phi_{i}(m_{i},0)$ are exactly the terms in parentheses on the right-hand sides of \eqref{E: error_analysis_annoying_1} and \eqref{E: error_analysis_annoying_2}, and $D^{+}\Phi_{i}(m_{i},0) - D^{-}\Phi_{i}(m_{i},0) = \frac{\Delta x}{\sqrt{\epsilon}}$.  
Therefore, if $\epsilon \leq 16 L_{G}^{2}$, we can appeal to \eqref{As: num_Ham_lip}, \eqref{As: num_Ham_consistent}, and \eqref{As: cut_off_bound} to find
	\begin{align*}
		U(m_{i}) - \sqrt{\epsilon} + H_{i} \left( \tilde{C}_{i} - \frac{m_{i}\Delta x}{\sqrt{\epsilon}} \right) - f_{i}(-m_{i}\Delta x)  &\leq \frac{C \Delta x}{\sqrt{\epsilon}},
	\end{align*}
where the constant $C$ only depends on $L_{G}$ and $L_{2}$.  

Finally, using the super-solution property of $u$ and Lemma \ref{L: neumann_type_lemma}, we find a $j \in \{1,\dots,K\}$ and a $\tilde{\theta} \in [0,1]$ such that
	\begin{equation*}
		u(0) + H_{j} \left( \tilde{C}_{j} - \frac{m_{j} \Delta x}{\sqrt{\epsilon}} - \tilde{\theta} \left(\sum_{i = 1}^{K} \tilde{C}_{i} - \frac{m_{i} \Delta x}{\sqrt{\epsilon}} \right)^{+} \right) \geq f_{j}(0).
	\end{equation*}
Notice that since $U$ is a solution of \eqref{E: difference_equation}, the definition of $\{\tilde{C}_{1},\dots,\tilde{C}_{K}\}$ yields
	\begin{equation*}
		\sum_{i = 1}^{K} \left(\tilde{C}_{i} - \frac{m_{i} \Delta x}{\sqrt{\epsilon}} \right)< \sum_{i = 1}^{K} D^{+}U(1_{i}) + \frac{K \Delta x}{\sqrt{\epsilon}} = \frac{K \Delta x}{\sqrt{\epsilon}}.
	\end{equation*} 
Therefore, by \eqref{As: continuity},
	\begin{equation*}
		u(0) + H_{j} \left( \tilde{C}_{j} - \frac{m_{j} \Delta x}{\sqrt{\epsilon}} \right) \geq f_{j}(0) - \frac{C \Delta x}{\sqrt{\epsilon}}.
	\end{equation*}
Putting this together with the equation we derived for $U$ and using \eqref{As: continuity} once more, we find
	\begin{equation*}
		U(m_{i}) - u(0) \leq \sqrt{\epsilon} + \frac{C \Delta x}{\sqrt{\epsilon}} + C|m_{i}\Delta x|,
	\end{equation*}
where $C$ depends on $L_{G}$, $L_{2}$, and the Lipschitz constant of $H_{j}$ in $\overline{I_{j}} \times [-(\overline{L}_{c} + 1), \overline{L}_{c} + 1]$.  
Using the inequality $\Phi_{i}(m_{i},-m_{i}\Delta x) \leq \Phi_{i}(m_{i},y_{i})$ and the Lipschitz continuity of $u$, we find $|m_{i} \Delta x| = |-m_{i} \Delta x - 0| \leq C \sqrt{\epsilon}$, where the constant depends only $\text{Lip}(u)$.  This and the CFL condition \eqref{As: CFL_stationary} together yield
	\begin{equation*}
		U(m_{i}) - u(0) \leq C \sqrt{\Delta x},
	\end{equation*}
where the constant depends only on $\text{Lip}(u)$, $L_{G}$, $L_{2}$, and the Lipschitz constant of the Hamiltonians restricted to gradients in $[-(\overline{L}_{c} + 1),\overline{L}_{c} + 1]$.  All of this is provided, again, that $\epsilon \leq 16 L_{G}^{2}$.  
	
The numerical generalization of the arguments in case (i) of Theorem \ref{T: rate_stationary_visc} follows similarly.  The key point is that we are ultimately sending $m \to \infty$ so the discretization errors of the form $\alpha_{n} \Delta x$ vanish in the limit.  Moreover, as in the vanishing viscosity case, the additional gradient terms of the form $-\alpha_{n} m_{i}(\alpha) \Delta x$ also vanish as $n \to \infty$.  Therefore, the penalization terms in case (i) do not contribute to the error.  
  
The remainder of the proof of Theorem \ref{T: rate_stationary_visc} generalizes readily so we omit the details.  Notice, however, that we only treated the case when $\epsilon$ is sufficiently small.  If $\epsilon > 16 L_{G}^{2}$, then we use the a priori estimate $U - u \leq 2M$ to deduce that $U - u \leq C \sqrt{\Delta x}$ in this case, where the constant $C$ now depends also on $M$ and the constant $L_{2}$ from \eqref{As: CFL_stationary}. \end{proof}

\begin{remark} \label{R: half_relaxed_stationary}  It is worth noting that a qualitative proof of existence of the solution of \eqref{E: stationary} and the convergence of the scheme can be obtained using the method of half-relaxed limits (cf.\ \cite[Section 6]{barles book}).  Let us denote by $U^{\Delta x}$ the solution of \eqref{E: difference_equation} at scale $\Delta x$, suppressing the dependence on $\epsilon$ while assuming that the assumptions of Theorem \ref{T: rate_stationary_diff} hold.  Define $u_{*}, u^{*} : \mathcal{I} \to \mathbb{R}$ by 
	\begin{align*}
		u^{*}(x) = \limsup_{\delta \to 0^{+}} \sup \left\{U^{\Delta x}(m) \, \mid \, d(-m\Delta x, x) + \Delta x \leq \delta \right\}, \\
		u_{*}(x) = \liminf_{\delta \to 0^{+}} \inf \left\{ U^{\Delta x}(m) \, \mid \, d(-m\Delta x, x) + \Delta x \leq \delta \right\}.
	\end{align*}
Notice that the estimate $\text{Lip}(U^{\Delta x}) \leq \overline{L}_{c} \Delta x$ implies $u_{*}$ and $u^{*}$ are uniformly Lipschitz continuous.  Similarly, they are bounded, and the definition implies $u_{*} \leq u^{*}$.  Standard arguments and the assumptions of Theorem \ref{T: rate_stationary_diff} imply that $u^{*}$ is a sub-solution of \eqref{E: stationary} and $u_{*}$ is a super-solution (cf.\ \cite[Proof of Theorem 2.1]{numerical}).  Therefore, by comparison, $u^{*} \leq u_{*}$.  In particular, $u_{*} = u^{*}$, and if we define $u$ by $u = u^{*} = u_{*}$, then $u$ is the solution.  Moreover, after studying the definition of $u^{*}$ and $u_{*}$ and using the equality $u_{*} = u^{*}$, it becomes clear that for each $R > 0$, convergence holds in the following sense:
	\begin{equation} \label{E: local_convergence}
		\lim_{\Delta x \to 0^{+}} \sup \left\{ |U^{\Delta x}(m) - u(-m\Delta x)| \, \mid \, d(-m\Delta x,0) \leq R \right\} = 0.
	\end{equation}
Of course, Theorem \ref{T: rate_stationary_diff} significantly improves \eqref{E: local_convergence}.  \end{remark}


\section{The Cauchy Problem: Vanishing Viscosity Limit} \label{S: cauchy_problem}

We now turn to the vanishing viscosity limit for the time-dependent equation.  The problem of error estimates for this approximation motivated our return to the comparison proof in \cite{time-dependent}, and in what follows the ideas presented in Section \ref{S: time_comparison} are instrumental.

\subsection{Preliminaries}

In this section, $u$ denotes the solution of \eqref{E: time} and $u^{\epsilon}$, the solution of \eqref{E: viscous_time}.  For the purposes of the proof, we will consider these up to time $T + 1$.  

In the arguments to follow, we assume, in addition to \eqref{As: data_lip}, the following condition:
	\begin{equation} \label{As: second_deriv}
		[u_{0}]_{2} < \infty.
	\end{equation}
Note that this, together with \eqref{As: data_lip}, implies $[u_{0}]_{1} + [u_{0}]_{2} < \infty$, and, thus, Theorem \ref{T: existence_time_viscous} implies $u^{\epsilon} \in \text{Lip}(\mathcal{I} \times [0,T + 1])$.  Later, in Remark \ref{R: second_deriv}, we discuss how to remove this assumption.  

Since we are assuming \eqref{As: data_lip} and \eqref{As: second_deriv}, Theorems \ref{T: existence_time} and \ref{T: existence_time_viscous} in Appendix \ref{A: cauchy_existence} imply that for each $K > 0$, there is an $L_{K,T} > 0$ such that 
	\begin{equation} \label{E: time_viscous_lipschitz}
		\sup \left\{ \text{Lip}(u) \vee \text{Lip}(u^{\epsilon}) \, \mid \, \epsilon \in (0,K]\right\} \leq L_{K,T},
	\end{equation}
where in the Lipschitz constants we are again considering $u$ and $(u^{\epsilon})_{\epsilon > 0}$ as solutions in the domain $\mathcal{I} \times [0,T + 1]$.  
We will use \eqref{E: time_viscous_lipschitz} throughout the error analysis.  

Note that, as in the time-independent case, once we prove $u^{\epsilon} \to u$, it will follow that $L_{K,T} = \sup \left\{\text{Lip}(u^{\epsilon}) \, \mid \, \epsilon \in (0,K]\right\}$ can be used in \eqref{E: time_viscous_lipschitz}.  Notice that this number depends only on $\text{Lip}(u_{0})$, $T$, and $K$ by Theorem \ref{T: existence_time_viscous}, as claimed in the statement of Theorem \ref{T: rate_cauchy_visc}.  

Henceforth, we fix $K > 0$ and assume as in the statement of Theorem \ref{T: rate_cauchy_visc} that $\epsilon \in (0,K]$.

\subsection{Sup- and Inf-convolutions}  We define the sup-convolution $u^{\epsilon,\theta}$ of $u^{\epsilon}$ and the inf-convolution $u_{\theta}$ of $u$ as in Subsection \ref{S: proof_of_comparison_time}.  Since $u, u^{\epsilon} \in \text{Lip}(\mathcal{I} \times [0,T + 1])$, the following generalization of Lemma \ref{P: sup_inf_convolution} holds with slightly improved estimates:   

\begin{prop} \label{P: special_sup_inf_equations}  Let $T_{\theta} = 2L_{K,T} \theta$.  If $T_{\theta} < T + 1$, then $u_{\theta}$ is a super-solution of \eqref{E: modified_equation_2} with $\omega(\xi) = L_{K,T} \xi$ and the right-endpoint $T$ replaced by $T + 1$, and $u^{\epsilon,\theta}$ is a sub-solution of the following equation:
	\begin{equation} \label{E: modified_equation_viscous}
		\left\{
			\begin{array}{r l}
				u^{\epsilon,\theta}_{t} - \epsilon u^{\epsilon,\theta}_{x_{i}x_{i}} + H_{i}(t,x,w^{\epsilon,\theta}_{x_{i}}) - DT_{\theta} = 0 & \text{in} \, \, I_{i} \times (T_{\theta},T + 1) \\
				\sum_{i = 1}^{K} u^{\epsilon,\theta}_{x_{i}} = 0 & \text{on} \, \, \{0\} \times (T_{\theta},T+1) \\
				u^{\epsilon,\theta} = u_{0} + 2L_{K,T} T_{\theta} & \text{on} \, \, \mathcal{I} \times \{T_{\theta}\} 
			\end{array}
		\right.
	\end{equation}
\end{prop} 

\begin{proof}  Everything follows as in the proof of Lemma \ref{L: sup_inf_convolution}.  We can improve the right-hand side $\omega(T)$ in \eqref{E: important_later} to $L_{K,T}|t - s|$, and thereby improve the estimate $T_{\theta} = C \sqrt{\theta}$ to $T_{\theta} = 2L_{K,T} \theta$.  \end{proof}  

Finally, before proceeding, we observe that $u^{\epsilon,\theta}$ and $u_{\theta}$ are uniformly Lipschitz continuous in space-time with the same Lipschitz constants as the functions from which they are derived, that is, 
	\begin{equation} \label{E: need_that_sup_inf_lipschitz}
		\text{Lip}(u^{\epsilon,\theta}) \vee \text{Lip}(u_{\theta}) \leq L_{K,T}.
	\end{equation}
We leave the verification of this fact as an exercise for the reader.

\subsection{Modulus of continuity of sub- and super-differentials}  In Section \ref{S: time_comparison}, we saw that the comparison principle for \eqref{E: time} is proved using the continuity properties of sub-differentials (resp.\ super-differentials) of semi-convex (resp.\ semi-concave) functions.  In order to obtain a rate of convergence in the vanishing viscosity limit, we will estimate the modulus of continuity of these multi-valued maps at a maximum point of $u^{\epsilon,\theta} - u_{\theta} - bt$.  Actually, for technical reasons, we will add a penalization to this function in order to prevent the maximum from occurring at $T + 1$.    

The next result provides this estimate:

\begin{prop} \label{P: time_rate}  Fix $b \in \mathbb{R}$, assume that $T_{\theta} < t_{0} \leq T + 1$ and $\theta \leq K$, and suppose $t_{0} > 0$ maximizes the function $t \mapsto u^{\epsilon,\theta}(x_{0},t) - u_{\theta}(x_{0},t) - bt - \sqrt{\theta}(T + 1 - t)^{-1}$ in $[0,T + 1]$.  Let $c_{1} = u^{\epsilon,\theta}_{t}(x_{0},t_{0})$ and $c_{2} = u_{\theta,t}(x_{0},t_{0})$.  Then there is a constant $A >0$ depending only on $L_{K,T}$ and $T$ such that:

(i) If $a \in \partial^{-} u^{\epsilon,\theta}(x,t_{0})$ and $d(x,x_{0}) \leq A \theta$, then 
\begin{equation*}
a \geq c_{1} - 2 \sqrt{A} \left(\frac{d(x,x_{0})}{\theta}\right)^{\frac{1}{2}}.
\end{equation*} 

(ii) If $a \in \partial^{+} u_{\theta}(x,t_{0})$ and $d(x,x_{0}) \leq A\theta$, then 
\begin{equation*}
a \leq c_{2} + 2 \sqrt{A} \left(\frac{d(x,x_{0})}{\theta}\right)^{\frac{1}{2}}.
\end{equation*}   
\end{prop}    

\begin{proof}  We only prove (i) since the proof of (ii) follows similarly. Observe that, by the semi-concavity of $u_{\theta}$, we can write
\begin{equation*}
u_{\theta}(x_{0},t) \leq u_{\theta}(x_{0},t_{0}) + c_{2}(t - t_{0}) + \frac{(t - t_{0})^{2}}{2 \theta}.
\end{equation*}
Moreover, since $(x_{0},t_{0})$ is a max of $u^{\epsilon,\theta} - u_{\theta} - bt - \sqrt{\theta}(T + 1 - t)^{-1}$, we find
\begin{align*} 
	u^{\epsilon,\theta}(x_{0},t) - u^{\epsilon,\theta}(x_{0},t_{0}) &\leq u_{\theta}(x_{0},t) -u_{\theta}(x_{0},t_{0}) + b(t - t_{0})  \\
		&\quad + \sqrt{\theta}(T + 1 - t)^{-1} - \sqrt{\theta}(T + 1 - t_{0})^{-1}
\end{align*}
Applying Taylor's Theorem with remainder to the last two terms, we find, for each $t \in (0,t_{0})$, a $t' \in [t,t_{0}]$ such that
\begin{align*}
		u^{\epsilon,\theta}(x_{0},t) - u^{\epsilon,\theta}(x_{0},t_{0}) &\leq c_{2}(t- t_{0}) + b(t - t_{0}) + \sqrt{\theta}(T + 1 - t_{0})^{-2}(t - t_{0})  \\
		&\quad + \frac{(t - t_{0})^{2}}{2 \theta} + \frac{\sqrt{\theta}(t - t_{0})^{2}}{(T + 1 - t')^{3}}
\end{align*}
Note that since $t_{0}$ is an interior maximum of $t \mapsto u^{\epsilon,\theta}(x_{0},t) - u_{\theta}(x_{0},t) - bt - \sqrt{\theta} (T + 1 - t)^{-1}$ in $[0,T + 1]$, Proposition \ref{P: derivatives_exist} implies that $b + \sqrt{\theta}(T + 1 - t_{0})^{-2}= c_{1} - c_{2}$.  Putting this together with the inequality $T + 1 - t' \geq T + 1 - t_{0}$, we arrive at
	\begin{equation} \label{E: important_inequality_unfortunately}
		u^{\epsilon,\theta}(x_{0},t) - u^{\epsilon,\theta}(x_{0},t_{0}) \leq c_{1}(t - t_{0})  + \frac{(t - t_{0})^{2}}{2 \theta} + \frac{\sqrt{\theta}(t - t_{0})^{2}}{(T + 1 - t_{0})^{3}}
	\end{equation}
	
From the inequality $u^{\epsilon,\theta}(x_{0},0) - u_{\theta}(x_{0},0) - \sqrt{\theta}(T + 1)^{-1} \leq u^{\epsilon,\theta}(x_{0},t_{0}) - u_{\theta}(x_{0},t_{0}) - bt_{0}-  \sqrt{\theta}(T + 1 - t_{0})^{-1}$ and the $t$-independent inequality $|u^{\epsilon,\theta}(x_{0},t) - u_{\theta}(x_{0},t)| \leq 2(L_{K,T}T + 2L_{K,T}T_{\theta})$, we obtain
	\begin{equation*}
		\sqrt{\theta}(T + 1- t_{0})^{-1} \leq 4(L_{K,T}T + 2L_{K,T}T_{\theta}) + \sqrt{\theta}(T + 1)^{-1}.
	\end{equation*}
Since $\theta \leq K$, there is a $C >0$ depending only on $K$, $L_{K,T}$, and $T$ such that
	\begin{equation*}
		(T + 1 - t_{0})^{-3} \leq C \sqrt{\theta}^{-\frac{3}{2}},
	\end{equation*}
and, thus, plugging this into \eqref{E: important_inequality_unfortunately}, we find
	\begin{equation}\label{E: important}
		u^{\epsilon,\theta}(x_{0},t) - u^{\epsilon,\theta}(x_{0},t_{0}) \leq c_{1}(t - t_{0})  + \frac{C(t - t_{0})^{2}}{\theta}
	\end{equation}

If $x \in \overline{I_{i}}$ and $a \in \partial^{-}u^{\epsilon,\theta}(x,t_{0})$, then 
\begin{equation*}
u^{\epsilon,\theta}(x,t) \geq u^{\epsilon,\theta}(x,t_{0}) + a (t - t_{0}) - \frac{(t - t_{0})^{2}}{2 \theta}.
\end{equation*}
Moreover, $u^{\epsilon,\theta}(x,t) = (u^{\epsilon,\theta}(x,t) - u^{\epsilon,\theta}(x_{0},t)) + u^{\epsilon,\theta}(x_{0},t)$.  Thus,
\begin{equation*}
L_{K,T} d(x,x_{0}) + u^{\epsilon,\theta}(x_{0},t) \geq u^{\epsilon,\theta}(x,t_{0}) + a(t - t_{0}) - \frac{(t - t_{0})^{2}}{2 \theta}.
\end{equation*}
Appealing to \eqref{E: important}, we obtain
\begin{equation*}
L_{K,T}d(x,x_{0}) + u^{\epsilon,\theta}(x_{0},t_{0}) + c_{1}(t - t_{0}) + \frac{C(t - t_{0})^{2}}{\theta} \geq u^{\epsilon,\theta}(x,t_{0}) + a(t - t_{0}).
\end{equation*}
Using the Lipschitz property of $u^{\epsilon,\theta}$ once more, we see that
\begin{equation*}
(a - c_{1})(t - t_{0}) \leq 2 L_{K,T} d(x,x_{0}) + \frac{C(t - t_{0})^{2}}{\theta}.
\end{equation*}
Setting $t - t_{0} = - \zeta$ for some $\zeta > 0$, we find
\begin{equation} \label{E: variational_formula}
a - c_{1} \geq - \left(\frac{2L_{K,T} d(x,x_{0})}{\zeta} + \frac{C\zeta}{\theta}\right).
\end{equation}
The choice $\zeta_{0} = \sqrt{2L_{K,T}C^{-1}d(x,x_{0}) \theta}$ maximizes the right-hand side of \eqref{E: variational_formula}, from which we obtain, assuming $t = t_{0} - \zeta_{0} \in [0,T + 1]$,
\begin{equation*}
a - c_{1} \geq -2 \sqrt{2 CL_{K,T}} \left(\frac{d(x,x_{0})}{\theta} \right)^{\frac{1}{2}}.
\end{equation*}
However, we know that $t_{0} > T_{\theta} = 2L_{K,T} \theta$.  Thus, we need to check that the following inequality holds:
\begin{equation*}
t_{0} - \zeta_{0} > T_{\theta} - \zeta_{0} = 2L_{K,T}\theta - \sqrt{2L_{K,T}C^{-1} d(x,x_{0}) \theta} \geq 0.
\end{equation*} 
This is the case if and only if $d(x,x_{0}) \leq 2CL_{K,T}\theta$.  We conclude by setting $A = 2C L_{K,T}$.
\end{proof}

\subsection{The proof of Theorem \ref{T: rate_cauchy_visc}} \label{S: ugly_proof}  Since the proofs of the upper and lower bounds in Theorem \ref{T: rate_cauchy_visc} are similar, here we only establish the former.  As in the proof of comparison in Section \ref{S: time_comparison}, the error estimate is obtained by studying the values of $b >0$ for which the function $(x,t) \mapsto u^{\epsilon,\theta}(x,t) - u_{\theta}(x,t) - bt$ is maximized at a point $(x_{0},t_{0})$ satisfying $t_{0} > 0$.

Let $b = (2T)^{-1} \sup \left\{(u^{\epsilon}(x,t) - u(x,t))^{+} \, \mid \, (x,t) \in \mathcal{I} \times [0,T]\right\}$, define a function $f_{b} : \mathcal{I} \times [0,T+1] \to \mathbb{R}$ by 
\begin{equation*}
f_{b}(x,t) = u^{\epsilon,\theta}(x,t) - u_{\theta}(x,t) - bt - \sqrt{\theta} (T + 1 - t)^{-1},
\end{equation*}  
and let $\delta > 0$ be a small parameter to be determined.  In the proof, we will always assume $\theta \leq K$, leaving the justification of this to Remark \ref{R: assumption_verified}.  Later, both $\delta$ and $\theta$ will be specified.  

To proceed, we consider the following three cases: 
\begin{itemize}
\item[(1)] The supremum of $f_{b}$ is approximated by points in $\mathcal{I} \times [0,T_{\theta} \wedge (T + 1)]$.
\item[(2)]  $T_{\theta} < T + 1$ and the supremum of $f_{b}$ is attained in $\bigcup_{i = 1}^{K} \overline{I_{i}^{\delta}} \times (T_{\theta},T + 1]$.
\item[(3)]  $T_{\theta} < T + 1$ and the supremum of $f_{b}$ is approximated by points in the domain $\bigcup_{i = 1}^{K} (I_{i} \setminus \overline{I_{i}^{\delta}}) \times (T_{\theta},T+1]$.  
\end{itemize}

To simplify the notation, in what follows we will simply write $L$ instead of $L_{K,T}$.  

\textbf{Case 1: $t_{0} \leq T_{\theta} \wedge (T+1)$.}  

Suppose there is a sequence $(x_{n},t_{n}) \in \mathcal{I} \times [0,T_{\theta} \wedge (T+1)]$ such that 
	\begin{equation*}
	\sup \left\{f_{b}(x,t) \, \mid \, (x,t) \in \mathcal{I} \times [0,T + 1]\right\} = \lim_{n \to \infty} f_{b}(x_{n},t_{n}).
	\end{equation*}  
Arguing as in Lemma \ref{L: sup_inf_convolution} and Proposition \ref{P: special_sup_inf_equations}, we find
\begin{equation*}
u^{\epsilon,\theta}(x_{n},t_{n}) - u_{\theta}(x_{n},t_{n}) \leq u^{\epsilon}(x_{n},t_{n}) - u(x_{n},t_{n}) + 4LT_{\theta}.
\end{equation*}
Since $t_{n} \in [0,T_{\theta} \wedge (T+1)]$ and $u^{\epsilon}(x_{n},0) = u(x_{n},0)$, the Lipschitz continuity implies
\begin{equation*}
u^{\epsilon,\theta}(x_{n},t_{n}) - u_{\theta}(x_{n},t_{n}) \leq 6L T_{\theta} \leq C \theta.
\end{equation*}
Thus, $\lim_{n \to \infty} f_{b}(x_{n},t_{n}) \leq C \theta$.  In particular, if $(x,t) \in \mathcal{I} \times [0,T]$, then
\begin{align*}
u^{\epsilon}(x,t) - u(x,t) &\leq u^{\epsilon,\theta}(x,t) - u_{\theta}(x,t) \\
	&\leq f_{b}(x,t) + bT + \sqrt{\theta} \\
	&\leq C\theta + \sqrt{\theta} + bT
\end{align*}
In view of the definition of $b$ and the assumption $\theta \leq K$, this gives
\begin{equation*}
b\leq C \sqrt{\theta}.
\end{equation*}

\textbf{Case 2: $T_{\theta} < t_{0} \leq T$ and the maximum is $\delta$-close to junction.}  

Let $(x_{0},t_{0}) \in \bigcup_{i = 1}^{K} \overline{I_{i}^{\delta}} \times (T_{\theta},T+1]$ be such that 
	\begin{equation*}
		f_{b}(x_{0},t_{0}) = \sup \left\{f_{b}(x,t) \, \mid \, (x,t) \in \mathcal{I} \times [0,T+1]\right\}.
	\end{equation*}  

Since $(x_{0},t_{0})$ maximizes $f_{b}$, both $u^{\epsilon,\theta}$ and $u_{\theta}$ are differentiable in time at that point by Proposition \ref{P: derivatives_exist}.  We let $c_{1} = u_{t}^{\epsilon,\theta}(x_{0},t_{0})$ and $c_{2} = u_{\theta,t}(x_{0},t_{0})$ and note that $b + \sqrt{\theta}(T + 1 - t_{0})^{-2} = c_{1} - c_{2}$.  

Next, we use Proposition \ref{P: time_rate}.  We proceed by freezing time and doubling variables, introducing penalization terms so that the maxima do not deviate too far from $x_{0}$.  As in Section \ref{S: time_comparison}, we use semi-convexity and semi-concavity to convert equations \eqref{E: time} and \eqref{E: viscous_time} into stationary equations.  The error estimate is then obtained using techniques similar to those employed in Section \ref{S: stationary_problem}.

For each $i$, define the test function $\Phi_{i} : \overline{I_{i}} \times \overline{I_{i}} \to \mathbb{R}$ by 
\begin{equation*}
\Phi_{i}(x,y) = u^{\epsilon,\theta}(x,t_{0}) - u_{\theta}(y,t_{0}) - \frac{(x - y)^{2}}{2 \eta} - (p_{i} + \nu)(x - y) + \nu y,
\end{equation*}
where $p_{i} = \liminf_{I_{i} \ni x \to 0} \frac{u^{\epsilon,\theta}(x,t_{0}) - u^{\epsilon,\theta}(0,t_{0})}{x}$ and $C_{0}, \eta,\nu > 0$ are parameters to be determined.  For now, we only require that the parameters satisfy the equation $\delta = \frac{C_{0} \eta}{\nu}$.  In what follows, let us also assume that $\eta, \nu \leq 1$.  This assumption will be verified in Remark \ref{R: assumption_verified} below.  

The term $\nu y$ ensures that we can fix a point $(x_{\eta},y_{\eta}) = (x_{i,\eta,\nu},y_{i,\eta,\nu})$ that maximizes $\Phi_{i}$.  Repeating the arguments in Lemma \ref{L: magic}, we see that the choice of $p_{i}$ implies that $x_{\eta} < 0$, independently of $i$.

The following estimates show that $y_{\eta}$ is not too far from $x_{0}$.

\begin{prop} \label{P: important_estimates}  We have: $|x_{\eta} - y_{\eta}| \leq 2(2L + 1) \eta$ and $|y_{\eta}| \leq 2C_{0}L \left(\frac{\eta}{\nu^{2}}\right) + 2(2L + 1)^{2} \left(\frac{\eta}{\nu}\right)$.  \end{prop}  

\begin{proof}  Using the fact that $(x_{\eta},y_{\eta})$ is a maximum of $\Phi_{i}$, we find
\begin{align*}
u^{\epsilon,\theta}(y_{\eta},t_{0}) - u_{\theta}(y_{\eta},t_{0}) + \nu y_{\eta} &\leq u^{\epsilon,\theta}(x_{\eta},t_{0}) - u_{\theta}(y_{\eta},t_{0}) - \frac{(x_{\eta} - y_{\eta})^{2}}{2 \eta} \\
	&\quad \quad - (p_{i} + \nu)(x_{\eta} - y_{\eta}) +\nu y_{\eta}.
\end{align*}  
Rearranging and using the inequality $\text{Lip}(u^{\epsilon,\theta}) \leq L$, we obtain
\begin{equation*}
\frac{(x_{\eta} - y_{\eta})^{2}}{2 \eta} \leq (2L +1) |x_{\eta} - y_{\eta}|,
\end{equation*}
and, thus,
\begin{equation*}
|x_{\eta} - y_{\eta}| \leq 2(2L + 1) \eta.
\end{equation*}

Appealing again to the fact that $(x_{\eta},y_{\eta})$ maximizes $\Phi_{i}$, we find 
\begin{align*}
u^{\epsilon,\theta}(x_{0},t_{0}) - u_{\theta}(x_{0},t_{0}) - 2L |x_{0}| &\leq u^{\epsilon,\theta}(0,t_{0}) - u_{\theta}(0,t_{0}) \\
	&\leq u^{\epsilon,\theta}(x_{\eta},t_{0}) - u_{\theta}(y_{\eta},t_{0}) - \frac{(x_{\eta} - y_{\eta})^{2}}{2 \eta} \\
	&\quad \quad - (p_{i} + \nu)(x_{\eta} - y_{\eta}) + \nu y_{\eta}.
\end{align*}
Arguing as before, using that $|x_{0}| \leq \delta = C_{0} \nu^{-1} \eta$, and recalling that $(x_{0},t_{0})$ maximizes $f_{b}$, we find
\begin{align*}
\nu |y_{\eta}| &\leq (u^{\epsilon,\theta}(x_{\eta},t_{0}) - u^{\epsilon,\theta}(y_{\eta},t_{0})) + (L + 1) |x_{\eta} - y_{\eta}| + 2L|x_{0}|  \\
	&\quad + f_{b}(y_{\eta},t_{0}) - f_{b}(x_{0},t_{0}) \\
	&\leq (2L + 1) |x_{\eta} - y_{\eta}| + 2L|x_{0}| \\
	&\leq  2(2L + 1)^{2} \eta + 2C_{0}L \nu^{-1} \eta.
\end{align*}
\end{proof}

Observe that Proposition \ref{P: important_estimates} implies that if $\frac{\eta}{\nu^{2}} = o(\theta)$ as $\epsilon \to 0^{+}$, then if $A$ is the constant from Proposition \ref{P: time_rate} and $\epsilon > 0$ is sufficiently small, we have
	\begin{equation} \label{E: technical_condition1}
		d(x_{\eta},x_{0}) \vee d(y_{\eta},x_{0}) \leq C \left(\frac{\eta}{\nu^{2}} + \frac{\eta}{\nu} + \eta\right) \leq A\theta.
	\end{equation}
When we choose the parameters $\eta,\nu$, we will see that, in fact, things can be arranged in such a way that \eqref{E: technical_condition1} holds independently of $\epsilon \in (0,K]$; see Remark \ref{R: assumption_verified} below.  Thus, for now, we assume that \eqref{E: technical_condition1} holds so that we can apply Proposition \ref{P: time_rate} with $x \in \{x_{\eta},y_{\eta}\}$.  

Using Proposition \ref{P: time_rate} and Lemma \ref{L: critical}, arguing as in Section \ref{S: time_comparison}, and letting $\tilde{c}_{1} = c_{1} - DT_{\theta}$ and $\tilde{c}_{2} = c_{2} + DT_{\theta}$, we find that $x \mapsto u^{\epsilon,\theta}(x,t_{0})$ and $x \mapsto u_{\theta}(x,t_{0})$ are respectively sub- and super-solutions of the following stationary equations:
\begin{align*}
&\left\{\begin{array}{r l}
	\tilde{c}_{1} - 2 \sqrt{A} \left(\frac{d(x,x_{0})}{\theta} \right)^{\frac{1}{2}} - \epsilon u^{\epsilon,\theta}_{x_{i}x_{i}}(\cdot,t_{0}) + H_{i}(t_{0},x,u^{\epsilon,\theta}_{x_{i}}(\cdot,t_{0}))= 0 &  \text{in} \, \, I_{i}^{A\theta} \\
	\sum_{i = 1}^{K} u^{\epsilon,\theta}_{x_{i}}(\cdot,t_{0}) = 0 & \text{on} \, \, \{0\}
	\end{array} \right.  \\
&\left\{ \begin{array}{r l}
	\tilde{c}_{2} + 2 \sqrt{A} \left(\frac{d(x,x_{0})}{\theta} \right)^{\frac{1}{2}} + H_{i}(t_{0},x,u_{\theta,x_{i}}(\cdot,t_{0})) = 0 &  \text{in} \, \, I_{i}^{A\theta} \\
	\sum_{i = 1}^{K} u_{\theta,x_{i}}(\cdot,0) = 0 &  \text{on} \, \, \{0\}	
	\end{array} \right. 	
\end{align*}
To obtain an estimate on $b$, we evaluate the equations above at $x = x_{\eta}$ and $x = y_{\eta}$.  

There are two cases to consider, namely, (i) $y_{i,\eta,\nu} < 0$ for some $i$, and (ii) $y_{i,\eta,\nu} = 0$ for all $i$.  In either case, we obtain the same estimate on $b$.  Since the computation for (i) is essentially an easier version of the one for (ii), we give the details only for the latter. 

It follows from Proposition \ref{P: neumann_sup} that there is a $j$ and a $\tilde{\theta} \in [0,1]$ such that
\begin{equation*}
c_{2} + 2 \sqrt{A} \left(\frac{|x_{0}|}{\theta}\right)^{\frac{1}{2}} + H_{j} \left(t_{0},0,p_{i} + 2\nu + \frac{x_{\eta}}{\eta} - \tilde{\theta} F(\nu)^{+}\right) + DT_{\theta} \geq 0,
\end{equation*}
where $F(\nu)$ is given by 
\begin{equation*}
F(\nu) = 2K \nu + \sum_{i = 1}^{K} \left(p_{i} + \frac{x_{i,\eta,\nu}}{\eta}\right).
\end{equation*}
In view of Lemma \ref{L: classical_kirchoff}, Remark \ref{R: calculus}, and the stationary equation satisfied by $u^{\epsilon,\theta}(\cdot,t_{0})$, we know that $\sum_{i = 1}^{K} p_{i} \leq 0$.  The inequality $F(\nu)^{+} < 2 K \nu$ follows.  Note also that since $x \mapsto \Phi_{j}(x,0)$ has a maximum at $x_{\eta} < 0$ and $\text{Lip}(u^{\epsilon,\theta}) \leq L$, we know that $|p_{j} + \nu + \frac{x_{j}}{\eta}| \leq L$.  Combining these two observations with \eqref{As: continuity}, we obtain
\begin{equation*}
c_{2} + 2 \sqrt{A} \left(\frac{d(y_{\eta},x_{0})}{\theta}\right)^{\frac{1}{2}} + H_{j} \left(t_{0},0,p_{i} + \nu + \frac{x_{\eta}}{\eta}\right) + DT_{\theta} \geq - C \nu,
\end{equation*}
where the constant $C$ depends only on the Lipschitz constant of $H_{j}$ in $[0,T] \times \overline{I_{j}} \times [-(L + 2K\nu),(L + 2K\nu)]$, which is bounded since we have assumed $\nu \leq 1$ holds.
At the same time, using the equation satisfied by $u^{\epsilon,\theta}$, we get
\begin{equation*}
c_{1} - 2 \sqrt{A} \left(\frac{d(x_{\eta},x_{0})}{\theta}\right)^{\frac{1}{2}} - \frac{\epsilon}{\eta} + H_{j} \left(t_{0},x_{\eta},p_{i} + \nu + \frac{x_{\eta}}{\eta}\right) - D T_{\theta} \leq 0.
\end{equation*}
Putting the last two estimates together, recalling that $b + \sqrt{\theta}(T + 1 - t_{0})^{-2}= c_{1} - c_{2}$, and appealing to \eqref{E: technical_condition1}, Proposition \ref{P: important_estimates}, and \eqref{As: continuity}, we obtain
\begin{equation} \label{E: error_estimate}
b \leq C\left( \left(\frac{\eta}{\theta \nu^{2}}\right)^{\frac{1}{2}} + \nu + \eta\right) + \frac{\epsilon}{\eta} + 2 D T_{\theta}.
\end{equation}

Neglecting the $2D T_{\theta}$ and $\eta$ terms, it is straightforward to verify that the rest of the terms in the right-hand side of \eqref{E: error_estimate} are optimized when $\nu =  C_{1}\left(\frac{\epsilon}{\theta}\right)^{\frac{1}{5}}$ and $\eta = C_{2} \epsilon^{\frac{4}{5}} \theta^{\frac{1}{5}}$ for some constants $C_{1},C_{2} > 0$, and, thus,
\begin{equation} \label{E: case2}
b \leq C \left( \left(\frac{\epsilon}{\theta}\right)^{\frac{1}{5}} + \epsilon^{\frac{4}{5}} \theta^{\frac{1}{5}}  \right) + 2 D T_{\theta}.
\end{equation}   
Henceforth, we will assume $\nu$ and $\eta$ have the form given above, although it will be convenient to adjust the constants $C_{1}$ and $C_{2}$; see Remark \ref{R: assumption_verified}.  Note that the choice of $C_{1},C_{2}$ does not change the order of the upper bound in \eqref{E: case2}, though it does change the constant.  At the end of Case 3, we fix $\theta$.  

\textbf{Case 3: $T_{\theta} < t_{0} \leq T + 1$ and the maximum is $\delta$-far from junction.}  

Since neither Case 1 nor Case 2 hold, there is a $j \in \{1,2,\dots,K\}$ such that the supremum of $f_{b}$ is approximated by points in $(I_{j} \setminus [-\delta,0]) \times (T_{\theta},T + 1]$.  
As in Case 2, we use the equations solved by $u^{\epsilon,\theta}$ and $u_{\theta}$.  However, it is no longer necessary to freeze the time.  We use a familiar variable-doubling argument to obtain estimates on $b$.  There is nonetheless a slight technicality since we wish to prevent the maximum of the test function from occurring at the junction.  To avoid this possibility, we introduce a penalization.

Fix $(x_{0},t_{0}) \in (I_{j} \setminus [-\delta,0]) \times (T_{\theta},T+1]$ such that 
	\begin{equation} \label{E: key_maximization}
		f_{b}(x_{0},t_{0}) > \sup \{f_{b}(x,t) \, \mid \, (x,t) \in \mathcal{I} \times [0,T + 1]\} - \eta,
	\end{equation}
let $R = (8(L + 1)\eta) \vee \nu |x_{0}|$, and fix a twice continuously differentiable function $g_{\eta,\nu} : [0,\infty) \to [0,\infty)$ such that
\begin{itemize}
\item[(1)] $g_{\eta,\nu}(s) = s$ if $s \leq R$,
\item[(2)] $g_{\eta,\nu}(s) = R + \eta$ if $s \geq R + 2$,
\item[(3)] $0 \leq g_{\eta,\nu} \leq R + \eta$ in $[0,\infty)$,
\item[(4)] $0 \leq g_{\eta,\nu}' \leq 1$ in $[0,\infty)$.
\end{itemize}
In the same way that we used a linear term in our test function in Case 2 to prevent maxima from straying too far from the junction, we use $g_{\eta,\nu}$ to keep maxima away from the junction in the present case.  

Let $\Psi_{j} : (\overline{I_{j}} \times [0,T +1])^{2} \to \mathbb{R}$ be defined by
\begin{align*}
\Psi_{j}(x,t,y,s) &= u^{\epsilon,\theta}(x,t) - u_{\theta}(y,s) - \frac{(x - y)^{2}}{2 \eta} - \frac{(t - s)^{2}}{2 \eta} - bt \\
	&\quad \quad - \sqrt{\theta}(T+1 - t)^{-1} + g_{\eta,\nu}(-\nu y) - \alpha x^{2}
\end{align*}
where $\alpha \in (0,1)$ is a free parameter that will eventually be sent to zero.  Due to the penalization term $\alpha x^{2}$, there is a $(\bar{x},\bar{t},\bar{y},\bar{s})$ that maximizes $\Psi_{j}$.  

Appealing to the fact $u$ and $g_{\eta,\nu}$ are both Lipschitz and $\nu \leq 1$, we use $\Psi_{j}(\bar{x},\bar{t},\bar{x},\bar{t}) \leq \Psi_{j}(\bar{x},\bar{t},\bar{y},\bar{s})$ to find
	\begin{align*}
		\frac{(\bar{x} - \bar{y})^{2} + (\bar{t} - \bar{s})^{2}}{2 \eta} &\leq u_{\theta}(\bar{x},\bar{t}) - u_{\theta}(\bar{y},\bar{s}) + g_{\eta,\nu}(-\nu \bar{y}) - g_{\eta,\nu}(-\nu \bar{x}) \\
			&\leq (L + 1) (|\bar{x} - \bar{y}| + |\bar{t} - \bar{s}|).
	\end{align*}
Thus, using Jensen's inequality, we obtain $|\bar{x} - \bar{y}| + |\bar{t} - \bar{s}| \leq 4(L + 1) \eta$.

We claim, in addition, that if $\alpha$ is small enough, then $\max\{\bar{x},\bar{y}\} < 0$.  If $\nu \bar{y} < -R$, the claim follows since $R \geq 8(L+1)\eta$.  Thus, we may assume $\nu \bar{y} \geq -R$. 
Since $\Psi_{j}(x_{0},t_{0},x_{0},t_{0}) \leq \Psi_{j}(\bar{x},\bar{t},\bar{y},\bar{s})$, $(x_{0},t_{0})$ satisfies \eqref{E: key_maximization}, and $g(s) = s$ if $s \in \{-\nu x_{0}, - \nu \bar{y}\}$, we find
\begin{equation*}
\nu \bar{y} \leq \nu x_{0} + \eta + \alpha |x_{0}|^{2} + 4L(L + 1)\eta \leq (-C_{0} + 4L(L+1) + 1)\eta + \alpha |x_{0}|^{2}.
\end{equation*}
Assuming then that $0 < \alpha < \frac{\eta}{ |x_{0}|^{2}}$, we obtain $\nu \bar{y} \leq (- C_{0} + 4L(L+1) + 2)\eta$ and, thus, 
\begin{equation*}
\nu \bar{y} < 0 \quad \text{provided} \, \, C_{0} > 4L(L+1)+ 2.
\end{equation*}  
Similarly, since $|\bar{x} - \bar{y}| \leq 4(L+1) \eta$, we find 
\begin{equation*}
\nu \bar{x} < 0 \quad \text{if} \, \,  C_{0} > 4(L+1)^{2} + 2.
\end{equation*}  

Henceforth fix a $C_{0}$ satisfying $C_{0} > 4(L+1)^{2} + 2$.  Then the previous arguments establish that $\max\{\bar{x},\bar{y}\} < 0$ provided $\alpha$ is small enough.  

Finally, we estimate the penalization term so that we can later pass to the limit $\alpha \to 0^{+}$.  Since $\Psi_{j}(0,0,0,0) \leq \Psi_{j}(\bar{x},\bar{t},\bar{y},\bar{s})$ and $u^{\epsilon}(\bar{x},\bar{t}) - u(\bar{y},\bar{s}) \leq L(2T + |\bar{x} - \bar{y}|)$, we find
\begin{equation*}
\alpha \bar{x}^{2} \leq C(1 + |\bar{x} - \bar{y}|) \leq C(1 +4L(L + 1) \eta)
\end{equation*}
and deduce from this that $\alpha \bar{x}^{2}$ is bounded.  

We now have all the estimates necessary to obtain an upper bound on $b$.  There are two cases to consider: (i) there is a sequence $\alpha_{n} \to 0^{+}$ such that $\min \{\bar{t},\bar{s}\} \leq T_{\theta}$ for all $n$, (ii) $\min\{\bar{t},\bar{s}\} > T_{\theta}$ for all sufficiently small $\alpha$.  First, we will consider case (ii).  

Taking advantage of the fact that $(\bar{x},\bar{t},\bar{y},\bar{s})$ maximizes $\Psi_{j}$, together with the fact that $\min\{\bar{t},\bar{s}\} > T_{\theta}$ and $\max\{\bar{x},\bar{y}\} < 0$, we can appeal to the equations to find
\begin{equation*}
b + \sqrt{\theta} (T + 1 - \bar{t})^{-2} + \frac{\bar{t} - \bar{s}}{\eta} - \frac{\epsilon}{\eta} - 2 \alpha \epsilon + H_{j} \left(\bar{t}, \bar{x}, \frac{\bar{x} - \bar{y}}{\eta} + 2 \alpha \bar{x} \right) - D T_{\theta} \leq 0,
\end{equation*}
and
\begin{equation*}
\frac{\bar{t} - \bar{s}}{\eta} + H_{j} \left(\bar{s}, \bar{y}, \frac{\bar{x} - \bar{y}}{\eta} - \nu g'_{\eta,\nu}(-\nu \bar{y})\right) + D T_{\theta} \geq 0.
\end{equation*}
Note that since $y \mapsto \Psi_{j}(\bar{x},\bar{t},y,\bar{s})$ is maximized at $\bar{y}$ and $x \mapsto \Psi_{j}(x,\bar{t},\bar{y},\bar{s})$ is maximized at $\bar{x}$, the Lipschitz continuity of $u_{\theta}$ and $u^{\epsilon,\theta}$ implies that $|\frac{\bar{x} - \bar{y}}{\eta} - \nu g_{\eta,\nu}'(-\nu \bar{y})| \vee |\frac{\bar{x} - \bar{y}}{\eta} + 2 \alpha \bar{x}| \leq L$.  Thus, we can subtract the inequalities above and use \eqref{As: continuity} to obtain
\begin{equation*}
b \leq \frac{\epsilon}{\eta} + 2 \alpha \epsilon + C(2 \alpha |\bar{x}| + \nu g'_{\eta,\nu}(-\nu \bar{y})) + C(|\bar{x} - \bar{y}| + |\bar{t} - \bar{s}|) + 2DT_{\theta}.
\end{equation*}  
Letting $\alpha \to 0^{+}$, we find
\begin{equation*}
b \leq \frac{\epsilon}{\eta} + C(\eta + \nu) + 2 D T_{\theta}.
\end{equation*}
Therefore, since $\nu = C_{1} \left(\frac{\epsilon}{\theta}\right)^{\frac{1}{5}}$ and $\eta = C_{2} \epsilon^{\frac{4}{5}} \theta^{\frac{1}{5}}$, we conclude
	\begin{equation} \label{E: case_3_bound}
		b \leq C \left( \epsilon^{\frac{4}{5}}\theta^{\frac{1}{5}} + \left(\frac{\epsilon}{\theta}\right)^{\frac{1}{5}} \right) + 2DT_{\theta}.
	\end{equation}

Finally, consider case (i).  Let us restrict attention to $\alpha = \alpha_{n}$ for some $n$.  First, note that, by the definition of $R$ and $g_{\eta,\nu}$, 
	\begin{equation*}
		g_{\eta,\nu}(-\nu \bar{y}) - g_{\eta,\nu}(-\nu x_{0}) \leq (8(L+1) + 1) \eta.
	\end{equation*}  
Using this and the fact that $\bar{t} \leq T_{\theta} + C \eta$ in this case, we obtain the following estimate on $f_{b}(x_{0},t_{0})$:
	\begin{align*}
		f_{b}(x_{0},t_{0}) &= \Psi_{j}(x_{0},t_{0},x_{0},t_{0}) - g_{\eta,\nu}(-\nu x_{0}) + \alpha_{n} x_{0}^{2} \\
			&\leq \Psi_{j}(\bar{x},\bar{t},\bar{y},\bar{s}) - g_{\eta,\nu}(-\nu x_{0}) + \alpha_{n} x_{0}^{2} \\
			&\leq u^{\epsilon,\theta}(\bar{x},\bar{t}) - u_{\theta}(\bar{x},\bar{t})+ C \eta + (g_{\eta,\nu}(-\nu \bar{y}) - g_{\eta,\nu}(-\nu x_{0})) + \alpha_{n} x_{0}^{2} \\
			&\leq (u_{0}(\bar{x}) - u_{0}(\bar{x})) + C(T_{\theta}+ \eta) + (g_{\eta,\nu}(-\nu \bar{y}) - g_{\eta,\nu}(-\nu x_{0})) + \alpha_{n}x_{0}^{2} \\
			&\leq C(T_{\theta} + \eta) + \alpha_{n} x_{0}^{2}
	\end{align*}
Sending $n \to \infty$ and recalling how we chose $(x_{0},t_{0})$, we find 
	\begin{align*}
		\sup \left\{u^{\epsilon}(x,t) - u(x,t) \, \mid \, (x,t) \in \mathcal{I} \times [0,T] \right\} &\leq f_{b}(x_{0},t_{0}) + \eta + bT + \sqrt{\theta} \\
			&\leq C(T_{\theta} + \eta) + bT + \sqrt{\theta}.
	\end{align*}
Appealing to the definition of $b$ and $\eta$ and subtracting the term $bT$ to the left-hand side, we conclude
	\begin{equation*}
		b \leq C(T_{\theta} + \epsilon^{\frac{4}{5}}\theta^{\frac{1}{5}} + \sqrt{\theta})
	\end{equation*}

Combining the last bound with \eqref{E: case_3_bound}, using our assumption $\theta \leq K$, and recalling that $T_{\theta} = 2L\theta$, we arrive at 
	\begin{equation*}
		b \leq C \left(\sqrt{\theta} + \epsilon^{\frac{4}{5}} \theta^{\frac{1}{5}} + \left(\frac{\epsilon}{\theta}\right)^{\frac{1}{5}}\right).
	\end{equation*}

\textbf{Conclusion}  

It remains to choose $\theta$ in such a way as to minimize the upper bounds on $b$ obtained in the previous three cases.  Observe that Case 1, Case 2, and Case 3 establish the following upper bound:
\begin{equation} \label{E: terrible}
b \leq C \max \left\{ \left(\frac{\epsilon}{\theta}\right)^{\frac{1}{5}} + \epsilon^{\frac{4}{5}} \theta^{\frac{1}{5}}, \sqrt{\theta}, \sqrt{\theta} + \left(\frac{\epsilon}{\theta}\right)^{\frac{1}{5}} + \epsilon^{\frac{4}{5}} \theta^{\frac{1}{5}} \right\} + CT_{\theta}.
\end{equation}
In what follows, we ignore the right-most term in the braces, the terms $\epsilon^{\frac{4}{5}} \theta^{\frac{1}{5}}$ and $CT_{\theta}$, and instead minimize $\max\{\left(\frac{\epsilon}{\theta}\right)^{\frac{1}{5}}, \sqrt{\theta}\}$ with respect to $\theta$.  In the next paragraph, we will see that this doesn't change the order of the right-hand side of \eqref{E: terrible}.  

Notice that $\max\{\left(\frac{\epsilon}{\theta}\right)^{\frac{1}{5}}, \sqrt{\theta}\}$ is minimized at the intersection of the two curves, that is, when $\sqrt{\theta} = \left(\frac{\epsilon}{\theta}\right)^{\frac{1}{5}}$.  In particular, the choice $\theta = \epsilon^{\frac{2}{7}}$ is the minimizer.  Plugging this into \eqref{E: terrible}, recalling that $T_{\theta} = 2 L \theta$ and $\epsilon \leq K$, we obtain a $C_{K} > 0$ such that
\begin{equation} \label{E: order}
b \leq C_{K} \epsilon^{\frac{1}{7}}.  
\end{equation}
In view of the definition of $b$, this implies the upper bound in Theorem \ref{T: rate_cauchy_visc}.  

\begin{remark} \label{R: assumption_verified} \emph{Notice that the choice $\theta = \epsilon^{\frac{2}{7}}$ implies $\eta = C_{1}\epsilon^{\frac{6}{7}}$ and $\nu = C_{2} \epsilon^{\frac{1}{7}}$.  Thus, $\frac{\eta}{\nu^{2}} = o(\theta)$ as $\epsilon \to 0^{+}$, as we previously assumed.  Moreover, we are free to make the constant $C_{1}$  as small as we like since this does not change the order of the error in \eqref{E: case2} or \eqref{E: order}, it only changes the constant $C$.  Thus, by appropriately choosing $C_{1}$, we can ensure that \eqref{E: technical_condition1} holds independently of $\epsilon \in (0,K]$.}

\emph{Finally, observe that $\eta, \nu \leq 1$ holds independently of $\epsilon \in (0,K]$ provided we choose $C_{1},C_{2}$ small enough.  This ties up the loose end in the first paragraph of Case 2 above.  Similarly, $\theta \leq K$ holds since $K \geq 1$ and $\epsilon \leq K$.}    \end{remark}  

\begin{remark} \label{R: second_deriv}   \emph{It remains to prove the error estimate when $[u_{0}]_{2} = \infty$.  In this case, we approximate $u_{0}$ using the functions $(v_{0}^{\epsilon})_{\epsilon > 0}$ defined in Proposition \ref{P: lipschitz_approximation} in Appendix \ref{A: cauchy_existence} below.  For the moment, fix $\epsilon \in (0,K]$.  Let $v$ and $v^{\epsilon}$ be the solutions of \eqref{E: time} and \eqref{E: viscous_time}, respectively, with initial datum $v_{0}^{\epsilon}$.}

\emph{By Proposition \ref{P: lipschitz_approximation} and Theorem \ref{T: existence_time_viscous}, there is a constant $L_{K,T} > 0$, independent of $\epsilon$, such that $\text{Lip}(v) \vee \text{Lip}(v^{\epsilon}) \leq L_{K,T}$ in $\mathcal{I} \times [0,T + 1]$ and $[u_{0} - v_{0}^{\epsilon}]_{0} \leq \text{Lip}(u_{0})\epsilon$.  Therefore, by Remark \ref{R: contractivity}, $[u^{\epsilon} - v^{\epsilon}]_{0} \vee [u - v]_{0} \leq \text{Lip}(u_{0}) \epsilon$.  At the same time, since $\text{Lip}(v) \vee \text{Lip}(v^{\epsilon}) \leq L_{K,T}$, the previous arguments show there is a constant $C_{K} > 0$ depending on $L_{K,T}$, but not $\epsilon$, such that $[v^{\epsilon} - v]_{0} \leq C_{K} \epsilon^{\frac{1}{7}}$.  Therefore, by the triangle inequality, we find 
	\begin{equation*}
		[u^{\epsilon} - u]_{0} \leq (C_{K} + K^{\frac{6}{7}}\text{Lip}(u_{0})) \epsilon^{\frac{1}{7}}.  
	\end{equation*}
Since $\epsilon$ was chosen arbitrarily from $(0,K]$, this proves Theorem \ref{T: rate_cauchy_visc} in the case when $u_{0}$ satisfies \eqref{As: data_lip}, but not \eqref{As: second_deriv}.}          \end{remark}

The argument of Remark \ref{R: second_deriv} is partly inspired by \cite[Remark 2]{armstrong cardaliaguet}.  


\section{The Cauchy Problem: Finite-Difference Approximation} \label{S: cauchy_problem_diff}

We study finite-difference schemes approximating \eqref{E: time}.  These schemes take the same basic form as those used to approximate \eqref{E: stationary} in Section \ref{S: stationary_diff}.  As in that section, the error analysis follows steps similar to the ones used to obtain the error estimate in the vanishing viscosity limit.  Therefore, we will only briefly review the differences between the proofs of Theorems \ref{T: rate_cauchy_visc} and \ref{T: rate_cauchy_diff}.

\subsection{Preliminaries} \label{S: cauchy_diff_explanation}  As in the approximation of the stationary equation, we begin by discretizing the space variables.  For each $i$, let $J_{i} = \{0,1,2,\dots\}$ and define the network as the union $\mathcal{J} := \bigcup_{i = 1}^{K} J_{i}$ glued at zero.  Given a spatial scale $\Delta x > 0$ and an index $i$, we identify $m \in J_{i}$ with the point $-m \Delta x \in I_{i}$, and, as before, we will write $1_{i}$ to specify $1 \in J_{i}$ where necessary.  We also discretize the time.  Given a temporal scale $\Delta t > 0$, let $N = \lceil \frac{T}{\Delta t} \rceil$.  The discretized time interval is $S = \{0,1,2,\dots,N\}$ and the discrete time $s$ is identified with the continuous time $s \Delta t$.  

We will study the explicit finite-difference approximation of \eqref{E: time} given by
\begin{equation} \label{E: cauchy_fd}
\left\{ \begin{array}{r l}
	P_{i}(m,s,U) = 0 & \quad \text{if} \, \, (m,s) \in (J_{i} \setminus \{0\}) \times (S \setminus \{N\}) \\
	U(0,s + 1) = \frac{1}{K}\sum_{i = 1}^{K} U(1_{i},s + 1) & \quad \text{if} \, \, s \in S \setminus \{N\} \\
	U(m,0) = u_{0}(-m \Delta x) & \quad \text{if} \, \, m \in J_{i}
\end{array} \right.
\end{equation}
where the operator $P_{i}$ has the form
\begin{equation*}
P_{i}(m,s,U) = D_{t}U(m,s) + F_{i}(D^{+}U(m,s),D^{-}U(m,s)) - f_{i}(s \Delta t, -m \Delta x),
\end{equation*}
$D^{+}$, $D^{-}$ and $D_{t}$ are defined in this context by
\begin{align} 
	D_{t}U(m,s) = \frac{U(m, s + 1) - U(m,s)}{\Delta t}, \label{E: diff_time} \\
	D^{+}U(m,s) = \frac{U(m - 1,s) - U(m,s)}{\Delta x}, \\
	D^{-}U(m,s) = \frac{U(m,s) - U(m + 1,s)}{\Delta x},
\end{align}
$\{F_{1},\dots,F_{K}\}$ are the same as in \eqref{E: F_op}, and we assume \eqref{As: num_Ham_lip} and \eqref{As: num_Ham_consistent} hold.  

As in the time-independent case, the scheme is monotone provided the artificial viscosity $\epsilon$ and scales $(\Delta x, \Delta t)$ are chosen appropriately.  This is made precise in Appendix \ref{A: cauchy_fd_explanation}.  Recalling the definition of $L_{G}$ from \eqref{As: LG_constant}, we assume the classical CFL condition, that is, there is an $L_{2} > 0$ such that
\begin{equation} \label{As: CFL_cauchy}
4L_{G} \leq \frac{2\epsilon}{\Delta x} \leq \frac{\Delta x}{\Delta t} \leq L_{2}.
\end{equation}  
This assumption guarantees both monotonicity of the scheme (through the lower bound) and control over the discretization errors (via the upper bound).  

Finally, as in the time-independent case, we impose a lower-bound on the cut-off $L_{c}$ appearing in \eqref{As: num_Ham_consistent} in order to ensure consistency of the scheme.  Using the constant $\tilde{L}_{c}$ specified in Proposition \ref{P: monotone_time} below, we assume the following bound on $L_{c}$:
	\begin{equation} \label{As: cut_off_bound_time}
		L_{c} \geq \tilde{L}_{c} + 1.
	\end{equation}
	
Note that the examples of Remark \ref{R: examples_Ham} are also applicable to \eqref{E: cauchy_fd}

\subsection{The proof of Theorem \ref{T: rate_cauchy_diff}}  The approximation error for the scheme \eqref{E: cauchy_fd} is obtained following the same outline as in the vanishing viscosity approximation.  Here we will need the results of Appendix \ref{A: cauchy_fd_explanation}, especially Proposition \ref{P: monotone_time}.  As in Subsection \ref{S: ugly_proof}, we consider the solution $u$ of \eqref{E: time} in $\mathcal{I} \times [0,T + 1]$, and we study the solution $U$ of the numerical scheme \eqref{E: cauchy_fd} in $\mathcal{J} \times S_{1}$, where $S_{1} = \{0,1,\dots,N_{1}\}$ and $N_{1} = \lceil \frac{T + 2}{\Delta t} \rceil$.  The reason we run the numerical scheme up to time $T + 2$ will become apparent in Proposition \ref{P: num_scheme_lip} below.  

For the purposes of the proof, we replace the finite-difference solution $U$ by its sup-convolution $U^{\theta} : \mathcal{J} \times [0,N_{1}\Delta t] \to \mathbb{R}$ defined by 
\begin{equation*}
U^{\theta}(k,t) = \sup \left\{ U(k,s) - \frac{(t - s \Delta t)^{2}}{2 \theta} \, \mid \, s \in \{0,1,2,\dots,N_{1}\} \right\}.
\end{equation*}
We quantify the distance between $U^{\theta}$ and $U$ in the next proposition:

\begin{prop} \label{P: numerical_sup_error} If $(k,t) \in \mathcal{J} \times [0,T+2]$ and $s = \lfloor \frac{t}{\Delta t} \rfloor$, then 
	\begin{equation*}
		U(k,s) - \frac{(\Delta t)^{2}}{2 \theta} \leq U^{\theta}(k,t) \leq U(k,s) + 2(L_{2}\tilde{L}_{c})^{2}\theta + 3 L_{2}\tilde{L}_{c}\Delta t. 	
	\end{equation*}
\end{prop}  

\begin{proof}  By the definition of $U^{\theta}(k,t)$, we can fix $\bar{s} \in S_{1}$ such that $U^{\theta}(k,t) = U(k,\bar{s}) - \frac{(t - \bar{s} \Delta t)^{2}}{2 \theta}$.  Fix $\tilde{s} \in S_{1}$ such that $|t - \tilde{s} \Delta t| < \Delta t$ and either $\bar{s} \Delta t \leq \tilde{s} \Delta t \leq t$ or $t \leq \tilde{s} \Delta t \leq \bar{s} \Delta t$.  The following then follows immediately from the definition of $U^{\theta}(k,t)$:
	\begin{equation*}
		U(k,\tilde{s}) - \frac{(\Delta t)^{2}}{2 \theta} \leq U^{\theta}(k,t) = U(k, \bar{s}) - \frac{(t - \bar{s} \Delta t)^{2}}{2 \theta}.
	\end{equation*}
Thus, using the Lipschitz estimate of Proposition \ref{P: monotone_time} and the inequality $\Delta x \leq L_{2} \Delta t$ implied by \eqref{As: CFL_cauchy}, we can write
	\begin{equation*}
		\frac{(t - \bar{s} \Delta t)^{2}}{2 \theta} \leq \tilde{L}_{c}|\bar{s} - \tilde{s}| \Delta x + \frac{(\Delta t)^{2}}{2 \theta} \leq L_{2}\tilde{L}_{c}|t - \bar{s} \Delta t| + \frac{(\Delta t)^{2}}{2 \theta}.
	\end{equation*}
Setting $\xi = |t - \bar{s} \Delta t|$, we find that $\xi^{2} - 2L_{2}\tilde{L}_{c}\theta \xi - (\Delta t)^{2} \leq 0$.  Therefore, solving explicitly for the roots of the quadratic, we conclude:
	\begin{equation*}
		|t - \bar{s} \Delta t| = \xi \leq L_{2}\tilde{L}_{c} \theta + \sqrt{(L_{2}\tilde{L}_{c})^{2}\theta^{2} + (\Delta t)^{2}} \leq 2L_{2}\tilde{L}_{c}\theta + \Delta t.
	\end{equation*}
From this, we find
	\begin{align*}
		U^{\theta}(k,t) &\leq U(k,\bar{s}) \\
				&\leq U(k,\tilde{s}) + (U(k,\bar{s}) - U(k,\tilde{s})) \\
				&\leq U(k,\tilde{s}) + \tilde{L}_{c}|\bar{s} - \tilde{s}| \Delta x\\
				&\leq U(k,\tilde{s}) + L_{2}\tilde{L}_{c}(|t - \bar{s} \Delta t| + |t - \tilde{s}\Delta t|) \\
				&\leq U(k,\tilde{s}) + 2(L_{2}\tilde{L}_{c})^{2} \theta + 2L_{2}\tilde{L}_{c} \Delta t.
	\end{align*}
Finally, from the inequality $|\tilde{s} \Delta t - s \Delta t| \leq \Delta t$, we obtain
	\begin{equation*}
		U^{\theta}(k,t) \leq U(k,s) + 2(L_{2}\tilde{L}_{c})^{2} \theta + 3L_{2}\tilde{L}_{c}\Delta t.
	\end{equation*}

To get the lower bound, observe that $U^{\theta}(k,t) \geq U(k,s) - \frac{(t - s \Delta t)^{2}}{2 \theta}$ by definition.  Since $|t - s\Delta t| \leq \Delta t$, the result follows.  
\end{proof}  

Next, we find the equation satisfied by $U^{\theta}$.  Here and in what follows, we set $T_{\theta} = 2 (L_{2} \tilde{L}_{c} + \text{Lip}(u))\theta$.  

\begin{prop} \label{P: sup_numerical}  If $t \geq 2T_{\theta} - \Delta t$ and $T_{\theta} - 2 \Delta t > 0$, then $U^{\theta}$ satisfies, for each $k \in J_{i} \setminus \{0\}$,
\begin{equation*}
\frac{U^{\theta}(k,t + \Delta t) - U^{\theta}(k,t)}{\Delta t} + F_{i}(D^{+}U^{\theta}(k,t), D^{-}U^{\theta}(k,t)) \leq g_{i}(t, -k \Delta x),
\end{equation*}
where $g_{i}(t,-k \Delta x) = f_{i}(t,-k\Delta x) + C(\theta + \Delta t)$ for some constant $C > 0$ depending only on $L_{2}$, $\tilde{L}_{c}$, and $D$,
and
\begin{equation*}
U^{\theta}(0,t + \Delta t) \leq \frac{1}{K} \sum_{i = 1}^{K} U^{\theta}(1_{i},t + \Delta t).
\end{equation*}
\end{prop}

\begin{proof}  Let $s_{1} = \lfloor \frac{t + \Delta t}{\Delta t} \rfloor$.  If $s \in S_{1}$ is such that
\begin{equation*}
U^{\theta}(k,t + \Delta t) = U(k,s) - \frac{(t + \Delta t - s \Delta t)^{2}}{2 \theta},
\end{equation*}
then there are two cases to consider, namely, (i) $s \Delta t < t + \Delta t$ and (ii) $s \Delta t \geq t + \Delta t$.  In the former, we obtain
\begin{align*}
\frac{(\Delta t)^{2}(s_{1} - s)^{2}}{2 \theta} \leq \frac{(t + \Delta t - s \Delta t)^{2}}{2 \theta} &= U(k,s) - U^{\theta}(k,t + \Delta t) \\
	&\leq U(k,s) - U(k,s_{1}) + \frac{(\Delta t)^{2}}{2 \theta} \\
	&\leq L_{2}\tilde{L}_{c}\Delta t|s - s_{1}| + \frac{(\Delta t)^{2}}{2 \theta}.
\end{align*}
Let $m = s_{1} - s$ and observe that $m \in \mathbb{N} \cup \{0\}$.  Moreover, the previous inequalities give
$
(m - 1)(m + 1) \leq \frac{2L_{2}\tilde{L}_{c}\theta m}{\Delta t}
$.
This implies 
\begin{equation*}
m\Delta t \leq L_{2} \tilde{L}_{c} \theta + \sqrt{(L_{2}\tilde{L}_{c})^{2} \theta^{2} + (\Delta t)^{2}} \leq 2 L_{2}\tilde{L}_{c} \theta + \Delta t,
\end{equation*}
and, thus, 
\begin{equation*}
|(t + \Delta t) - s \Delta t| \leq |(t + \Delta t) - s_{1} \Delta t| + m \Delta t \leq 2L_{2} \tilde{L}_{c} \theta + 2\Delta t.
\end{equation*}
Hence
\begin{equation*}
s \Delta t \geq (t + \Delta t) - 2L \theta - 2\Delta t \geq T_{\theta} - 2\Delta t > 0,
\end{equation*}
and, in particular, $s = s_{0} + 1$ for some $s_{0} \in S_{1}$.  

In case (ii), since $s \Delta t \geq t + \Delta t > 0$, we immediately deduce the existence of an $s_{0}$ as in the previous paragraph.  

Suppose $k \in J_{i} \setminus \{0\}$.  Since $s \Delta t = (s_{0} + 1) \Delta t$ for some $s_{0} \in S_{1}$, we find
\begin{equation*}
\frac{U(k,s_{0} + 1) - U(k,s_{0})}{\Delta t} + F_{i}(D^{+}U(k,s_{0}),D^{-}U(k,s_{0})) = f_{i}(s_{0}\Delta t, -k \Delta x).
\end{equation*}
In view of the inequalities
\begin{equation} \label{E: useful_ineq}
\left\{ \begin{array}{r l}
U^{\theta}(k,t + \Delta t) = U(k,s_{0} + 1) - \frac{(t - s_{0}\Delta t)^{2}}{2 \theta} \\
U^{\theta}(k',t) \geq U(k',s_{0}) - \frac{(t - s_{0}\Delta t)^{2}}{2 \theta} & \quad \text{if} \, \, k' \in J_{i}
\end{array} \right.
\end{equation}
monotonicity of the scheme, assumption \eqref{As: time_bound}, and the estimate $|t - s_{0} \Delta t| \leq 2L_{2} \tilde{L}_{c}\theta + 2 \Delta t$ obtained above, we find
\begin{equation*}
\frac{U^{\theta}(k,t + \Delta t) - U^{\theta}(k,t)}{\Delta t} + F_{i}(D^{+}U^{\theta}(k,t), D^{-}U^{\theta}(k,t)) \leq g_{i}(t,-k\Delta x).
\end{equation*}

On the other hand, if $k = 0$, then summing over $i$ in \eqref{E: useful_ineq} yields
\begin{equation*}
	\sum_{i = 1}^{K} \left(U^{\theta}(0,t + \Delta t) - U^{\theta}(1_{i},t + \Delta t)\right) \leq 0.
\end{equation*}
\end{proof}

As in the vanishing viscosity case, $U^{\theta}$ inherits regularity from $U$.  In the discrete spatial variable, it is straightforward to show that $\text{Lip}(U^{\theta}(\cdot,t)) \leq \tilde{L}_{c} \Delta x$ independent of $t \in [0,T + 2]$.  The time variable is more involved, and for this reason it is convenient to consider $U$ up to time $T + 2$.  Here is the result that we use:

\begin{prop} \label{P: num_scheme_lip}  Suppose $\Delta t \leq 1$, $2T_{\theta} \leq t \leq T + 1$, and $k \in \bigcup_{i = 1}^{K} J_{i}$.  If $p \in \partial^{-} U^{\theta}(k,t)$, then $|p| \leq L_{2}\tilde{L}_{c} + \frac{\Delta t}{2 \theta}$.  In particular, for such a $k$, if $t_{1},t_{2} \in [2T_{\theta},T + 1]$, then $|U^{\theta}(k,t) - U^{\theta}(k,s)| \leq (L_{2}\tilde{L}_{c} + \frac{\Delta t}{2 \theta})|t - s|$.   \end{prop} 

\begin{proof}  Observe that the inequalities \eqref{E: useful_ineq} provide a bound on the time derivatives of $U^{\theta}$.  Indeed, if $2T_{\theta} \leq t \leq T + 1$ and $p \in \partial^{-} U^{\theta}(k,t)$, then
\begin{equation*}
U^{\theta}(k,t + \Delta t) \geq U^{\theta}(k,t) + p \Delta t - \frac{(\Delta t)^{2}}{2 \theta}
\end{equation*}
and
\begin{equation*}
U^{\theta}(k, t - \Delta t) \geq U^{\theta}(k,t) - p \Delta t - \frac{(\Delta t)^{2}}{2 \theta}.
\end{equation*}
Since $2 T_{\theta} - \Delta t \leq t - \Delta t$, we can use \eqref{E: useful_ineq} to find 
\begin{equation*}
p \geq \frac{U^{\theta}(k,t) - U^{\theta}(k,t - \Delta t)}{\Delta t} - \frac{\Delta t}{2 \theta} \geq - \left(L + \frac{\Delta t}{2 \theta}\right).
\end{equation*}
Note that since $t \leq T + 1$ and $\Delta t \leq 1$, it follows that $t + \Delta t \in [0,T + 2]$, and, in particular, $U^{\theta}(k, t + \Delta t)$ is well-defined.  Thus, arguing as before, we find
\begin{equation*}
p \leq \frac{U^{\theta}(k, t + \Delta t) - U^{\theta}(k,t)}{\Delta t} + \frac{\Delta t}{2 \theta} \leq L_{2} \tilde{L}_{c} + \frac{\Delta t}{2 \theta}.
\end{equation*}

To see that $|U^{\theta}(k,t_{1}) - U^{\theta}(k,t_{2})| \leq (L_{2} \tilde{L}_{c} + \frac{\Delta t}{2 \theta})|t_{1}- t_{2}|$, we argue using the fact that $U^{\theta}(k,\cdot)$, as the sum of a convex function and a smooth one, is Lipschitz in $[2T_{\theta},T + 1]$, and $\partial^{-}U^{\theta}(k,t) = \{U_{t}^{\theta}(k,t)\}$ almost everywhere.   
\end{proof}

With Propositions \ref{P: numerical_sup_error}, \ref{P: sup_numerical}, and \ref{P: num_scheme_lip} in hand, we obtain an upper bound on $U - u$ arguing as in the vanishing viscosity case with only minor modifications needed to accommodate the fact that the space variable in $U^{\theta}$ is discrete.  To start with, we define 
	\begin{equation*}
	b = \frac{1}{2T}\sup \left\{ U(k,s) - u(-k\Delta x, s \Delta t) \, \mid \, (k,s) \in \mathcal{I} \times S\right\},
	\end{equation*}
we set $\theta = \epsilon^{\frac{2}{7}}$, and we let $f_{b} : \mathcal{J} \times [0,T + 1] \to \mathbb{R}$ be given by
	\begin{equation*}
		f_{b}(x,t) = U^{\theta}(k,t) - u_{\theta}(-k\Delta x, t) - bt - \sqrt{\theta}(T + 1 - t)^{-1}.
	\end{equation*}
As before, we study points where $f_{b}$ attains, or almost attains, its supremum, splitting the analysis into the same three cases as in Subsection \ref{S: ugly_proof}.

We argue in Case 1 exactly as before using Proposition \ref{P: numerical_sup_error} to quantify the difference between $U^{\theta}$ and $U$.

In Case 2, we once again reduce to stationary equations.  Here we are assuming that $f_{b}$ has a maximum at $(m_{0},t_{0})$ with $d(-m_{0}\Delta x,0) \leq \delta$ and $t_{0} > 2 T_{\theta} - \Delta t$, where once again $\delta = \frac{C_{0} \eta}{\nu}$ for some large enough $C_{0}$ to be determined in Case 3, $\eta = C_{1} \epsilon^{\frac{6}{7}}$, and $\nu = C_{2} \epsilon^{\frac{1}{7}}$ for some $C_{1},C_{2} > 0$ sufficiently small.  Applying Proposition \ref{P: sup_numerical} and adapting Proposition \ref{P: time_rate} to $U^{\theta}$, we find an $A > 0$ such that if $m \in J_{i} \setminus \{0\}$ satisfies $d(-m\Delta x,-m_{0}\Delta x) \leq A \theta$, then
\begin{align} \label{E: discrete_stationary_equation}
&U_{t}^{\theta}(m_{0},t_{0}) - 2 \sqrt{A} \left(\frac{d(-m\Delta x, -m_{0} \Delta x)}{\theta}\right)^{\frac{1}{2}} \nonumber\\
	&\qquad \qquad + F_{i}(D^{+}U^{\theta}(m,t_{0}),D^{-}U^{\theta}(m,t_{0})) \leq g_{i}(t, -m\Delta x) + \frac{\Delta t}{2 \theta}.
\end{align}
This is the discrete stationary equation satisfied by $m \mapsto U^{\theta}(m,t_{0})$.  Notice also that this equation is monotone in the sense of Definition \ref{D: monotone} by \eqref{As: CFL_cauchy}.  We use this equation together with the stationary equation solved by $u_{\theta}$ to carry out the same analysis as in Section \ref{S: ugly_proof}.  

We remark that in this case, we use the test function $\Phi_{i} : J_{i} \times \overline{I_{i}} \to \mathbb{R}$ given by 
	\begin{equation*}
		\Phi_{i}(k,y) = U^{\theta}(k,t_{0}) - u_{\theta}(y,t_{0}) - \frac{(-k\Delta x - y)^{2}}{2 \eta} - p_{i}(-k \Delta x - y) + \nu y,
	\end{equation*}
where $p_{i} = D^{+}U^{\theta}(1_{i},t_{0}) + \frac{\Delta x}{\eta}$.  Notice that the choice of $(p_{1},\dots,p_{K})$ forces the first component of any maximum point of $\Phi_{i}$ to be away from the junction, just as in the analysis of the time-independent scheme in Subsection \ref{S: error_analysis_num_stationary}. 

In Case 3, the arguments are as in the vanishing viscosity case with minor changes to accommodate the discrete spatial variable in $U^{\theta}$.  Again, the definition of the test function $\Psi_{j}$ from Subsection \ref{S: ugly_proof} is modified slightly.  In particular, we define $\Psi_{j} : (J_{j} \times [0,T]) \times (\overline{I_{j}} \times [0,T]) \to \mathbb{R}$ by
	\begin{align*}
		\Psi_{j}(k,t,y,s) &= U^{\theta}(k,t) - u_{\theta}(y,s) - \frac{(-k\Delta x - y)^{2}}{2 \eta} - \frac{(t - s)^{2}}{2 \eta} - bt \\ 
			&\quad- \sqrt{\theta}(T + 1 - t)^{-1} + g_{\eta,\nu}(-\nu y) - \alpha (k\Delta x)^{2}.
	\end{align*}
where $g_{\eta,\nu}$ is defined as before, but with $-m_{0}\Delta x$ replacing $x_{0}$.

As in the analysis of the finite-difference approximation of \eqref{E: stationary} there are discretization errors, but none of these effect the rate.  For example, a $\frac{\Delta t}{2 \theta}$ term appears as a discretization error in \eqref{E: discrete_stationary_equation}, but this is much smaller than $\epsilon^{\frac{1}{7}}$ as $\epsilon \to 0^{+}$.  

Finally, we remark that, as in the vanishing viscosity case, the arguments we previously described only work if $\epsilon$ is small enough.  For example, in Proposition \ref{P: num_scheme_lip}, we required that $\Delta t \leq 1$, and, in order to use \eqref{As: num_Ham_consistent}, we will need to impose another upper bound on $\epsilon$.  On the other hand, if $\epsilon$ is too large, we use a priori estimates to get the desired bounds.  Indeed, if $(m,s) \in \mathcal{J} \times S$, then 
	\begin{align*}
		U(m,s) - u(-m\Delta x, s \Delta t) &\leq (u_{0}(-m\Delta x) + \text{Lip}(u)T) - (u_{0}(-m\Delta x) \\
		&\quad \quad - L_{2}\tilde{L}_{c}T) \\
		&\leq (\text{Lip}(u) + L_{2} \tilde{L}_{c})T.
	\end{align*}
Thus, for a given $\epsilon_{0} > 0$, if $C$ is sufficently large, then we find $U(m,s) - u(-m\Delta x, s \Delta t) \leq (\text{Lip}(u) + L_{2} \tilde{L}_{c})T \leq C \epsilon^{\frac{1}{6}}$ for all $\epsilon \geq \epsilon_{0}$.  

Notice that in this case, the constant $C$ does not depend on an upper bound on $\epsilon$, unlike the vanishing viscosity case.  That is, there is no $K$ in Theorem \ref{T: rate_cauchy_diff}.  The reason for this is the Lipschitz bound $\tilde{L}_{c}$ does not depend on $\epsilon$, but instead on $L_{2}$.  However, this is not really an improvement.  In particular, by \eqref{As: CFL_cauchy}, if $\epsilon > 0$ is too large, then $S = \{0,1\}$ and then the scheme tells us very little.  Thus, as in the vanishing viscosity approximation of \eqref{E: time}, the rate of convergence only really makes sense when $\epsilon$ is small.

\begin{remark} \label{R: half_relaxed_time} As in the time-independent setting, the finite-difference scheme \eqref{E: cauchy_fd} can be used together with the method of half-relaxed limits to prove the existence of solutions of \eqref{E: time} when $u_{0} \in \text{Lip}(\mathcal{I})$.  If we let $U^{\epsilon}$ denote the solution of \eqref{E: cauchy_fd}, suppressing the dependence on $\Delta x$ and $\Delta t$ while assuming that the assumptions of Theorem \ref{T: rate_cauchy_diff} hold, we define $u^{*}$ and $u_{*}$ in $\mathcal{I} \times [0,T]$ by
	\begin{align*}
		u^{*}(x,t) = \limsup_{\delta \to 0^{+}} \sup \left\{ U^{\epsilon}(m,s) \, \mid \, d(-m\Delta x, x) + |t - s \Delta t| + \epsilon < \delta\right\}, \\
		u_{*}(x,t) = \liminf_{\delta \to 0^{+}} \inf \left\{ U^{\epsilon}(m,s) \, \mid \, d(-m\Delta x, x) + |t - s \Delta t| + \epsilon < \delta\right\}. 
	\end{align*} 
In view of the uniform estimates on $\text{Lip}(U^{\epsilon})$, one can prove that $u^{*},u_{*} \in \text{Lip}(\mathcal{I} \times [0,T])$.  Moreover, $u^{*}$ is a sub-solution of \eqref{E: time}, while $u_{*}$ is a super-solution.  Therefore, by comparison, $u^{*} \leq u_{*}$.  Since the definition implies $u_{*} \leq u^{*}$, we conclude that $u^{*} = u_{*}$.  Letting $u = u^{*} = u_{*}$, we see that $u$ is a uniformly continuous solution of \eqref{E: time}, and a statement similar to \eqref{E: local_convergence} also holds.  

Now that we have proved that \eqref{E: time} has a solution when $u_{0} \in \text{Lip}(\mathcal{I})$, the general case when $u_{0} \in UC(\mathcal{I})$ can be recovered by approximation as in the proof of Theorem \ref{T: existence_time} below.  \end{remark}     



\appendix


\section{Reformulated Kirchoff Condition} \label{A: reformulation}

We present an equivalent definition of viscosity solutions of Kirchoff problems for the general problem:
\begin{equation} \label{K eq}
	\left\{
		\begin{array}{r l}
			F_{i}(x,u,u_{x_{i}}) = 0 & \text{in} \, \, I_{i} \\
			\sum_{i = 1}^{K} u_{x_{i}} = B & \text{on} \, \, \{0\} 
		\end{array}
	\right.
\end{equation}
where $F_{i} : \overline{I_{i}} \times \mathbb{R} \times \mathbb{R} \to \mathbb{R}$ is continuous for each $i$.  We start with sub-solutions:

\begin{prop} \label{P: alternative}  $u \in USC(\mathcal{I})$ is a sub-solution of \eqref{K eq} if and only if for each $\varphi \in C^{2}(\mathcal{I})$ and any local maximum $x_{0}$ of $u - \varphi$, the following inequalities are satisfied:
\begin{equation*}
\left\{ \begin{array}{r l}
		F_{i}(x_{0},u(x_{0}),\varphi_{x_{i}}(x_{0})) \leq 0 & \text{if} \, \, x_{0} \in I_{i} \\
		\min_{i, \tilde{\theta} \in [0,1]} F_{i}\left(0,u(0),\varphi_{x_{i}}(0) + \tilde{\theta} \left( \sum_{j = 1}^{K} \varphi_{x_{i}}(0) - B\right)^{-}\right) \leq 0 & \text{if} \, \, x_{0} = 0
		\end{array} \right.
\end{equation*}
\end{prop}

In fact, the proposition only requires $\varphi \in C^{1}(\mathcal{I})$, but we will not expand on that here.

Before proceeding, we need to recall the definitions of first-order differential sub-jets and super-jets.  Given an upper semi-continuous function $u : \mathcal{I} \to \mathbb{R}$ and an $x \in \overline{I_{i}}$, we say that $p \in J^{+}_{i}u(x)$ if and only if 
\begin{equation*}
u(y) \leq u(x) + p(y - x) + o(|y - x|) \quad \text{if} \, \, y \in \overline{I_{i}},
\end{equation*}
where $\lim_{y \to x} \frac{o(|y - x|)}{|y - x|} = 0$.  Similarly, given a lower semi-continuous function $v : \mathcal{I} \to \mathbb{R}$ and an $x \in \overline{I_{i}}$, we say that $q \in J^{-}_{i}v(x)$ if and only if 
\begin{equation*}
v(y) \geq v(x) + q(y - x) + o(|y-x|) \quad \text{if} \, \, y \in \overline{I_{i}}.
\end{equation*}

Note that there is a $\varphi \in C^{2}(\mathcal{I})$ such that $u - \varphi$ has a local maximum at $0$ if and only if $\xi_{i} := \varphi_{x_{i}}(0) \in J_{i}^{+}u(0)$ for each $i$.  Therefore, in what follows, we often work with $K$-tuples $(\xi_{1},\dots,\xi_{K})$ instead of test functions.

As in the Neumann problem, Proposition \ref{P: alternative} rests on the next lemma.    

\begin{lemma} \label{L: slopes}  Fix $i \in \{1,\dots,K\}$ and assume that $u$ is an upper semi-continuous sub-solution of
\begin{equation*}
F_{i}(x,u,u_{x_{i}}) = 0 \quad \text{in} \, \, I_{i}.
\end{equation*}
Let $\xi_{i} \in J_{i}^{+}u(0)$ and set $\lambda_{i,0} = \sup\{\lambda \geq 0 \, \mid \, \xi_{i} + \lambda \in J^{+}_{i}u(0) \}$.  If $\lambda_{i,0} < \infty$, then 
\begin{equation*}
F_{i}(0,u(0),\xi_{i} + \lambda_{i,0}) \leq 0.
\end{equation*}
\end{lemma}

\begin{proof}  The proof is the same as the one in \cite[Lemma 3]{neumann}.  For the sake of completeness, we reproduce it here.  Since $J_{i}^{+}u(0)$ is closed, we can pick $\psi \in C^{2}(\overline{I_{i}})$ such that $u - \psi$ has a strict local maximum at $0$ in $\overline{I_{i}}$, $u(0) = \psi(0)$, and $\psi_{x_{i}}(0) = \xi_{i} + \lambda_{i,0}$.  Fix $\delta >0$, set $\mu(\delta) = \psi(-\delta_{i}) - u(-\delta_{i})$ and $\alpha(\delta) = \min\{\delta,\frac{\mu(\delta)}{2 \delta}\}$, and let $x_{\delta}$ be a maximum of $u - \psi - \alpha(\delta)x$ in $\overline{I^{\delta}_{i}} \subseteq \overline{I_{i}}$.  

Observe that $x_{\delta} \neq 0$.  Indeed, if $x_{\delta} = 0$, then 
$$\xi_{i} + \lambda_{i,0} + \alpha(\delta) =\psi_{x_{i}}(0) + \alpha(\delta) \in J^{+}u_{i}(0),$$
contradicting the definition of $\lambda_{i,0}$.  

Additionally, $x_{\delta} \neq -\delta_{i}$ since
\begin{align*}
u(-\delta_{i}) - \psi(-\delta_{i}) + \alpha(\delta)\delta &\leq \frac{u(-\delta_{i}) - \psi(-\delta_{i})}{2} \\
	&< 0 = u(0) - \psi(0) < u(x_{\delta}) - \psi(x_{\delta}) - \alpha(\delta) x_{\delta}.
\end{align*}
Thus, $x_{\delta} \in (0,\delta)$ and the sub-solution property of $u$ gives
\begin{equation*}
F_{i}(x_{\delta},u(x_{\delta}),\psi_{x_{i}}(x_{\delta}) + \alpha(\delta)) \leq 0.
\end{equation*}
From the inequality $u(0) - \psi(0) < u(x_{\delta}) - \psi(x_{\delta}) - \alpha(\delta) x_{\delta}$, the upper semi-continuity of $u$ implies $u(x_{\delta}) \to u(0)$.  Therefore, observing that 
	\begin{equation*}
		\lim_{\delta \to 0^{+}} (x_{\delta},\alpha(\delta),u(x_{\delta})) = (0,0,u(0)),
	\end{equation*}
the result follows.  
\end{proof}  

In the proof of Proposition \ref{P: alternative} we will use the following fact about sub-solutions of \eqref{K eq}.  In fact, the corresponding property of the Neumann problem is actually embedded in the definition in \cite{neumann}. 

\begin{prop} \label{P: helpful_neumann_idea} Suppose $u$ is a sub-solution of \eqref{K eq}, $\varphi \in C^{2}(\mathcal{I})$, and $u - \varphi$ has a local maximum at $0$.  If $\sum_{i = 1}^{K} \varphi_{x_{i}}(0) \geq B$, then $\min_{i} F_{i}(0,u(0),\varphi_{x_{i}}(0)) \leq 0$.\end{prop}  

\begin{proof}  By the definition of sub-solution, it suffices to consider the case when $\sum_{i = 1}^{K} \varphi_{x_{i}}(0) = B$.  Define $(\xi_{1},\dots,\xi_{K})$ by $\xi_{i} = \varphi_{x_{i}}(0)$ and let $(\lambda_{1,0},\dots,\lambda_{K,0})$ be defined as in Lemma \ref{L: slopes}.  If $\lambda_{j,0} = 0$ for some $j$, then Lemma \ref{L: slopes} implies $F_{j}(0,u(0),\xi_{j}) \leq 0$.  On the other hand, if $\min_{i}\lambda_{i,0} > 0$, then, for small enough $\delta > 0$, $\xi_{i} + \delta \in J_{i}^{+}u(0)$ holds, no matter the choice of $i$.  From $\sum_{i = 1}^{K} (\xi_{i} + \delta) = B + \delta K$ and the sub-solution property, we find $\min_{i}F_{i}(0,u(0),\xi_{i} + \delta) \leq 0$.  We conclude by sending $\delta \to 0^{+}$.     \end{proof}

We continue with the

\begin{proof}[Proof of Proposition \ref{P: alternative}]  Since one direction is immediate, here we prove only the ``only if" statement.  

Suppose $(\xi_{1},\dots,\xi_{K})$ is a $K$-tuple satisfying $\xi_{i} \in J^{+}_{i}u(0)$ for each $i$.   

In what follows, we use the notation in the statement of Lemma \ref{L: slopes}.  If there is a $j \in \{1,2,\dots,K\}$ such that $\left(\sum_{i = 1}^{K} \xi_{i} -B\right)^{-} < \lambda_{j,0}$,
then let $\tilde{\xi}_{k} = \xi_{k}$ if $k \neq j$ and $\tilde{\xi}_{j} = \xi_{j} + \left(\sum_{i = 1}^{K} \xi_{i} -B\right)^{-}$.  For each $i$, $\tilde{\xi}_{i} \in J_{i}^{+} u(0)$ and
\begin{equation} \label{E: flux}
\sum_{i = 1}^{K} \tilde{\xi}_{i} = \left(\sum_{i = 1}^{K} \xi_{i}\right) + \left(\sum_{j = 1}^{K} \xi_{j} - B\right)^{-} \geq B.
\end{equation}
Thus, Proposition \ref{P: helpful_neumann_idea} implies $\min_{i} F_{i} \left(0,u(0), \tilde{\xi}_{i} \right) \leq 0$.
From this and the definition of $(\tilde{\xi}_{1},\dots,\tilde{\xi}_{K})$, we conclude
\begin{equation*}
\min_{j} \min_{\tilde{\theta} \in [0,1]} F_{j}\left(0,u(0), \xi_{j} + \tilde{\theta} \left(\sum_{i = 1}^{K} \xi_{i} -B\right)^{-} \right) \leq 0.
\end{equation*}

It only remains to consider the case when $\lambda_{j,0} \leq \left(\sum_{i = 1}^{K} \xi_{i}-B\right)^{-}$ independently of the choice of $j$.  In this case, we can fix $(\tilde{\theta}_{1},\tilde{\theta}_{2},\dots,\tilde{\theta}_{K}) \in [0,1]^{K}$ such that $\lambda_{j,0} = \tilde{\theta}_{j} \left(\sum_{i = 1}^{K} \xi_{i} - B\right)^{-}$.  Therefore, Lemma \ref{L: slopes} yields that, for each $j$,
\begin{equation*}
 \min_{\tilde{\theta} \in [0,1]} F_{j} \left(0,u(0), \xi_{j} + \tilde{\theta} \left(\sum_{i = 1}^{K} \xi_{i} - B\right)^{-} \right) \leq F_{j}(0,u(0),\xi_{j} + \lambda_{j,0}) \leq 0.
\end{equation*}\end{proof}

The result for super-solutions is stated next.  Since the proof is so similar, we omit the details.

\begin{prop} \label{P: neumann_sup}  A function $v \in \text{LSC}(\mathcal{I})$ is a viscosity super-solution of \eqref{K eq} if and only if for each $\varphi \in C^{2}(\mathcal{I})$ and any local minimum $x_{0}$ of $u - \varphi$, the following inequalities are satisfied:
\begin{equation*}
\left\{ \begin{array}{r l}
		F_{i}(x_{0},v(x_{0}),\varphi_{x_{i}}(x_{0})) \geq 0 & \text{if} \, \, x_{0} \in I_{i}\\
		\max_{i, \tilde{\theta} \in [0,1]} F_{i}\left(0,v(0),\varphi_{x_{i}}(0) - \tilde{\theta} \left( \sum_{j = 1}^{K} \varphi_{x_{i}}(0) -B \right)^{+}\right) \geq 0 & \text{if} \, \, x_{0} = 0
		\end{array} \right.
\end{equation*}
\end{prop}

There is an analogous reformulation of time-dependent equations like \eqref{E: time}.  We do not prove it here since we have no immediate use for it and it does not simplify the uniqueness proof presented in Section \ref{S: time_comparison}.


\section{Dimensionality Reduction Lemma}  \label{A: dimensionality_reduction}

In this section, we show how to obtain time-independent equations from those in which time-derivatives do not appear.  The following result implies Proposition \ref{P: time_freezing}:

\begin{lemma} \label{L: critical}  Assume that, for each $i$, $F_{i} : [0,T] \times \mathcal{I} \times \mathbb{R}^{2} \to \mathbb{R}$ is a continuous function, and fix $B \in \mathbb{R}$ and $\delta > 0$.  Let the upper semi-continuous function $u : \bigcup_{i = 1}^{K} \overline{I_{i}^{\delta}} \times [0,T] \to \mathbb{R}$ be a sub-solution of
\begin{equation} \label{E: junction}
\left\{ \begin{array}{r l}
			F_{i}(t,x,u,u_{x_{i}}) = 0 & \text{in} \, \, I_{i}^{\delta} \times (0,T) \\
			\sum_{i = 1}^{K} u_{x_{i}} = B & \text{on} \, \, \{0\} \times (0,T)
			\end{array} \right.
\end{equation}
For each $t_{0} \in (0,T]$, the function $u(\cdot,t_{0}) : \bigcup_{i = 1}^{K} \overline{I_{i}^{\delta}} \to \mathbb{R}$ is a sub-solution of
\begin{equation*}
\left\{ \begin{array}{r l}
			F_{i}(t_{0},x,u(\cdot,t_{0}),u_{x_{i}}(\cdot,t_{0})) = 0 & \text{in} \, \, I_{i}^{\delta} \\
			\sum_{i = 1}^{K} u_{x_{i}}(\cdot,t_{0}) = B & \text{on} \, \, \{0\}
			\end{array} \right.
\end{equation*} \end{lemma}  

We remark that a version of Lemma \ref{L: critical} for super-solutions follows from it by replacing $u$ by $-u$.

\begin{proof}  Fix a $t_{0} \in (0,T]$.  Given $\varphi \in C^{2}(\mathcal{I})$, suppose $u(\cdot,t_{0})- \varphi$ has a strict global maximum in $\bigcup_{i = 1}^{K} \overline{I_{i}^{\delta}}$ at $x_{0} \in \bigcup_{i = 1}^{K} I_{i}^{\delta}$.  We consider only the case when $x_{0} = 0$, the other case being slightly easier.

For each $\epsilon > 0$, let $\Phi_{\epsilon} : \bigcup_{i = 1}^{K} \overline{I_{i}^{\delta}} \times [0,T] \to \mathbb{R}$ be the function given by 
\begin{equation*}
\Phi_{\epsilon}(x,t) = u(x,t) - \varphi(x) - \frac{(t- t_{0})^{2}}{2 \epsilon}.
\end{equation*}
Write $\Phi_{\epsilon}(x,t) = u(x,t) - \Psi_{\epsilon}(x,t)$ and note that $\Psi_{\epsilon} \in C^{2,1}(\mathcal{I} \times [0,T])$.  Let $(x_{\epsilon}, t_{\epsilon})$ denote a maximum point of $\Phi_{\epsilon}$ its domain.  Since $0$ is a strict global maximum of $u(\cdot,t_{0}) - \varphi$, it follows that $t_{\epsilon} \to t_{0}$, $x_{\epsilon} \to 0$, and $u(x_{\epsilon},t_{\epsilon}) \to u(0,t_{0})$ as $\epsilon \to 0^{+}$.  Fix $\epsilon_{1} > 0$ such that $t_{\epsilon} > 0$ if $\epsilon \in (0,\epsilon_{1})$.  If there is a sequence $\epsilon_{n} \to 0$ such that $x_{\epsilon_{n}} \in I_{j}$ for some $j$ and each $n$, then we immediately obtain
\begin{equation*}
F_{j}(t_{\epsilon_{n}},x_{\epsilon_{n}},u(x_{\epsilon_{j}},t_{\epsilon_{j}}),\varphi_{x_{j}}(x_{\epsilon_{n}})) = F_{j}(t_{\epsilon_{n}}, x_{\epsilon_{n}}, u(x_{\epsilon_{j}},t_{\epsilon_{j}}),\Psi_{x_{j}}(x_{\epsilon_{n}},t_{\epsilon_{n}})) \leq 0.
\end{equation*} 
Sending $n \to \infty$, we recover $F_{j}(t_{0},0,u(x_{0},t_{0}),\varphi_{x_{j}}(0)) \leq 0$.  

It remains to consider the case when there is an $\epsilon_{2} > 0$ such that $x_{\epsilon} = 0$ for all $\epsilon \in (0,\epsilon_{2})$.  Fix such an $\epsilon$.  For each $j \in \{1,2,\dots,K\}$, the map $$(x,t) \mapsto u(x,t) - \varphi(x) - \frac{(t - t_{0})^{2}}{2 \epsilon}$$ defined in $\overline{I_{j}^{\delta}} \times [0,T]$ has a local maximum at $(0,t_{\epsilon})$.  Thus, 
\begin{align*}
\min \left\{ \min_{i} F_{i}(t_{\epsilon},0,u(0,t_{\epsilon}),\varphi_{x_{i}}(0)), \sum_{i = 1}^{K} \varphi_{x_{i}}(0) - B \right\} \leq 0.
\end{align*}
We conclude by sending $\epsilon \to 0^{+}$ and appealing to continuity of the functions $F_{1},\dots,F_{K}$.
\end{proof}

\section{Existence of Solutions of the Cauchy Problems} \label{A: cauchy_existence}

In this section, we prove the existence of solutions of \eqref{E: time} and \eqref{E: viscous_time}.  The main results proved herein follow:

\begin{theorem} \label{T: existence_time}  If $u_{0} \in \text{UC} \left(\mathcal{I} \right)$, then there is a $u \in \text{UC}(\mathcal{I} \times [0,T])$ solving \eqref{E: time}.  If, in addition, $u_{0} \in \text{Lip}(\mathcal{I})$, then $u \in \text{Lip}(\mathcal{I} \times [0,T])$, and $\text{Lip}(u)$ depends on $u_{0}$ only through $\text{Lip}(u_{0})$.  \end{theorem}

\begin{theorem}  \label{T: existence_time_viscous} Fix $\epsilon > 0$.  If $u_{0} \in \text{UC} \left( \mathcal{I} \right)$, then there is a $u^{\epsilon} \in \text{UC}(\mathcal{I} \times [0,T])$ solving \eqref{E: viscous_time}.  Moreover, if $[u_{0}]_{1} + [u_{0}]_{2} < \infty$, then there is a $C > 0$ depending only on $[u_{0}]_{1} + \epsilon[u_{0}]_{2}$ such that $\text{Lip}(u^{\epsilon}) \leq C$.     \end{theorem}  

We prove Theorem \ref{T: existence_time} by sending $\epsilon \to 0^{+}$ in Theorem \ref{T: existence_time_viscous}.  Therefore, the main thrust of this section is the proof of Theorem \ref{T: existence_time_viscous} and associated estimates.  

The proof of Theorem \ref{T: existence_time_viscous} is divided into three steps.  First, we use the estimate proved by von Below in \cite{von below} and Schaefer's fixed point theorem to obtain solutions of \eqref{E: viscous_time} when $u_{0}$ is a smooth function satisfying some compatibility conditions.  Next, we prove Lipschitz estimates when the initial data is sufficiently regular.  Finally, we approximate arbitrary initial data by smooth data and use the comparison principle to pass to the limit.    

Recall that in Remark \ref{R: half_relaxed_time} above we observed that an alternative proof of Theorem \ref{T: existence_time} can be obtained using the finite-difference scheme \eqref{E: cauchy_diff} and the method of half-relaxed limits.  

\subsection{Existence for Regular Data}  Here we obtain solutions of \eqref{E: viscous_time} using a priori H\"{o}lder estimates for linear parabolic equations on networks and Schaefer's fixed point theorem.  

To begin with, for a given $R > 0$, we let $\{\tilde{H}_{1}^{(R)},\dots,\tilde{H}_{K}^{(R)}\}$ take the form 
	\begin{equation*}
		\tilde{H}_{i}^{(R)}(t,x,p) = \psi^{(R)}(p) H_{i}(t,x,p) + (1 - \psi^{(R)}(p)) R,
	\end{equation*}
where $\psi^{(R)} : \mathbb{R} \to [0,1]$ is a smooth cut-off function satisfying $\psi^{(R)}(p) = 1$ if $|p| \leq \frac{R}{2}$ and $\psi^{(R)}(p) = 0$ if $|p| \geq R$.  Notice that $\{\tilde{H}_{1}^{(R)},\dots,\tilde{H}_{K}^{(R)}\}$ are bounded functions on their respective domains, and the assumptions \eqref{As: continuity} and \eqref{As: time_bound} continue to hold. 

The result is stated next:

\begin{prop} \label{P: regular_existence}  Suppose $a > 0$ and $u_{0} \in C^{3}(\mathcal{I})$ satisfies, for each $i \in \{1,2,\dots,K\}$,
	\begin{align}
		\epsilon u_{0,x_{i}x_{i}}(0) - H_{i}(0,0,u_{0,x_{i}}(0)) &= \epsilon u_{0,x_{1}x_{1}}(0) - H_{1}(0,0,u_{0,x_{1}}(0)) \label{E: compatibility_1} \\
		\sum_{i = 1}^{K} u_{0,x_{i}}(0) &= 0 \label{E: compatibility_2} \\
		[u_{0}]_{1} + [u_{0}]_{2} + [u_{0}]_{3} &< \infty
	\end{align}
Assume, in addition, that $R \geq 2 [u_{0}]_{1}$.  
Then there is a viscosity solution $u^{(a)}: \bigcup_{i = 1}^{K} \overline{I^{a}_{i}} \times [0,T] \to \mathbb{R}$ of the following equation:
	\begin{equation*}
		\left\{
			\begin{array}{r l}
				u^{(a)}_{t} - \epsilon u^{(a)}_{x_{i}x_{i}} + \tilde{H}^{(R)}_{i}(t,x,u^{(a)}_{x_{i}}) = 0 & \text{in} \, \, I_{i}^{a} \times (0,T) \\
				\sum_{i = 1}^{K} u^{(a)}_{x_{i}} = 0 & \text{on} \, \, \{0\} \times (0,T) \\
				u^{(a)} = u_{0} & \text{on} \, \, \bigcup_{i = 1}^{K} \overline{I_{i}^{a}} \times \{0\} \\
				u^{(a)} = \beta_{i} & \text{on} \, \, \{-a_{i}\} \times (0,T)
			\end{array}
		\right.
	\end{equation*}   
where the functions $\beta_{1},\dots,\beta_{K} : [0,T] \to \mathbb{R}$ are given by 
	\begin{equation} \label{E: compatibility_3}
		\beta_{i}(t) = u_{0}(-a_{i}) + \left(\epsilon u_{0,x_{i}x_{i}}(-a_{i}) - \tilde{H}^{(R)}_{i}(0,-a_{i},u_{0,x_{i}}(-a_{i}))\right)t.
	\end{equation}   
For each $i \in \{1,2,\dots,K\}$, the functions $u^{(a)}$, $u^{(a)}_{t}$, $u^{(a)}_{x_{i}}$, and $u^{(a)}_{x_{i}x_{i}}$ are H\"{o}lder continuous in $\overline{I_{i}^{a}} \times [0,T]$.  
\end{prop}   

A similar result has been obtained in \cite{achdou} starting with weak solutions in $L^{p}$ spaces.  

Our proof of Proposition \ref{P: regular_existence} follows the same general outline presented in \cite[Chapter 5]{ladyzenskaja}.  As in the fixed point arguments contained there, the next remark will play a significant role here.  For a proof, see, for example, \cite[Lemma 3.1]{ladyzenskaja}.

\begin{remark} \label{R: interpolation}  Suppose $I \subseteq \mathbb{R}$ is an open interval and $u : \overline{I} \times [0,T] \to \mathbb{R}$ is twice continuously differentiable in space and once continuously differentiable in time.  Then $u_{x}$ is $\frac{1}{2}$-H\"{o}lder continuous in time with a H\"{o}lder constant that only depends on $I$ and the suprema of $|u_{t}|$ and $|u_{xx}|$.    \end{remark}  

It will be convenient in what follows to use the semi-norms $[\cdot]_{\alpha}$ and $[\cdot]_{1 + \alpha}$ on functions with domain $\bigcup_{i = 1}^{K} \overline{I_{i}^{a}} \times [0,T]$, abusing the notation somewhat.  By this, we mean the semi-norms as defined in Subsection \ref{S: notation}, but with $\overline{I_{i}}$ replaced everywhere in the definitions with $\overline{I_{i}^{a}}$.  

\begin{proof}[Proof of Proposition \ref{P: regular_existence}]  First, for each $\alpha \in (0,1)$, define a norm on functions $v : \bigcup_{i = 1}^{K} \overline{I_{i}^{a}} \times [0,T] \to \mathbb{R}$ by 
	\begin{equation*}
		\|v\|_{\alpha} = [v]_{0} + [v]_{\alpha} + \max_{i} \, [v_{x_{i}}]_{i,0} + [v]_{1 + \alpha}.
	\end{equation*}
Let $V_{\alpha}$ be the Banach space of functions $v$ with $\|v\|_{\alpha} < \infty$.  We will find the solution as the fixed point of a certain operator on $V_{\alpha}$.  

Fix $\alpha \in (0,1)$.  We claim we can define a compact, continuous operator $T : V_{\alpha} \to V_{\alpha}$ so that $u = T(v)$ solves the equation
	\begin{equation} \label{E: approximate_equation}
            \left\{
            	\begin{array}{r l}
            		u_{t} - \epsilon u_{x_{i}x_{i}} + \tilde{H}^{(R)}_{i}(x,t,v_{x_{i}}) = 0 & \text{in} \, \, I_{i}^{a} \times (0,T) \\
            		\sum_{i = 1}^{K} u_{x_{i}} = 0 & \text{on} \, \, \{0\} \times (0,T) \\
            		u = u_{0} & \text{on} \, \, \bigcup_{i = 1}^{K} \overline{I_{i}^{a}} \times \{0\} \\
            		u = \beta_{i} & \text{on} \, \, \{-a\} \times (0,T)
            	\end{array}
            \right.
           \end{equation}
 Indeed, since $[v]_{1 + \alpha} < \infty$ and $\tilde{H}^{(R)}_{i}(0,0,u_{0,x_{i}}(0)) = H_{i}(0,0,u_{0,x_{i}}(0))$ for all $i$ by the choice of $R$, the compatibility conditions \eqref{E: compatibility_1}, \eqref{E: compatibility_2}, and \eqref{E: compatibility_3} imply the result of \cite{von below} is applicable.  In particular, a bounded solution $u$ of \eqref{E: approximate_equation} exists and the functions $u_{t},u_{x_{1}x_{1}},\dots,u_{x_{K}x_{K}}$ are bounded and continuous.  Thus, Remark \ref{R: interpolation} implies $[u]_{2} < \infty$, and $u \in V_{\alpha}$ follows.  Arguing as in Theorem \ref{T: viscous_comparison_time}, we see that $u$ is uniquely determined.  Therefore, $T$ is well-defined.
 
We claim that $T$ is compact and continuous.  Suppose $(v_{n})_{n \in \mathbb{N}} \subseteq V_{\alpha}$ and $\|v_{n}\|_{\alpha} \leq C$ independently of $n$.  Let $u^{(n)} = T(v_{n})$.  The main result of \cite{von below} implies $u^{(n)},u^{(n)}_{t},u^{(n)}_{x_{1}x_{1}},\dots,u^{(n)}_{x_{K}x_{K}}$ are bounded continuous functions with bounds depending on $(v_{n})_{n \in \mathbb{N}}$ only through the constant $C$.  Thus, Remark \ref{R: interpolation} implies $[u^{(n)}]_{1} + [u^{(n)}]_{2}$ is uniformly bounded.  Since $\alpha < 1$ and $\{u^{(n)}(\cdot,0)\}_{n \in \mathbb{N}} = \{u_{0}\}$, it follows that $(u^{(n)})_{n \in \mathbb{N}}$ is relatively compact in $V_{\alpha}$.  Therefore, by definition, $T$ is compact.  Since solutions of \eqref{E: approximate_equation} in $V_{\alpha}$ are unique, if, in addition, $v_{n} \to v$ in $V_{\alpha}$, then it is straightforward to check that $T(v)$ is the only possible subsequential limit point of $(u^{(n)})_{n \in \mathbb{N}}$.  In particular, $u^{(n)} \to T(v)$ in $V_{\alpha}$, which proves that $T$ is continuous.  

Finally, we check the hypotheses of Schaefer's fixed point theorem (cf.\ \cite[Section 9.2.2]{evans}).  Recall that we need to find a $C >0$ such that if $u \in V_{\alpha}$ satisfies $u= \sigma T(u)$ for some $\sigma \in [0,1]$, then $\|u\|_{\alpha} \leq C$.
Indeed, arguing as in Proposition \ref{P: linearized_argument} below, we see that $u_{t}$ is bounded independently of $\sigma$.  Since $\{\tilde{H}_{1}^{(R)},\dots,\tilde{H}_{K}^{(R)}\}$ are bounded functions, the equation implies $u_{x_{i}x_{i}}$ is also bounded independently of $\sigma$ and $i$.  From this, we obtain a bound on $[u]_{2}$ by Remark \ref{R: interpolation}.  Finally, the regularity of $u_{0}$ and the uniform bound on $[u]_{2}$ gives a bound on $\max_{i} \, [u_{x_{i}}]_{i,0} + [u]_{1 + \alpha}$, and this together with the uniform bound on $u_{t}$ provides one for $[u]_{0} + [u]_{\alpha}$.  It follows that $\|u\|_{\alpha}$ is bounded independently of $\sigma$.  

By Schaefer's theorem, we conclude there is a $u^{(a)} \in V_{\alpha}$ such that $T(u^{(a)}) = u^{(a)}$.  The regularity of $u^{(a)}$ and its derivatives follows directly from the result of \cite{von below}.  
 \end{proof}  
 
%

\subsection{A priori bounds}  In the previous subsection, we showed that smooth initial data have smooth solutions, provided certain compatibility conditions are satisfied.  Now we prove some a priori estimates satisfied by these solutions.  

We start with a bound on the time derivative, which follows from \eqref{As: time_bound} and the maximum principle.

\begin{prop} \label{P: linearized_argument} Let $a> 0$.  If $u_{0}$ and $R > 0$ satisfy the hypotheses of Proposition \ref{P: regular_existence}, and if $u^{(a)}$ is the solution obtained therein, then 
	\begin{equation*}
		u^{(a)}_{t}(x,0) = \epsilon u_{0,x_{i}x_{i}}(x) - \tilde{H}^{(R)}_{i}(x,0,u_{0,x_{i}}(x)) \quad \text{if} \, \, x \in \overline{I_{i}^{a}}, \, \, i \in \{1,2,\dots,K\},
	\end{equation*}
and, for each $(x,t) \in \mathcal{I} \times [0,T]$,
	\begin{equation} \label{E: time_derivative_bound}
		|u^{(a)}_{t}(x,t)| \leq [u^{(a)}_{t}(\cdot,0)]_{0} + Dt.
	\end{equation}
\end{prop}  

\begin{proof}  The claim concerning $u^{(a)}_{t}(\cdot,0)$ follows from the regularity established in Proposition \ref{P: regular_existence}.

Given $\zeta \in (0,T)$, define $v^{\zeta} : \bigcup_{i = 1}^{K} \overline{I_{i}^{a}} \times [0,T - \zeta] \to \mathbb{R}$ by $v^{\zeta}(x,t) = \frac{u^{(a)}(x,t + \zeta) - u^{(a)}(x,t)}{\zeta}$.  An immediate computation shows $v^{\zeta}$ is a classical solution of a linear parabolic equation of the following form:
	\begin{equation*}
		\left\{
			\begin{array}{r l}
				v^{\zeta}_{t} - \epsilon v^{\zeta}_{x_{i}x_{i}} + b_{i}^{\zeta}(x,t) v^{\zeta}_{x_{i}} - g_{i}^{\zeta}(x,t) = 0 & \text{in} \, \, I^{a}_{i} \times (0,T - \zeta) \\
				\sum_{i = 1}^{K} v^{\zeta}_{x_{i}} = 0 & \text{on} \, \, \{0\} \times (0,T - \zeta) \\
				v^{\zeta} = u^{(a)}_{t}(\cdot,0) & \text{on} \, \, \{-a_{i}\} \times (0,T - \zeta)
			\end{array}
		\right.
	\end{equation*}
Notice that $\{b_{1}^{\zeta},\dots,b_{K}^{\zeta}\}$ and  $\{g_{1}^{\zeta},\dots,g_{K}^{\zeta}\}$ are bounded functions by \eqref{As: continuity} and \eqref{As: time_bound}.  Specifically, the functions $\{g_{1}^{\zeta},\dots,g_{K}^{\zeta}\}$ are bounded above and below by $D$ and $-D$, respectively, independently of the choice of $\zeta$.  

We claim that if $ (x,t) \in \bigcup_{i = 1}^{K} \overline{I_{i}^{a}} \times [0,T-\zeta]$, then
	\begin{equation} \label{E: linear_upper_bound}
		v^{\zeta}(x,t) - Dt \leq \sup \left\{v^{\zeta}(x,0) \, \mid \, x \in \bigcup_{i = 1}^{K} \overline{I_{i}^{a}}\right\}.
	\end{equation}
To see this, fix $K > 0$ strictly greater than the suprema of the functions $\{|b_{1}^{\zeta}|,\dots,|b_{K}^{\zeta}|\}$ and notice that if $\delta > 0$, then the function $(x,t) \mapsto v^{\zeta}(x,t) - \delta x - (D + K\delta)t$ cannot attain its maximum in $\bigcup_{i = 1}^{K} (\overline{I^{a}_{i}} \setminus \{-a_{i}\})\times (0,T - \zeta]$.  Recalling that $v^{\zeta}$ is constant on $\bigcup_{i = 1}^{K} \{-a_{i}\} \times [0,T - \zeta]$ and sending $\delta \to 0^{+}$, we recover \eqref{E: linear_upper_bound}.  

Notice that for each $\zeta' \in (0,T)$, the H\"{o}lder continuity of $u^{(a)}_{t}$ implies $v^{\zeta} \to u^{(a)}_{t}$ uniformly in $\bigcup_{i = 1}^{K} \overline{I_{i}^{a}} \times [0,T - \zeta']$ as $\zeta \to 0^{+}$.  Thus, after sending $\zeta \to 0^{+}$ in \eqref{E: linear_upper_bound}, we find $u^{(a)}_{t}(x,t) \leq [u^{(a)}_{t}(\cdot,0)]_{0} + Dt$.  To see that $u^{(a)}_{t}(x,t) \geq -( [u^{(a)}_{t}(\cdot,0)]_{0} + Dt)$, we repeat the previous argument, replacing $v^{\zeta}$ by $-v^{\zeta}$. \end{proof}

Next, we leverage the bound on the time derivative to obtain a matching bound on the first order space derivatives.

\begin{prop} \label{P: space_lipschitz_estimate}  If $u^{(a)}$ is the solution obtained in Proposition \ref{P: regular_existence} and we define $C_{1} = [u_{t}(\cdot,0)]_{0} + DT$, then there is an $L > 0$ independent of $a$, depending on $u_{0}$ only through $[u_{0}]_{1} + \epsilon [u_{0}]_{2}$, and such that if $a > 2\left(2C_{1} + 1\right)T$ and $R \geq 2KL$, then
	\begin{equation} \label{E: space_lipschitz_estimate}
		|u^{(a)}(x,t) - u^{(a)}(y,t)| \leq KLd(x,y) \quad \text{if} \, \, d(x,0), d(y,0) \leq \frac{a}{2}, \, \, t \in [0,T].
	\end{equation}
\end{prop}  

\begin{proof}  First, let $L_{0} = [u_{0}]_{1} + \epsilon [u_{0}]_{2}$.  By \eqref{As: coercive}, there is an $L_{1} \geq 1$ such that
	\begin{equation*}
		-(C_{1} + 1) + H_{i}(t,x,p) \geq 1 \quad \text{if} \, \, |p| \geq L_{1}, \, \, i \in \{1,2,\dots,K\}.
	\end{equation*}
Let $L_{2} = L_{0} + L_{1}$.  Notice that since $C_{1}$ depends on $u_{0}$ only through $[u_{0}]_{1} + \epsilon [u_{0}]_{2}$, it follows that $L_{2}$ depends on $u_{0}$ only through that quantity.  Assume in what follows that $R \geq 2K(3L_{2} + 1)$.  

Fix $i \in \{1,2,\dots,K\}$ and $(x,t) \in \overline{I_{i}^{a}} \times (0,T)$ such that $d(x,0) \leq \frac{a}{2}$.  Define the test function $\varphi : \mathcal{I} \to \mathbb{R}$ exactly as in \eqref{E: knife_test_function}, but with $u^{(a)}(x,t)$ in place of $u(x)$ and $3L_{2} + 1$ in place of $L$.  Finally, define $w : \mathcal{I} \times [0,t] \to \mathbb{R}$ by 				
	\begin{equation*}
		w(y,s) = \varphi(y) + (C_{1} + 1)(t - s).
	\end{equation*}

We claim that the function $(y,s) \mapsto u^{(a)}(y,s) - w(y,s)$ defined in $\bigcup_{i  =1}^{K} \overline{I_{i}^{a}} \times [0,t]$ is maximized at $(x,t)$.  First, note that $u^{(a)}(x,t) - w(x,t) = 0$.  Moreover, if $s < t$, then the inequality $[u_{t}^{(a)}]_{0} \leq C_{1}$ implies
	\begin{equation*}
		u^{(a)}(x,s) - w(x,s) = u^{(a)}(x,s) - u^{(a)}(x,t) - (C_{1} + 1)(t - s) \leq - (t -s) < 0.
	\end{equation*}  
Therefore, the maximum does not occur at a point of the form $(x,s)$, where $s \in [0,t)$.  
	
If $(y,s)$ is the maximum of $u^{(a)} - w$ in $\bigcup_{i = 1}^{K} \overline{I^{a}_{i}} \times [0,t]$, $y \in \bigcup_{i = 1}^{K} I_{i}^{a} \setminus \{x\}$, and $s \in (0,t]$, then, in view of the choice of $R$, the equation yields
	\begin{equation*}
		-(C_{1} + 1) + H_{i}(t,x, u (3L_{2} + 1)) \leq 0 \quad \text{for some} \, \, u \in \{-1,1, K,-K\},
	\end{equation*}
contradicting the choice of $L_{1}$.  We get a contradiction similarly in the case when $y = 0$ and $s \in (0,t]$.  

If $(y,s)$ is the maximum, $y \neq x$, and $s = 0$, then the inequalities $[u_{0}]_{1} \leq L_{2}$ and $[u^{(a)}_{t}]_{0} \leq C_{1}$ yield the following
	\begin{equation*}
		0 \leq u^{(a)}(y,0) - w(y,0) \leq u_{0}(y) - u_{0}(x) - (3L_{2}+1) d(x,y) < 0,
	\end{equation*}
which is a contradiction.  

Finally, if $(y,s)$ is the maximum and $y = -a_{j}$ for some $j$, then the assumption $d(x,0) \leq \frac{a}{2}$, the inequalities $[u_{0}]_{1} \leq L_{2}$ and $[u^{(a)}_{t}]_{0} \leq C_{1}$, and the assumption $a > 2\left(2C_{1} + 1\right)T$ all come together to imply the following:
		\begin{align*}
		u^{(a)}(-a_{j},s) &\geq w(-a_{j},s) \\
			&= \varphi(-a_{j}) + (C_{1} + 1)(t - s) \\
			&\geq u^{(a)}(x,t) + (3L_{2} + 1)\left(\frac{a}{2}\right) + (C_{1} + 1)(t - s) \\
			&\geq \left(u_{0}(x) - u_{0}(-a_{j})\right) - C_{1}(t + s) + u^{(a)}(-a_{j},s) + (3L_{2} + 1) \left(\frac{a}{2}\right) \\
			&\quad\quad  + (C_{1} + 1)(t - s) \\
			&\geq -L_{2} \left(\frac{3a}{2}\right) - (2C_{1} + 1) T + (3L_{2} + 1) \left(\frac{a}{2}\right) + u^{(a)}(-a_{j},s) \\
			&> u^{(a)}(-a_{j},s),
	\end{align*}
which is another contradiction.  Therefore, the function $(y,s) \mapsto u^{(a)}(y,s) - w(y,s)$ is maximized in $\bigcup_{i = 1}^{K} \overline{I_{i}^{a}} \times [0,t]$ at the point $(x,t)$.  

Thus, restricting to points $(y,s) = (y,t)$, we find
	\begin{equation*}
		u^{(a)}(y,t) - u^{(a)}(x,t) \leq K (3L_{2} + 1) |x - y| \quad \text{if} \, \, y \in \overline{I_{j}}.  
	\end{equation*}
After setting $L = 3L_{2} + 1$, we conclude that \eqref{E: space_lipschitz_estimate} holds.    \end{proof}

\subsection{Viscosity solutions}  Now we prove Theorems \ref{T: existence_time} and \ref{T: existence_time_viscous}.

To prove these, we need to ensure that we can approximate the initial datum with a regular function that satisfies the compatibility conditions \eqref{E: compatibility_1} and \eqref{E: compatibility_2}.  That is the purpose of the next two results.  

\begin{lemma} \label{L: calculus_exercise_compatibility}  Suppose $p : (-\infty,0] \to \mathbb{R}$ is a thrice continuously differentiable function for which there is a constant $C_{p} > 0$ such that, for each $x \in (-\infty,0]$,
\begin{equation*}
|p'(x)| + |p''(x)| \leq C_{p}
\end{equation*}
and $\sup \left\{|p'''(x)| \, \mid \, x \in (-\infty,0]\right\} < \infty$.  
Let $b \in \mathbb{R}$.  There is a universal constant $C'_{p} > 0$ depending only on $C_{p}$ and $b$ such that if $\zeta > 0$, then there is a thrice continuously differentiable function $p_{\zeta} : (-\infty,0] \to \mathbb{R}$ such that $p_{\zeta}(0) = p(0)$, $p_{\zeta}'(0) = p'(0)$, $p_{\zeta}''(0) = b$, $\sup \left\{|p_{\zeta}'''(x)| \, \mid \, x \in (-\infty,0]\right\} < \infty$, and, for each $x \in (-\infty,0]$,
	\begin{align*}
		|p_{\zeta}'(x)| + |p_{\zeta}''(x)| &\leq C_{p}'\\
		|p_{\zeta}(x) - p(x)| &\leq C_{p}' \zeta^{2}
	\end{align*}
\end{lemma}

\begin{proof}  Given $\zeta > 0$, choose a smooth function $\varphi_{\zeta} : (-\infty,0] \to \mathbb{R}$ such that 
\begin{align*}
\varphi_{\zeta}(x) = 0 \quad \text{if} \, \, x \in (-\infty,-2\zeta]&, \quad
\varphi_{\zeta}(x) = 1 \quad \text{if} \, \, x \in [-\zeta,0], \\
\max \left\{ \zeta |\varphi_{\zeta}'(x)|, \zeta^{2} |\varphi_{\zeta}''(x)| \right\} &\leq C_{0} \quad \text{if} \, \, x \in (-\infty,0],
\end{align*}
where $C_{0} \geq 1$ is a universal constant independent of $\zeta$, $p$, $C_{p}$, and $b$.  

Define $q(x) = p(0) + p'(0)x + \frac{b x^{2}}{2}$ and then let $p_{\zeta} : (-\infty,0] \to \mathbb{R}$ be given by
\begin{equation*}
p_{\zeta}(x) = (1 - \varphi_{\zeta}(x))p(x) + \varphi_{\zeta}(x)q(x).
\end{equation*}
The choice of $\varphi_{\zeta}$ implies $p_{\zeta}(0) = p(0)$, $p_{\zeta}'(0) = p'(0)$, and $p_{\zeta}''(0) = b$.  Moreover, $\sup \left\{ |p'''_{\zeta}(x)| \, \mid \, x \in(-\infty,0]\right\} < \infty$ holds.

Differentiating $p_{\zeta}$, we find
\begin{align*}
	p_{\zeta}'(x) &= (1 - \varphi_{\zeta}(x))p'(x) + \varphi_{\zeta}(x) q'(x) + \varphi_{\zeta}'(x)(q(x) - p(x)), \\
	p_{\zeta}''(x) &= (1 - \varphi_{\zeta}(x))p''(x) + \varphi_{\zeta}(x) q''(x) + 2\varphi_{\zeta}'(x)(q'(x) - p'(x)) \\
		&\qquad+ \varphi_{\zeta}''(x)(q(x) - p(x)).
\end{align*}
Thus, the regularity of $p$ and the definition of $\varphi_{\zeta}$ imply the desired bounds by Taylor expansion at $0$.  \end{proof}  

%
%

Now we use Lemma \ref{L: calculus_exercise_compatibility} to show how to approximate a $C^{3}(\mathcal{I})$ function by one that satisfies the compatibility conditions.  

\begin{prop} \label{P: compatibility_conditions_stuff} Suppose $u_{0} \in C^{3}(\mathcal{I})$ satisfies $\sum_{i = 1}^{K} u_{0,x_{i}}(0) = 0$ and $[u_{0}]_{1} + [u_{0}]_{2} + [u_{0}]_{3} < \infty$.  
Then there is a universal constant $C' > 0$ depending only on $[u_{0}]_{1} + [u_{0}]_{2}$ such that for all $\zeta > 0$, there is a $u_{0}^{\zeta} \in C^{3}(\mathcal{I})$ satisfying the following conditions:
	\begin{itemize}
		\item[(i)] $[u_{0}^{\zeta}]_{1} + [u_{0}^{\zeta}]_{2} + [u_{0}^{\zeta}]_{3} < \infty$ 
		\item[(ii)] For each $i \in \{1,2,\dots,K\}$,
			\begin{align*}
				u^{\zeta}_{0,x_{i}}(0) &= u_{0,x_{i}}(0) \\
				-\epsilon u^{\zeta}_{0,x_{i}x_{i}}(0) + H_{i}(0,0,u^{\zeta}_{0,x_{i}}(0)) &= - \epsilon u^{\zeta}_{0,x_{1}x_{1}}(0) + H_{1}(0,0,u^{\zeta}_{0,x_{1}}(0)) \\
			\end{align*}
		\item[(iii)] For each $i \in \{1,2,\dots,K\}$ and each $x \in \overline{I_{i}}$,
			\begin{align}
				|u^{\zeta}_{0,x_{i}}(x)| + |u^{\zeta}_{0,x_{i}x_{i}}(x)| &\leq C' \label{E: key_uniform_bound} \\
				|u^{\zeta}_{0}(x) - u_{0}(x)| &\leq C'\zeta^{2} \label{E: easy_convergence}
			\end{align}
	\end{itemize}
\end{prop}  

\begin{proof}  Define $\{b_{1},\dots,b_{K}\}$ by $b_{1} = 1$ and
\begin{equation}
b_{i} = \epsilon^{-1}\left(\epsilon - H_{1}(0,0,u_{0,x_{1}}(0)) + H_{i}(0,0,u_{0,x_{i}}(0))\right).
\end{equation}
Notice that this immediately implies $\{b_{1},\dots,b_{K}\}$ satisfy, for each $i$,
\begin{equation}
-\epsilon b_{i} + H{i}(0,0,u_{0,x_{i}}(0)) = -\epsilon b_{1} + H_{1}(0,0,u_{0,x_{1}}(0)).
\end{equation}

Now apply Lemma \ref{L: calculus_exercise_compatibility} to obtain functions $\{\psi^{\zeta,(1)},\dots,\psi^{\zeta,(K)}\}$ and a constant $C' > 0$ so that, for each $i \in \{1,2,\dots,K\}$, $\psi^{\zeta,(i)}$ has domain $(-\infty,0)$ and the following relations hold:
\begin{align}
\sup \left\{ |\psi^{\zeta,(i)}_{x}(x)| + |\psi^{\zeta,(i)}_{xx}(x)| \, \mid \, x \leq 0 \right\} &\leq C' \\
\psi^{\zeta,(i)}(0) &= u_{0}(0) \label{E: function} \\
\psi^{\zeta,(i)}_{x}(0) &= u_{0,x_{i}}(0) \\
\psi^{\zeta,(i)}_{xx}(0) &= b_{i} \\
\sup \left\{ |\psi^{\zeta,(i)}(x) - u_{0}(x)| \, \mid \, x \in \overline{I_{i}}\right\} &\leq C' \zeta^{2}
\end{align}
By construction, $\{\psi^{\zeta,(1)},\dots,\psi^{\zeta,(K)}\}$ come together to form a function $u^{\zeta}_{0} \in C^{3}(\mathcal{I})$ with the desired properties.  
\end{proof}  

In the proof that follows, we will not use the $\epsilon$ superscript to denote solutions of \eqref{E: viscous_time}.  Since we are only dealing with \eqref{E: viscous_time} and not \eqref{E: time} in the proof, we hope this will not cause too much confusion.

\begin{proof}[Proof of Theorem \ref{T: existence_time_viscous}]  First, assume $u_{0} \in C^{3}(\mathcal{I})$ and $[u_{0}]_{1} + [u_{0}]_{2} + [u_{0}]_{3} < \infty$.  For $\zeta > 0$ sufficiently small, let $u^{\zeta}_{0}$ be the function obtained from Proposition \ref{P: compatibility_conditions_stuff}, and fix $R \geq 2C'$, where $C'$ is the constant defined in the proposition.  For each $a > 0$, let $u^{(a),\zeta}$ be the solution of \eqref{E: approximate_equation}  with initial datum $u_{0}^{\zeta}$ obtained in Proposition \ref{P: regular_existence}.  

By Propositions \ref{P: linearized_argument} and \ref{P: space_lipschitz_estimate} and the uniform bound \eqref{E: key_uniform_bound}, there are constants $B, L, a_{0} > 0$, all independent of $\zeta$, such that if $a \geq a_{0}$ and $R \geq 2KL$, then $[u_{t}^{(a),\zeta}]_{0} \leq B$ and \eqref{E: space_lipschitz_estimate} holds with $u^{(a)} = u^{(a),\zeta}$.  Henceforth, assume $R \geq 2KL$.  

The estimates obtained in the previous paragraph imply we can fix a sequence $(a_{n})_{n \in \mathbb{N}} \subseteq [a_{0},\infty)$ and a function $u^{\zeta} : \mathcal{I} \times [0,T] \to \mathbb{R}$ such that $\lim_{n \to \infty} a_{n} = \infty$ and $u^{\zeta} = \lim_{n \to \infty} u^{(a_{n}),\zeta}$ locally uniformly in $\mathcal{I} \times [0,T]$.  The local uniform convergence and the stability of viscosity solutions together imply $u^{\zeta}$ is a solution of \eqref{E: viscous_time} with Hamiltonians $\{\tilde{H}^{(R)}_{1},\dots,\tilde{H}^{(R)}_{K}\}$ and initial datum $u_{0}^{\zeta}$.  

Since $\lim_{n \to \infty} a_{n} = \infty$, \eqref{E: space_lipschitz_estimate} shows that $u^{\zeta}$ satisfies $\text{Lip}(u^{\zeta}(\cdot,t)) \leq KL$ for all $t \in [0,T]$.  
Thus, as $\tilde{H}^{(R)}_{i}(t,x,p) = H_{i}(t,x,p)$ for all $|p| \leq KL$, it follows that $u^{\zeta}$ is actually a solution of \eqref{E: viscous_time} with the Hamiltonians $\{H_{1},\dots,H_{K}\}$. By Theorem \ref{T: viscous_comparison_time}, we deduce that the limit is unique and, in fact, $u^{\zeta} = \lim_{a \to \infty} u^{(a),\zeta}$ locally uniformly in $\mathcal{I} \times [0,T]$.

Finally, we  send $\zeta \to 0^{+}$.  Since $u^{\zeta}_{0} \to u_{0}$ uniformly in $\mathcal{I}$ as $\zeta \to 0^{+}$, Remark \ref{R: contractivity} implies $(u^{\zeta})_{\zeta > 0}$ is uniformly Cauchy in $\mathcal{I} \times [0,T]$.  In particular, $u = \lim_{\zeta \to 0^{+}} u^{\zeta}$ exists uniformly in $\mathcal{I} \times [0,T]$ and the stability of viscosity solutions implies $u$ solves \eqref{E: viscous_time} with initial datum $u_{0}$.  

Since $L$ and $B$ were independent of $\zeta$, the uniform convergence $u^{\zeta} \to u$ implies $\text{Lip}(u) \leq B + L$.

To remove the $C^{3}(\mathcal{I})$ assumption, we argue by approximation.  That is, if $u_{0} \in C^{1}(\mathcal{I})$ and $[u_{0}]_{1} + [u_{0}]_{2} < \infty$, we obtain the solution $u$ of \eqref{E: viscous_time} and show that it is in $\text{Lip}(\mathcal{I} \times [0,T])$ by approximating $u_{0}$ with functions $(u_{0,n}) \subseteq C^{3}(\mathcal{I})$ such that $u_{0,n} \to u_{0}$ uniformly in $\mathcal{I}$ and $\sup \left\{ [u_{0,n}]_{1} + [u_{0,n}]_{2} \, \mid \, n \in \mathbb{N} \right\} < \infty$.  Since the proof that it is possible to do this is very similar to some of the arguments presented in Section \ref{S: useful_approximations}, we omit it.  

Finally, if $u_{0} \in UC(\mathcal{I})$, then, arguing as in Remark \ref{R: BUC_approximation} below, we can fix a sequence $(u_{0}^{(n)})_{n \in \mathbb{N}} \in C^{1}(\mathcal{I})$ satisfying $[u_{0}^{(n)}]_{1} + [u_{0}^{(n)}]_{2} < \infty$ for each $n$ and such that $u_{0}^{(n)} \to u_{0}$ uniformly in $\mathcal{I}$ as $n \to \infty$.  By the previous step, we can let $u^{(n)}$ be the solution of \eqref{E: viscous_time} with initial datum $u_{0}^{(n)}$, and Remark \ref{R: contractivity} shows that $(u^{(n)})_{n \in \mathbb{N}}$ is uniformly Cauchy in $\mathcal{I} \times [0,T]$.  Therefore, as before, the limit $u = \lim_{n \to \infty} u^{(n)}$ is a continuous viscosity solution of \eqref{E: viscous_time}.  In fact, $u \in UC(\mathcal{I} \times [0,T])$, being the uniform limit of such functions.  \end{proof}

\subsection{Existence of solutions of \eqref{E: time}}  Finally, we establish the existence of solutions of \eqref{E: time}.  Here, as in the error analysis, we invoke Proposition \ref{P: lipschitz_approximation}.

\begin{proof}[Proof of Theorem \ref{T: existence_time}] First, assume $u_{0} \in \text{Lip}(\mathcal{I})$.  By Proposition \ref{P: lipschitz_approximation} below, there is a family $(v_{0}^{\epsilon})_{\epsilon > 0} \subseteq C^{1} \left(\mathcal{I} \right)$ such that $\lim_{\epsilon \to 0^{+}} [v_{0}^{\epsilon} - u_{0}]_{0} =0$ and $\sup\{[v_{0}^{\epsilon}]_{1} + \epsilon [v_{0}^{\epsilon}]_{2} \, \mid \, \epsilon > 0\} \leq C'$, where $C'$ only depends on $\text{Lip}(u_{0})$.

For each $\epsilon > 0$, let $v^{\epsilon}$ be the solution of \eqref{E: viscous_time} with initial datum $v_{0}^{\epsilon}$.  Since $[v_{0}^{\epsilon}]_{1} + \epsilon [v_{0}^{\epsilon}]_{2}$ is bounded, Theorem \ref{T: existence_time_viscous} implies there is an $L' > 0$ depending on $C'$, but not $\epsilon$, such that $\text{Lip}(v^{\epsilon}) \leq L'$.

In view of the uniform Lipschitz estimate, we can fix $(\epsilon_{n})_{n \in \mathbb{N}}$ and a function $u : \mathcal{I} \times [0,T] \to \mathbb{R}$ such that $\lim_{n \to \infty} \epsilon_{n} = 0$ and $u = \lim_{n \to \infty} v^{\epsilon_{n}}$.  By the stability of viscosity solutions, $u$ solves \eqref{E: time} with initial datum $u_{0}$.  In fact, Theorem \ref{T: inviscid_comparison_time} implies $u$ is independent of the choice of subsequence, and, thus, $u = \lim_{\epsilon \to 0^{+}} v^{\epsilon}$.  Moreover, $\text{Lip}(u) \leq L'$.  

In general, if $u_{0} \in UC(\mathcal{I})$, then there is a sequence $(u_{0}^{(n)})_{n \in \mathbb{N}} \subseteq \text{Lip}(\mathcal{I})$ such that $u_{0}^{(n)} \to u_{0}$ uniformly in $\mathcal{I}$ as $n \to \infty$.  (See Remark \ref{R: BUC_approximation}.)  Let $u^{(n)}$ denote the solution of \eqref{E: time} with initial datum $u_{0}^{(n)}$.  By Remark \ref{R: contractivity}, $(u^{(n)})_{n \in \mathbb{N}}$ is uniformly Cauchy in $\mathcal{I} \times [0,T]$.  Invoking stability of viscosity solutions, we conclude that the limit $u = \lim_{n \to \infty} u^{(n)}$ is a solution of \eqref{E: time} with initial datum $u_{0}$.  Moreover, as a uniform limit of such functions, $u \in UC(\mathcal{I} \times [0,T])$.  \end{proof}    

\subsection{A useful approximation result} \label{S: useful_approximations} In the error analysis of Section \ref{S: cauchy_problem}, we used the following result:

\begin{prop} \label{P: lipschitz_approximation} Let $u_{0} \in \text{Lip}\left( \mathcal{I} \right)$.  For each $\epsilon >0$, there is a $v_{0}^{\epsilon} \in C^{1} \left(\mathcal{I}\right)$ and a universal constant $C > 0$ such that:
	\begin{align*}
		[v_{0}^{\epsilon} - u_{0}]_{0} &\leq \text{Lip}(u_{0})\epsilon \\
		[v_{0}^{\epsilon}]_{1} &\leq C \text{Lip}(u_{0}) \\
		[v_{0}^{\epsilon}]_{2} &\leq C \epsilon^{-1} \text{Lip}(u_{0})
	\end{align*}
Moreover, $v_{0}^{\epsilon}$ can be chosen in such a way that both $v_{0}^{\epsilon}(0) = u_{0}(0)$ and $\sum_{i = 1}^{K} v^{\epsilon}_{0,x_{i}}(0) = 0$.  
\end{prop}  

The same method used to prove Proposition \ref{P: lipschitz_approximation} below can be used to establish more general approximation results for functions on $\mathcal{I}$ with varying degrees of regularity.  We will not expound on those here.  However, since we use an approximation result for functions in $\text{UC}(\mathcal{I})$ in the proof of Theorems \ref{T: existence_time} and \ref{T: existence_time_viscous}, we include its statement as a remark:

\begin{remark} \label{R: BUC_approximation}  Arguing as in the proof that follows, we can show that if $u_{0} \in \text{UC}(\mathcal{I})$, then there is a sequence of functions $(u^{(n)}_{0})_{n \in \mathbb{N}} \subseteq C^{2}(\mathcal{I})$ satisfying $[u^{(n)}_{0}]_{1} + [u^{(n)}_{0}]_{2} < \infty$, $\sum_{i = 1}^{K} u_{0,x_{i}}^{(n)}(0) = 0$, and such that 
	\begin{equation*}
		[u_{0}^{(n)} - u_{0}]_{0} \leq \omega(2 n^{-1}),
	\end{equation*}
where $\omega$ is the modulus of continuity of $u_{0}$ in $\mathcal{I}$.  \end{remark}

\begin{proof}[Proof of Proposition \ref{P: lipschitz_approximation}]  First, given $\epsilon > 0$, let $\varphi_{\epsilon}$ be as in the proof of Lemma \ref{L: calculus_exercise_compatibility}.  Additionally, let $\rho : \mathbb{R} \to [0,\infty)$ be a smooth symmetric function supported in $(-1,1)$ and satisfying $\int_{-\infty}^{\infty} \rho(x) \, dx = 1$.  

Define $\tilde{\psi}_{i}^{\epsilon} : \overline{I_{i}} \to \mathbb{R}$ by $\tilde{\psi}_{i}^{\epsilon}(x) = \epsilon^{-1} \int_{-\infty}^{\infty} u_{0,i}(y) \rho(\epsilon^{-1}(x - y)) \, dy$, where $u_{0,i}$ is given by $u_{0,i}(x) = u_{0}(x_{i})$ if $x < 0$ and $u_{0,i}(x) = u_{0}(0)$, otherwise.  Recall the following well-known properties of $\tilde{\psi}_{i}^{\epsilon}$:
	\begin{align*}
		\sup \left\{|\tilde{\psi}_{i}^{\epsilon}(x) - u_{0}(x)| \, \mid \, x \in I_{i} \right\} &\leq \text{Lip}(u_{0})\epsilon \\
		\sup \left\{ |\tilde{\psi}_{i,x_{i}}^{\epsilon}(x)| \, \mid \, x \in I_{i} \right\} &\leq \text{Lip}(u_{0}) \\
		\sup \left\{|\tilde{\psi}_{i,x_{i}x_{i}}^{\epsilon}(x)| \, \mid \, x \in I_{i} \right\} &\leq C \text{Lip}(u_{0}) \epsilon^{-1} 
	\end{align*}    
We proceed by combining $\{\tilde{\psi}^{\epsilon}_{1},\dots,\tilde{\psi}^{\epsilon}_{K}\}$ into a function on $\mathcal{I}$.  

Define $v_{0}^{\epsilon} : \mathcal{I} \to \mathbb{R}$ by 
	\begin{equation*}
		v_{0}^{\epsilon}(x) = (1 - \varphi_{\epsilon}(x)) \tilde{\psi}^{\epsilon}_{i}(x) + \varphi_{\epsilon}(x) u_{0}(0) \quad \text{if} \, \, x \in \overline{I_{i}}, \, \, i \in \{1,2,\dots,K\}.
	\end{equation*}
Observe that $\min\{|\varphi_{\epsilon}'(x)|, |\varphi_{\epsilon}''(x)|\} > 0$ only if $x \in [-2\epsilon, \epsilon]$.  Moreover, for such $x$, the following inequality holds:
	\begin{equation*}
		|\tilde{\psi}^{\epsilon}_{i}(x) - u_{0}(0)| \leq |\tilde{\psi}^{\epsilon}_{i}(x) - u_{0}(x)| + |u_{0}(x) - u_{0}(0)| \leq 3 \text{Lip}(u_{0}) \epsilon.
	\end{equation*}
Therefore, we can argue as in Lemma \ref{L: calculus_exercise_compatibility} to see that $v_{0}^{\epsilon}$ satisfies the required estimates.  

Finally, $v_{0}^{\epsilon}(x) = u_{0}(0)$ if $x \in \bigcup_{i = 1}^{K} \overline{I_{i}^{\epsilon}} $ so $\sum_{i = 1}^{K} v^{\epsilon}_{0,x_{i}}(0) = 0$.  
\end{proof}



\section{Time-Dependent Finite-Difference Schemes}  \label{A: cauchy_fd_explanation}

In this section, we show that the finite-difference scheme approximating \eqref{E: time} is monotone provided a CFL-type condition is satisfied.  We also establish the required regularity properties of the solution.

We begin by introducing the necessary terminology.  A function $V : \mathcal{J} \times S \to \mathbb{R}$ is said to be a sub-solution of the scheme \eqref{E: cauchy_fd} if it satisfies the system of inequalities obtained by replacing all equal signs with $\leq$.  Analogously, a function $W$ on the same domain is called a super-solution of the scheme \eqref{E: cauchy_fd} if it satisfies the system of inequalities obtained by replacing all equal signs with $\geq$.  As in the stationary case, the scheme is monotone when sub- and super-solutions obey a discrete maximum principle.  This is made precise in the following definition.

\begin{definition} \label{D: monotone_time} The finite-difference scheme \eqref{E: cauchy_fd} is called \emph{monotone} if the following two criteria hold:

(i) If $V,\chi : \mathcal{J} \times \{0,1,\dots,N\} \to \mathbb{R}$, $V$ is a sub-solution of \eqref{E: cauchy_fd}, and $V - \chi$ has a global maximum at $(m,s)$ with $s > 0$ and $m \in J_{i}$, then
	\begin{equation*}
	\frac{\chi(m,s) - \chi(m,s - 1)}{\Delta t} + F_{i}(D^{+}\chi(m,s - 1),D^{-}\chi(m,s - 1)) \leq f_{i}(s\Delta t, -m \Delta x)
	\end{equation*}
if $m \neq 0$, and
	\begin{equation*}
	\sum_{i = 1}^{K} (\chi(0,s) - \chi(1_{i},s)) \leq 0, \quad \text{otherwise}.
	\end{equation*}

(ii) If $W,\chi : \mathcal{J} \times \{0,1,\dots,N\} \to \mathbb{R}$, $W$ is a super-solution of \eqref{E: cauchy_fd}, and $W - \chi$ has a global minimum at $(m,s)$ with $s > 0$ and $m \in J_{i}$, then
	\begin{equation*}
	\frac{\chi(m,s) - \chi(m,s - 1)}{\Delta t} + F_{i}(D^{+}\chi(m,s - 1),D^{-}\chi(m,s - 1)) \geq f_{i}(s\Delta t, -m \Delta x)
	\end{equation*} 
if $m \neq 0$, and
	\begin{equation*}
	\sum_{i = 1}^{K} (\chi(0,s) - \chi(1_{i},s)) \geq 0, \quad \text{otherwise}.
\end{equation*}
\end{definition}  

As in the time-independent setting, when we use the term ``monotone" in reference to \eqref{E: cauchy_fd}, we always mean it in the sense of the previous definition.

The error analysis of \eqref{E: cauchy_fd} uses a discrete version of Lipschitz continuity.  Specifically, given a function $U : \mathcal{J} \times S \to \mathbb{R}$, we say that $U$ is Lipschitz if 
\begin{equation*}
\text{Lip}(U) := \sup \left\{|U(m,s) - U(k,r)| \, \mid \, \frac{d(-m\Delta x, -k \Delta x)}{\Delta x} + |s - r| \leq 1 \right\} < \infty.
\end{equation*}

The following result gives sufficient conditions under which the scheme \eqref{E: cauchy_fd} is monotone and the solution is Lipschitz.  Recall that $L_{G}$ is a uniform bound on the Lipschitz constants of the numerical Hamiltonians $G_{1},\dots,G_{K}$, and $L_{c}$ is the cut-off in assumption \eqref{As: num_Ham_consistent}. 

\begin{prop} \label{P: monotone_time} There is an $\tilde{L}_{c} > 0$ depending only on $\text{Lip}(u_{0})$, $D$, $L_{G}$, $L_{2}$, and $T$ such that if \eqref{As: CFL_cauchy} holds and $L_{c} \geq \tilde{L}_{c}$, then the finite-difference scheme \eqref{E: cauchy_fd} is monotone and the solution $U$ of \eqref{E: cauchy_fd} satisfies $\text{Lip}(U) \leq \tilde{L}_{c} \Delta x$.    \end{prop}

\begin{proof}  From \eqref{As: CFL_cauchy}, we see that $\epsilon \geq 2L_{G} \Delta x$ and $\frac{\Delta x}{\Delta t} - \frac{ \epsilon}{\Delta x} - 2L_{G} \geq 0$.  From this, it follows that the expression	
	\begin{equation*}
		\frac{\chi(k,s) - \chi(k,s-1)}{\Delta t} + F_{i}(D^{+}\chi(k,s-1),D^{-}\chi(k,s-1))
	\end{equation*}  
is non-increasing in the variables $\chi(k,s-1)$, $\chi(k+1,s-1)$, and $\chi(k-1,s-1)$.  We leave it to the reader to verify that this implies \eqref{E: cauchy_fd} is monotone according to Definition \ref{D: monotone_time}.  

To see that $U$ is Lipschitz, we argue as in the continuum case.  To start with, define $V : \mathcal{J} \times (S \setminus \{N\}) \to \mathbb{R}$ by $V(k,s) = \frac{U(k,s + 1) - U(k,s)}{\Delta t}$.  Observe that if $s \in S \setminus \{N,N - 1\}$ and $k \in \mathcal{J} \setminus \{0\}$, then
	\begin{equation} \label{E: linearized_scheme}
		D_{t}V(k,s) + B_{i}^{+}(k,s) D^{+}V(k,s) + B_{i}^{-}(k,s) D^{-}V(k,s) - D_{t} \Gamma(k,s) = 0,
	\end{equation}
where the coefficients of the equation are defined as follows:
	\begin{align*}
	B_{i}^{+}(k,s) &= \frac{F_{i}(D^{+}U(k,s + 1),D^{-}U(k,s+1)) - F_{i}(D^{+}U(k,s),D^{-}U(k,s+1))}{D^{+}U(k,s + 1) - D^{+}U(k,s)}, \\
	B_{i}^{-}(k,s) &= \frac{F_{i}(D^{+}U(k,s),D^{-}U(k,s+1)) - F_{i}(D^{+}U(k,s),D^{-}U(k,s))}{D^{-}U(k,s+1) - D^{-}U(k,s)}, \\
	\Gamma(k,s) &= f_{i}(s \Delta t, -k\Delta x).
	\end{align*}     
The discussion in the previous paragraph implies $B_{i}^{+} \leq 0$ and $B_{i}^{-} \geq 0$ pointwise in $\mathcal{J} \times (S \setminus \{N\})$.  

In addition to \eqref{E: linearized_scheme}, $V$ satisfies $\sum_{i = 1}^{K} D^{+}V(1_{i},s) = 0$ if $s \in S \setminus \{1,N\}$.  Notice that if we define a scheme using \eqref{E: linearized_scheme} and this discrete Kirchoff condition, then the signs of $B_{i}^{+}$ and $B_{i}^{-}$ imply it is monotone in $\mathcal{J} \times (S \setminus \{N\})$ in the sense of Definition \ref{D: monotone_time}. 

By \eqref{As: time_bound}, $|D_{t}\Gamma| \leq D$ pointwise in $\mathcal{J} \times S \setminus \{N\}$.  Therefore, using monotonicity and arguing as in Proposition \ref{P: linearized_argument}, we find that if $(k,s) \in \mathcal{J} \times (S \setminus \{N\})$, then
	\begin{equation*}
		|V(k,s)| \leq \sup \left\{ |V(k,0)| \, \mid \, k \in \mathcal{J} \right\} + DT. 
	\end{equation*} 
In particular, since $V(k,0)$ is determined by $u_{0}$, there is a constant $C_{0}$ depending only on $\text{Lip}(u_{0})$ such that $|V| \leq C_{0} + DT$ pointwise.  Notice that, by \eqref{As: CFL_cauchy}, $C_{0}$ can be chosen independent of $\Delta x$ and $\epsilon$, though it does depend on $L_{2}$.

Now we show that the finite differences $D^{+}U$ and $D^{-}U$ are uniformly bounded.  Indeed, if we fix $s \in S \setminus \{N\}$, then the function $m \mapsto U(m,s)$ defined in $\mathcal{J}$ satisfies the stationary finite difference equation
	\begin{equation} \label{E: stopped_time_discrete}
		V(m,s) + F_{i}(D^{+}U(m,s),D^{-}U(m,s)) = f_{i}(s\Delta t,-m\Delta x) \quad \text{in} \, \, J_{i}.
	\end{equation} 
Since $V$ is uniformly bounded and the assumption $\epsilon \geq 2 L_{G} \Delta x$ implies the difference equation \eqref{E: stopped_time_discrete} is monotone, we can argue exactly as in Theorem \ref{T: stationary_scheme} to see that there is an $\tilde{L}_{c} > 0$ depending only on $C_{0}$ and $D$, but not on $s$, such that if $L_{c} \geq \tilde{L}_{c}$, then $\text{Lip}(U(\cdot,s)) \leq \tilde{L}_{c}\Delta x$.  

The bound we obtained through the equation only applies if $s < N$.  To get a bound at $s = N$, observe that the assumption $\frac{\Delta x}{\Delta t} \geq 2 L_{G}$ implies
	\begin{align*}
		|U(k + 1,N) - U(k,N)| &\leq |U(k + 1,N) - U(k+1,N-1)| \\
			&\quad + |U(k+1,N-1)- U(k,N-1)| \\
			&\quad + |U(k,N-1) - U(k,N)| \\
			&\leq 2(C_{0} + DT)\Delta t + \tilde{L}_{c}\Delta x \\
			&\leq \left(\frac{C_{0} + DT}{L_{G}} + \tilde{L}_{c} \right) \Delta x
	\end{align*}
Thus, making $\tilde{L}_{c}$ larger if necessary, we can assume that $\text{Lip}(U(\cdot,s)) \leq \tilde{L}_{c}$ independently of $s \in S$.  Making $\tilde{L}_{c}$ larger again, we can assume that $\tilde{L}_{c} \geq \frac{C_{0} + DT}{L_{G}}$ and, thus, $|V| \leq \tilde{L}_{c}L_{G}$ pointwise.  From this and the assumption that $\frac{\Delta t}{\Delta x} \leq L_{G}^{-1}$, we conclude that $\text{Lip}(U) \leq \tilde{L}_{c}\Delta x$ on $\mathcal{J} \times S$.  \end{proof}

%
 


\section{Proof of Theorem \ref{T: comparison}} \label{A: purely_technical}  In this section, we take on the hardest step in the comparison results presented above.  In order to apply \cite[Lemma 3.1]{time-dependent}, we need to understand, roughly speaking, the extent to which the equation ``sees" the differentiability (or lack thereof) of a sub-solution or super-solution at the junction.

%
%
%

In what follows, given $u : (-\infty,0] \to \mathbb{R}$ and $x \in (-\infty,0]$, we define $J^{+}u(x)$ to be the set of all $p \in \mathbb{R}$ such that
	\begin{equation*}
		u(y) \leq u(x) + p(y - x) + o(|y -x|) \quad \text{as} \, \, y \to x.
	\end{equation*}
$J^{-}u(x)$ is defined by $J^{-}u(x) = -J^{+}(-u)(x)$.  

Notice that this is analogous to the definitions in Appendix \ref{A: reformulation}.  In particular, given $x \in \overline{I_{i}}$ and $u : \mathcal{I} \to \mathbb{R}$, if $u_{i} : (-\infty,0] \to \mathbb{R}$ is defined by restricting $u$ to $\overline{I}_{i}$, then $J^{+}_{i}u(x) = J^{+}u_{i}(x)$ and $J^{-}_{i}u(x) = J^{-}u_{i}(x)$.  

\begin{lemma} \label{L: differentiable_structure} If $u : (-\infty,0] \to \mathbb{R}$ is upper semi-continuous continuous and $u_{x}(0)$ exists, then there are sequences $(x_{n}^{+})_{n \in \mathbb{N}} \subseteq (-\infty,0)$, $(p_{n}^{+})_{n \in \mathbb{N}} \subseteq \mathbb{R}$ such that 
\begin{itemize}
\item[(a)] $p_{n}^{+} \in J^{+}u(x_{n}^{+})$ for each $n \in \mathbb{N}$
\item[(b)] $\lim_{n \to \infty} p_{n}^{+} = u_{x}(0)$ 
\item[(c)] $\lim_{n \to \infty} x_{n}^{+} = 0$ 
\item[(d)] $\lim_{n \to \infty} u(x_{n}^{+}) = u(0)$
\end{itemize}

Similarly, if $v : (-\infty,0] \to \mathbb{R}$ is lower semi-continuous and $v_{x}(0)$ exists, then there are sequences $(x_{n}^{-})_{n \in \mathbb{N}} \subseteq (-\infty,0)$ and $(q_{n}^{-})_{n \in \mathbb{N}} \subseteq \mathbb{R}$ such that
\begin{itemize}
\item[(a)] $q_{n}^{-} \in J^{-}v(x_{n}^{-})$ for each $n \in \mathbb{N}$
\item[(b)] $\lim_{n \to \infty} q_{n}^{-} = v_{x}(0)$ 
\item[(c)] $\lim_{n \to \infty} x_{n}^{-} = 0$ 
\item[(d)] $\lim_{n \to \infty} v(x_{n}^{-}) = v(0)$
\end{itemize}
\end{lemma}  

\begin{proof}  Regarding $(x_{n}^{+},p_{n}^{+})$, this follows from the proof of Lemma \ref{L: slopes} and the fact that, in this case, $J^{+}u(0) = (-\infty,u_{x}(0)]$.  To obtain the sequences $(x_{n}^{-},p_{n}^{-})$, use the fact that $-v$ is upper semi-continuous and $J^{+}(-v)(x) = -J^{-}v(x)$.  \end{proof}

When the solution is not differentiable at the junction, Lemma \ref{L: differentiable_structure} is replaced by the following one:

\begin{lemma} \label{L: oscillatory_structure} Suppose $u : (-\infty,0] \to \mathbb{R}$ is continuous and $u_{x}(0)$ does not exist.  Let $p^{+} = \limsup_{x \to 0^{-}} \frac{u(x) - u(0)}{x}$ and $p^{-} = \liminf_{x \to 0^{-}} \frac{u(x) - u(0)}{x}$.  If $p \in (p^{-},p^{+})$, then there is a sequence $(x_{n}^{+})_{n \in \mathbb{N}} \subseteq (-\infty,0)$ such that 
\begin{itemize}
\item[(a)] $p \in J^{+}u(x_{n}^{+})$ for all $n \in \mathbb{N}$
\item[(b)] $\lim_{n \to \infty} x_{n}^{+} = 0$ 
\item[(c)] $\lim_{n \to \infty} u(x_{n}^{+}) = u(0)$
\end{itemize}

Similarly, suppose $v : (-\infty,0] \to \mathbb{R}$ is continuous and $v_{x}(0)$ does not exist.  Let $q^{+} = \limsup_{x \to 0^{-}} \frac{v(x) - v(0)}{x}$ and $q^{-} = \liminf_{x \to 0^{-}} \frac{v(x) - v(0)}{x}$.  If $q \in (q^{-},q^{+})$, then there is a sequence $(x_{n}^{-})_{n \in \mathbb{N}} \subseteq (-\infty,0)$ such that
\begin{itemize}
\item[(a)] $q \in J^{-}v(x_{n}^{-})$ for all $n \in \mathbb{N}$
\item[(b)] $\lim_{n \to \infty} x_{n}^{-} = 0$
\item[(c)] $\lim_{n \to \infty} v(x_{n}^{-}) = v(0)$
\end{itemize}
\end{lemma}

\begin{proof}  We only provide the arguments in the upper semi-continuous case since the lower semi-continuous case follows by a transformation as in the previous lemma.  

First, observe that since $p^{+} > p^{-}$, $u$ crosses the line $x \mapsto px$ infinitely often as $x \to 0^{-}$.  Therefore, there is a sequence $(y_{n})_{n \in \mathbb{N}} \subseteq (-\infty,0)$ such that 
\begin{itemize}
\item[(i)] $y_{n} < y_{n + 1}$ for all $n \in \mathbb{N}$,
\item[(ii)] $\lim_{n \to \infty} y_{n} = 0$, 
\item[(iii)] $\frac{u(y_{n}) - u(0)}{y_{n}} \leq p$ for all $n \in \mathbb{N}$, and
\item[(iv)] For all $n \in \mathbb{N}$, $y \mapsto u(y) - u(0) - py$ has a positive maximum in $[y_{n},y_{n + 1}]$.
\end{itemize}
For each $n \in \mathbb{N}$, let $x_{n}^{+}$ be a point in $[y_{n},y_{n + 1}]$ where $y \mapsto u(y) - u(0) - py$ is maximized.  Notice that (iii) and (iv) imply $x_{n}^{+} \in (y_{n},y_{n + 1})$.  Therefore, $p \in J^{+}u(x_{n}^{+})$.  Moreover, $\lim_{n \to \infty} x_{n}^{+} = 0$, and, thus, by assumption, $\lim_{n \to \infty} u(x_{n}^{+}) = u(0)$.    \end{proof}  

Finally, we have the ingredients necessary to establish Theorem \ref{T: comparison}.  For the sake of clarity, we begin by boiling Lemmas \ref{L: differentiable_structure} and \ref{L: oscillatory_structure} down into the form we will use in the proof.  

\begin{prop}  \label{P: hard_part}  Fix $i \in \{1,2,\dots,K\}$.  Suppose that $u : (-\infty,0] \to \mathbb{R}$ is a continuous sub-solution (resp.\ super-solution) of $u + H_{i}(x,u_{x}) = 0$ in $(-\infty,0)$.  Let $p^{+} = \limsup_{x \to 0^{-}} \frac{u(x) - u(0)}{x}$ and $p^{-} = \liminf_{x \to 0^{-}} \frac{u(x) - u(0)}{x}$.  If $p \in (p^{-},p^{+})$, then $u(0) + H_{i}(0,p) \leq 0$ (resp.\ $\geq 0$).

If $|p^{+}| < \infty$ (resp.\ $|p^{-}| < \infty$), then the conclusion holds with $p = p^{+}$ (resp.\ $p = p^{-}$) as well.  \end{prop}  

\begin{proof}  We only provide the arguments when $u$ is a sub-solution since the super-solution case follows in the same way.  

Fix $p \in (p^{-},p^{+})$.  Notice that Lemmas \ref{L: differentiable_structure} and \ref{L: oscillatory_structure} together imply that there is a sequence $(x_{n},p_{n})_{n \in \mathbb{N}} \subseteq (-\infty,0) \times \mathbb{R}$ such that $p_{n} \in J^{+}u(x_{n})$ independently of $n \in \mathbb{N}$ and $\lim_{n \to \infty} (x_{n},p_{n},u(x_{n})) = (0,p,u(0))$.  Since $x_{n} < 0$, we can invoke the sub-solution property to find
\begin{equation*}
u(x_{n}) + H_{i}(x_{n},p_{n}) \leq 0,
\end{equation*}
which, upon sending $n \to \infty$, becomes $u(0) + H_{i}(0,p) \leq 0$.

If $|p^{+}| < \infty$, then $u(0) + H_{i}(0,p^{+}) \leq 0$ follows from the continuity of $p \mapsto H_{i}(0,p)$.  The same can be said if $|p^{-}| < \infty$.       \end{proof}  

The proof of Theorem \ref{T: comparison} is now an application of Proposition \ref{P: hard_part} and Remark \ref{R: calculus}:

\begin{proof}[Proof of Theorem \ref{T: comparison}] We will only give the details for sub-solutions.  In addition to $(p_{1}^{+},\dots,p_{K}^{+})$, let us also define $(p_{1}^{-},\dots,p_{K}^{-})$ by 
	\begin{equation*}
		p_{i}^{-} = \liminf_{I_{i} \ni x \to 0} \frac{u(x) - u(0)}{x}.
	\end{equation*}  

Proposition \ref{P: hard_part} implies (i) directly.  Additionally, it shows that if $\tilde{p}_{i} \geq p_{i}^{-}$ for some $i \in \{1,2,\dots,K\}$, then \eqref{E: crucial_part} in (ii) holds.

It only remains to establish (ii) in the case when $\tilde{p}_{i} < p_{i}^{-}$ for all $i \in \{1,2,\dots,K\}$.  Remark \ref{R: calculus} implies that, in this case, if $\varphi : \mathcal{I} \to \mathbb{R}$ is given by 
\begin{equation*}
\varphi(x) = u(0) + \tilde{p}_{i}x \quad \text{if} \, \, x \in \overline{I_{i}},
\end{equation*} 
then $u - \varphi$ has a local maximum at $0$.  Therefore, since $u$ is a sub-solution, we find
\begin{equation*}
\min \left\{ \sum_{i = 1}^{K} \tilde{p}_{i}, u(0) + \min_{i} H_{i}(0,\tilde{p}_{i}) \right\} \leq 0.
\end{equation*}
Thus, \eqref{E: crucial_part} holds, as claimed.           
\end{proof}  

%
%
%
%
%


\section*{Acknowledgments}

It is a pleasure to acknowledge P.E.\ Souganidis for suggesting this problem and for enlightening discussions.  Credit is due as well to the anonymous reviewers for their sage advice and for pointing out a number of typos, and to M.\ Sardarli for helpful comments.

The author was partially supported by the National Science Foundation Research Training Group grant DMS-1246999.  


\end{document}